\documentclass[lettersize,journal]{IEEEtran}
\usepackage{amsthm,amsmath,amsfonts,amssymb,commath,upgreek,enumitem,bbm}
\allowdisplaybreaks %%%% to allow page breaks in align envir
\usepackage{algorithmic}
\usepackage{algorithm}
\usepackage{array}
\usepackage[caption=false,font=normalsize,labelfont=sf,textfont=sf]{subfig}
\usepackage{textcomp}
\usepackage{stfloats}
\usepackage{url}
\usepackage{verbatim}
\usepackage{graphicx}
\usepackage{cite}
\theoremstyle{plain}
\newtheorem{remark}{Remark}
\newtheorem{theorem}{Theorem}
\newtheorem{lemma}{Lemma}
\newtheorem{teo}{Theorem}
\newtheorem{prop}[teo]{Proposition}

\newcommand\E{\mathbb{E}}

\newcommand{\etahatS}{\hat{\eta}_{S}}
\newcommand{\etahatH}{\hat{\eta}_{H}}

\newcommand{\etahatL}{\hat{\eta}_{L}}
\newcommand{\etahatE}{\hat{\eta}_{E}}
\DeclareMathOperator*{\argmax}{arg\,max}
\DeclareMathOperator*{\argmin}{arg\,min}
\newcommand\supp{\operatorname{supp}}
\newcommand{\lambdas}{\lambda_{*}}
\newcommand\simiid{\overset{i.i.d}{\sim}}
\newcommand{\Var}{{\rm Var}}
\newcommand{\calP}{\mathcal{P}}
\newcommand{\calM}{\mathcal{M}}
\newcommand{\R}{\mathbb{R}}
\newcommand{\calN}{\mathcal{N}}
\makeatletter
\newcommand*{\rom}[1]{\expandafter\@slowromancap\romannumeral #1@}
\makeatother
\newcounter{relctr} %% <- counter for relations
\everydisplay\expandafter{\the\everydisplay\setcounter{relctr}{0}} %% <- reset every eq
\newcommand\labelrel[2]{%
  \begingroup
    \refstepcounter{relctr}%
    \stackrel{\textnormal{(\alph{relctr})}}{\mathstrut{#1}}%
    \originallabel{#2}%
  \endgroup
}
\AtBeginDocument{\let\originallabel\label} %% <- store original definition
\newcommand{\secThmIV}{\ref{proof:thm4:rg1}}
\newcommand{\secThmV}{\ref{thm5:proof:sec}}
\newcommand{\secPropII}{\ref{app:proof_prop2}}
\newcommand{\secThmVI}{\ref{prf:thm:nine}}
\newcommand{\secPropIV}{\ref{sec::lasso-risk-regime_3}}
\newcommand{\secScaling}{Section \ref{sec:scaling}}
\newcommand{\thmMinimaxLower}{\ref{thm::minimax-lower-bound}}
\newcommand{\secRegimeILower}{Section \ref{sec::lower-bound-regime-1}}
\newcommand{\EqBlockPriorLower}{Equation \eqref{reduce:1d}}
\newcommand{\lemSpikeBayes}{\ref{lem::spike-bayes-risk-regime-1}}
\newcommand{\lemSpikeBayesDecompose}{ \ref{lem::spike-bayes-risk-decompose}}
\newcommand{\thmMinimax}{\ref{thm:minimax}}
\newcommand{\thmIndepLessFavor}{\ref{thm:indep_less_favor}}
\newcommand{\secPropI}{Section \ref{app:proof_lasso_regime_1}}
\newcommand{\lemPropItuning}{\ref{pre:order:tuning}}
\newcommand{\eqLemPropItuning}{Equation \eqref{eq:label:optlambda-regime_1}}
\newcommand{\eqSoftThresholdZeroDerivative}{Equation \eqref{eq::soft-thresholding-zero-derivative}}
\newcommand{\eqLassoRegimeILoII}{Equation \eqref{lasso:regime_1-lo2}}
\newcommand{\eqLimitCaseILastNew}{Equations \eqref{eq:limit:case1:last2:new}-\eqref{eq:limit:case1:last1:new}}
\newcommand{\eqAKeyMaster}{Equation \eqref{a:key:master}}
\newcommand{\lemRefineOrder}{Lemma \ref{refine:order}}
\newcommand{\eqRiskAtZertoSoftSecondMom}{Equations \eqref{risk:at:zero}-\eqref{eq::soft-thresholding-second-moment}}
\newcommand{\eqRiskOfLambdatoSoftSecondMom}{Equations \eqref{eq::soft-thresholding-risk-of-lambda}-\eqref{eq::soft-thresholding-second-moment}}

\newcommand{\lemGaussianMill}{Lemma \ref{lem::gaussian-tail-mills-ratio}}
\begin{document}
\title{Signal-to-noise ratio aware minimaxity and higher-order asymptotics}

\author{Yilin Guo, Haolei Weng, Arian Maleki
        % <-this % stops a space
\thanks{This work is supported by NSF-DMS 2210506, and NSF-DMS 2210505.}% <-this % stops a space
\thanks{Y. Guo is with the Department of Statistics, Columbia University, New York, USA. (e-mail: yilinguo97@gmail.com). H. Weng is with the Department of Statistics and Probability, Michigan State University, East Lansing, Michigan, USA. (e-mail: wenghaol@msu.edu). A. Maleki is with the Department of Statistics, Columbia University, New York, USA. (e-mail: arian.maleki@gmail.com).}
}



\maketitle

\begin{abstract}
Since its development, the minimax framework has been one of the corner stones of theoretical statistics, and has contributed to the popularity of many well-known estimators, such as the regularized M-estimators for high-dimensional problems. In this paper, we will first show through the example of sparse Gaussian sequence model, that the theoretical results under the classical minimax framework are insufficient for explaining empirical observations. In particular, both hard and soft thresholding estimators are (asymptotically) minimax, however, in practice they often exhibit sub-optimal performances at various signal-to-noise ratio (SNR) levels. The first contribution of this paper is to demonstrate that this issue can be resolved if the signal-to-noise ratio is taken into account in the construction of the parameter space. We call the resulting minimax framework the signal-to-noise ratio aware minimaxity. The second contribution of this paper is to showcase how one can use higher-order asymptotics to obtain accurate approximations of the SNR-aware minimax risk and discover minimax estimators. The theoretical findings obtained from this refined minimax framework provide new insights and practical guidance for the estimation of sparse signals.
\end{abstract}


% \section{Introduction}
% \IEEEPARstart{T}{his} file is intended to serve as a ``sample article file''
% for IEEE journal papers produced under \LaTeX\ using
% IEEEtran.cls version 1.8b and later. The most common elements are covered in the simplified and updated instructions in ``New\_IEEEtran\_how-to.pdf''. For less common elements you can refer back to the original ``IEEEtran\_HOWTO.pdf''. It is assumed that the reader has a basic working knowledge of \LaTeX. Those who are new to \LaTeX \ are encouraged to read Tobias Oetiker's ``The Not So Short Introduction to \LaTeX ,'' available at: \url{http://tug.ctan.org/info/lshort/english/lshort.pdf} which provides an overview of working with \LaTeX.

\section{Introduction}

\subsection{Motivation}

\IEEEPARstart{T}{he} minimax framework is one of the most popular approaches for comparing the performance of estimators and obtaining the optimal ones. Since its development, the minimax framework has been used for the study of optimality and the design of optimal estimators in a broad range of areas including, among others, classical statistical decision theory \cite{le1986asymptotic, lehmann1998theory}, non-parametric statistics \cite{johnstone19, Tsybakov:2008:INE:1522486}, high-dimensional statistics \cite{wainwright2019high}, and mathematical data science \cite{fan2020statistical}. Despite its popularity, when the parameter space is set too general, since the minimax framework focuses on particular areas of the parameter space, its conclusions can be misleading if translated and used in practice. Take the high-dimensional sparse linear regression for example. It has been proved that the best subset selection is minimax rate-optimal over the class of $k$-sparse parameters \cite{raskutti2011minimax}. Nevertheless, recent empirical and theoretical works demonstrate the inferior performance of best subset selection in low signal-to-noise ratio (SNR) \cite{hastie2020best, mazumder2022subset, zheng2017does}. The key issue in this problem is that the parameter space in the minimax analysis only incorporates sparsity structure and does not control the signal strength for non-zero components of the sparse vector.

In this paper, we focus on the popular example of the sparse Gaussian sequence model -- a special case of the sparse linear regression model with an orthogonal design. We first discuss in detail the limitations of classical minimaxity in Section \ref{sec::misleading-sparse-model}. The rest of the paper is then devoted to the development of a much more informative minimax framework that alleviates major drawbacks of the classical one. This is made possible by controlling and monitoring the signal-to-noise ratio and sparsity level through the parameter space. As will be discussed later, solving this new constrained minimax problem is much more challenging than the original minimax analysis. Hence, we resort to higher-order asymptotic analysis to obtain approximate minimax results. The conclusions of this signal-to-noise ratio aware minimax framework turn out to provide new insights into the estimation of sparse signals.  

\subsection{Classical minimaxity and its limitations in sparse Gaussian sequence model} \label{sec::misleading-sparse-model}

We consider the Gaussian sequence model:
\begin{eqnarray}\label{model::gaussian_sequence}
    y_i=\theta_i+\sigma_n z_i, \quad i=1,2,\ldots, n.
\end{eqnarray}
Here, $y=(y_1,\ldots, y_n)$ is the vector of observations, $\theta=(\theta_1,\ldots,\theta_n)$ is the unknown signal consisting of $n$ unknown parameters, $z_i$'s are i.i.d. standard Gaussian error variables, and $\sigma_n>0$ is the noise level that may vary with sample size $n$.  
The goal is to estimate $\theta$ from the sparse parameter space 
\begin{eqnarray}\label{model::sparse-parameter-space}
    \Theta(k_n)=\Big\{\theta \in \mathbb{R}^n: ~\| \theta\|_0\leq k_n\Big\},
\end{eqnarray}
where $\norm{\theta}_{0}$ denotes the number of non-zero components of $\theta$, and the sparsity $k_n$ is allowed to change with $n$. The most popular approach for studying this estimation problem and obtaining the optimal estimators is the {\em minimax} framework. Considering the squared loss, the minimax framework aims to find the estimator that achieves the minimax risk given by   
\begin{eqnarray}\label{eq::sparse-minimax}
    R(\Theta(k_n),\sigma_n)=\inf_{\hat{\theta}}\sup_{\theta\in \Theta(k_n)}\mathbb{E}_{\theta}\|\hat{\theta}-\theta\|_2^2,
\end{eqnarray}
where $\mathbb{E}_{\theta}(\cdot)$ is the expectation taken under \eqref{model::gaussian_sequence} with true parameter value $\theta$. 

% \subsection{Classical minimaxity and its limitations in sparse Gaussian sequence model} \label{sec::misleading-sparse-model}

% We consider the Gaussian sequence model:
% \begin{eqnarray}\label{model::gaussian_sequence}
%     y_i=\theta_i+\sigma_n z_i, \quad i=1,2,\ldots, n.
% \end{eqnarray}
% Here, $y=(y_1,\ldots, y_n)$ is the vector of observations, $\theta=(\theta_1,\ldots,\theta_n)$ is the unknown signal consisting of $n$ unknown parameters, $z_i$'s are i.i.d. standard Gaussian error variables, and $\sigma_n>0$ is the noise level that may vary with sample size $n$.  
% The goal is to estimate $\theta$ from the sparse parameter space 
% \begin{eqnarray}\label{model::sparse-parameter-space}
%     \Theta(k_n)=\Big\{\theta \in \mathbb{R}^n: ~\| \theta\|_0\leq k_n\Big\},
% \end{eqnarray}
% where $\norm{\theta}_{0}$ denotes the number of non-zero components of $\theta$, and the sparsity $k_n$ is allowed to change with $n$. The most popular approach for studying this estimation problem and obtaining the optimal estimators is the {\em minimax} framework. Considering the squared loss, the minimax framework aims to find the estimator that achieves the minimax risk given by   
% \begin{eqnarray}\label{eq::sparse-minimax}
%     R(\Theta(k_n),\sigma_n)=\inf_{\hat{\theta}}\sup_{\theta\in \Theta(k_n)}\mathbb{E}_{\theta}\|\hat{\theta}-\theta\|_2^2,
% \end{eqnarray}
% where $\mathbb{E}_{\theta}(\cdot)$ is the expectation taken under \eqref{model::gaussian_sequence} with true parameter value $\theta$. 

Gaussian sequence model plays a fundamental role in non-parametric and high-dimensional statistics. There exists extensive literature on the minimax estimation of $\theta$ or its functionals over various structured parameter spaces such as Sobolev ellipsoids, hyperrectangles and Besov bodies. These parameter spaces usually characterize the smoothness properties of functions in terms of their Fourier or wavelet coefficients. We refer to \cite{gine2021mathematical, johnstone19, Tsybakov:2008:INE:1522486} and references therein for a systematic treatment of this topic. The estimation problem over $\Theta(k_n)$ has been also well studied in statistical decision theory (e.g., with application to wavelet signal processing) since 1990s. Define the soft thresholding estimator $\hat{\eta}_S(y,\lambda)\in \mathbb{R}^n$ and hard thresholding estimator $\hat{\eta}_H(y,\lambda)\in \mathbb{R}^n$ with coordinates: for $1\leq i \leq n$,
\begin{align}
     [ \hat{\eta}_{S}(y,\lambda) ]_{i} &= \argmin _{\mu\in \mathbb{R}} ~(y_i-\mu)^{2} + 2\lambda|\mu| \nonumber \\ 
     &= {\rm sign}(y_i)(|y_i|-\lambda)_{+}, \label{eq::softthresh} \\
     [ \hat{\eta}_{H}(y,\lambda) ]_{i} &= \argmin _{\mu\in \mathbb{R}}~(y_i-\mu)^{2} + \lambda^{2} I(\mu\neq 0) \nonumber \\
     &= y_i I(|y_i|>\lambda), \label{eq::hardthresh}
\end{align}
where sign$(u), u_+$ represent the sign and positive part of $u$ respectively, $I(\cdot)$ denotes the indicator function, and $\lambda \geq 0$ is a tuning parameter. We summarize a classical asymptotic minimax result in the following theorem. 

\begin{theorem}[\cite{donoho1994minimax, donoho1992maximum, johnstone19}]\label{thm::Thm8.21-Johnstone}
Assume model \eqref{model::gaussian_sequence} and parameter space \eqref{model::sparse-parameter-space} with $ k_n/n \rightarrow 0$ as $n\rightarrow \infty$. Then the minimax risk, defined in \eqref{eq::sparse-minimax}, satisfies
\begin{equation*}
    R(\Theta(k_n),\sigma_n) = (2+o(1))\cdot \sigma_{n}^{2} k_n \log(n/k_n).
\end{equation*}
Moreover, both the soft and hard thresholding estimators with tuning $\lambda_n=\sigma_{n}\sqrt{2\log (n/k)}$ are asymptotically minimax, i.e., for $\hat{\theta}=\hat{\eta}_S(y,\lambda_n)$ or $\hat{\eta}_H(y,\lambda_n)$, it holds that
\begin{align*}
 \sup_{\theta \in \Theta(k_n)} \mathbb{E}_{\theta} \|\hat{\theta} - \theta\|_2^2~=(2+o(1))\cdot \sigma_{n}^{2} k_n \log(n/k_n).
\end{align*}
\end{theorem}

\begin{figure}[!t]
\centering
\setlength\tabcolsep{-8.pt}
\begin{tabular}{c}
\includegraphics[scale=0.37]{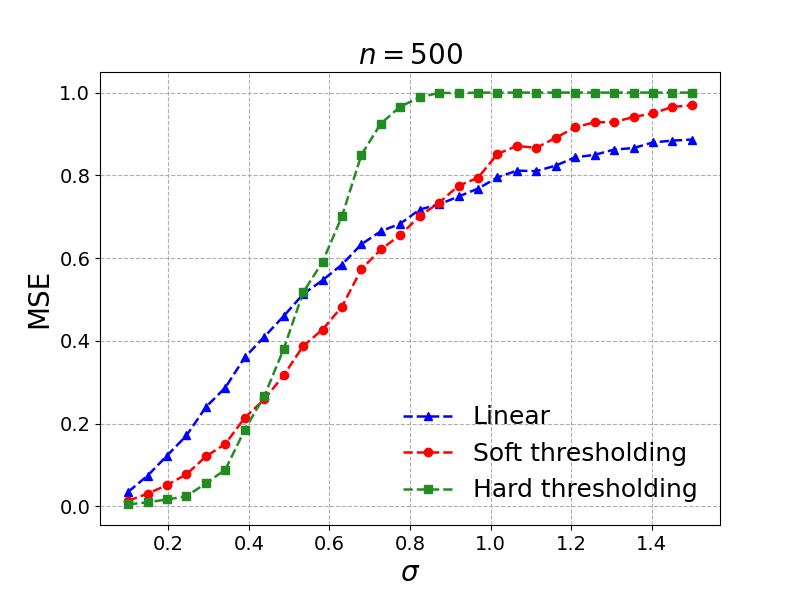} \\
\includegraphics[scale=0.37]{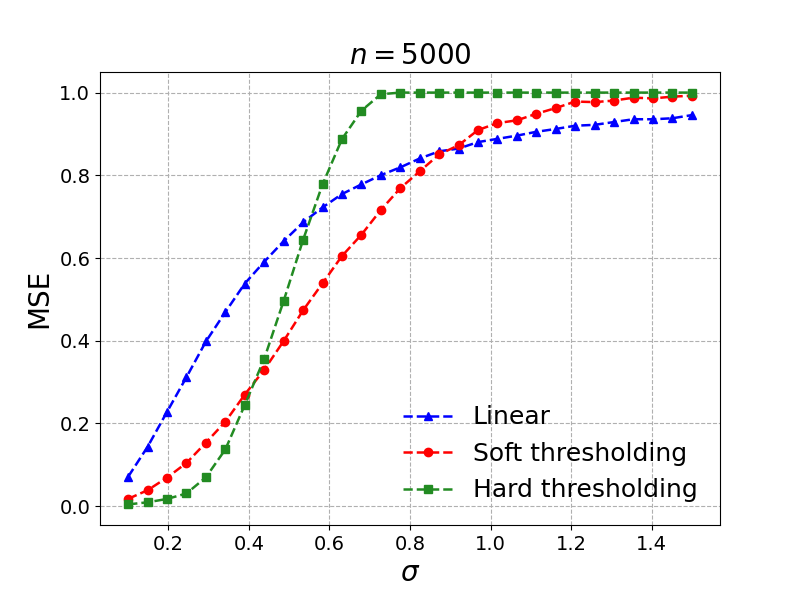} 
\end{tabular}
\vspace{-0.45cm}
\caption{Mean squared error comparison at different noise levels. Data is generated according to \eqref{model::gaussian_sequence} with $k_n= \lfloor n^{2/3}\rfloor$ and $\theta$ having $k_n$ components equal to $1.5$. ``linear" denotes the simple linear estimator $\frac{1}{1+\lambda}y$. All the three estimators are optimally tuned. MSE is averaged over 20 repetitions along with standard error. Other details of the simulation can be found in Section \ref{simulations}.} \label{fig:one}
\end{figure}

Theorem \ref{thm::Thm8.21-Johnstone} shows that both soft and hard thresholding estimators are minimax optimal for estimating sparse signals (with small values of $k_n/n$). Despite the mathematical beauty of the above results, its practical implications seem not clear. We demonstrate this point by a simulation in Figure \ref{fig:one}. As is clear from the upper panel, when the noise level is low, hard thresholding performs the best among the three estimators; as the noise level increases, hard thresholding starts to be outperformed by soft thresholding, and eventually both hard and soft thresohlding are outperformed by the linear estimator. The same comparison holds in the lower panel as  the sample size increases from $500$ to $5000$. This phenomenon can be widely observed for different types of sparse signals. We provide more simulations in Section \ref{simulations}. 

 In light of Theorem \ref{thm::Thm8.21-Johnstone} and Figure \ref{fig:one}, we would like to raise a few critical comments:

 \begin{enumerate}
     \item Despite their minimax optimality, both hard and soft thresholding estimators selected by the classical minimaxity do not perform well compared to a simple linear estimator when the noise is large. 
     \item The hard and soft thresholding estimators have distinct performances at different noise levels, despite they are both asymptotically minimax.
     \item Figure \ref{fig:one} implies that the signal-to-noise ratio (SNR) has a significant impact on the estimation. However, the effect of SNR is not well captured in the classical minimax results (Theorem \ref{thm::Thm8.21-Johnstone}). 
      \end{enumerate}
These observations lead us to the following question: is it possible to develop a refined minimax framework which addresses differences between hard and soft thresholding estimators and characterizes the role of SNR in the recovery of sparse signals? Such a framework will provide more proper insights and sound guidance for practical purpose.

\subsection{Our contributions and paper structure}

To overcome the limitations of the classical minimaxity discussed in Section \ref{sec::misleading-sparse-model}, in this paper, we aim to develop a signal-to-noise-ratio-aware minimax framework. This framework imposes direct constraints on the signal strength over the parameter space and performs the corresponding minimax analysis that accounts for the impact of signal-to-noise ratio (SNR). To obtain accurate minimax results in the SNR-aware setting, we will derive higher-order asymptotics which provides asymptotic approximations precise up to the second order. As will be discussed in detail in Section \ref{SNR-aware:minmaxity}, our proposed framework reveals three regimes in which distinct estimators achieve minimax optimality. In particular, hard-thresholding estimator outperforms soft-thresholding estimator and remains (asymptotically) minimax optimal in the high SNR regime; as SNR decreases, new optimal estimators will emerge. These new theoretical findings offer much better explanations for what is happening in Figure \ref{fig:one}, and are much more informative towards understanding the sparse estimation problem in practice.

The rest of the paper is organized as follows. Section \ref{SNR-aware:minmaxity} presents the main results from the SNR-aware minimax framework. Section \ref{simulations} includes more simulations to support our theoretical findings. Section \ref{discuss} summarizes the main messages of the paper and discusses some related works. All the proofs are presented in Section \ref{sec::proofs}.

We collect the notations used throughout the paper here for convenience. For a scalar $x\in \mathbb{R}$, $x_{+}$ and $\text{sign}(x)$ denote the positive part of $x$ and its sign respectively; $\lfloor x \rfloor$ is the largest integer less than or equal to $x$. For an integer $n$, $[n]=\{1,2,\ldots, n\}$. We use $I_{A}$ and $I(A)$ to represent the indicator function of the set $A$ interchangeably. For a given vector $v = (v_{1}, \ldots, v_{p})\in \mathbb{R}^{p}$, $\norm{v}_{0} = \# \{i: v_{i}\neq 0\}$, $\norm{v}_{\infty} = \max_{i}|v_{i}|$, and $\norm{v}_{q} = \left(\sum_{i=1}^{p} |v_{i}|^{q}\right)^{1/q}$ for $q \in (0,\infty)$. We use the notation $\delta_{\mu}$ as the point mass at $\mu \in \mathbb{R}$. We also use $\{e_j\}_{j=1}^p$ to denote the natural basis in $\mathbb{R}^p$. For two non-zero real sequences $\{a_n\}_{n=1}^{\infty}$ and $\{b_n\}_{n=1}^{\infty}$, we use $a_n = o(b_n)$ to represent $|a_n/b_n| \rightarrow 0$ as $n \rightarrow \infty$, and $a_n=\omega(b_n)$ if and only if $b_n=o(a_n)$; $a_n = O(b_n)$ means $\sup_n|a_n/b_n| < \infty$, and $a_n=\Omega(b_n)$ if and only if $b_n=O(a_n)$; $a_n=\Uptheta(b_n)$ denotes $a_n=O(b_n)$ and $a_n=\Omega(b_n)$. For a distribution $\pi$, $\supp(\pi)$ denotes its support. Finally, we reserve the notations $\phi(y)$ and $\Phi(y) = \int_{-\infty}^{y}\phi(s)d s $ for the standard normal density and its cumulative distribution function respectively. 

\section{SNR-aware minimaxity}
\label{SNR-aware:minmaxity}

\subsection{SNR-aware minimax framework}

We focus on the above-mentioned Gaussian sequence model \eqref{model::gaussian_sequence}. To develop the SNR-aware minimax framework, we start by inserting a notion of signal-to-noise ratio in the minimax setting. To this end, we consider the following SNR-aware parameter space:
\begin{equation}\label{eq::parameter-space-SNR}
    \Theta(k_n,\tau_{n}) = \Big \{\theta\in\mathbb{R}^{n}: ~\norm{\theta}_{0}\leq k_n, ~\norm{\theta}_{2}^{2}\leq k_n\tau_{n}^{2}\Big \}.
\end{equation}
Here, as before, $k_n$ is the parameter that controls the number of nonzero components of the signal $\theta\in \mathbb{R}^{n}$. The new parameter $\tau_n$ can be considered as a measure of signal strength (on average) for each non-zero coordinate of $\theta$. Unlike $\Theta(k_n)$, the new parameter space $\Theta(k_n,\tau_{n})$ is responsive to changing signal strength. Minimax analysis based on it may thus provide a viable path for revealing the impact of SNR on the estimation of sparse signals. Define the corresponding minimax risk (for squared loss):
\begin{equation}\label{eq::minimax-risk}
    R(\Theta(k_n,\tau_{n}),\sigma_{n}) = \inf_{\hat{\theta}} \sup_{\theta \in \Theta(k_n,\tau_{n})} \mathbb{E}_{\theta}\|\hat{\theta}-\theta\|_{2}^{2}.
\end{equation}
We aim to investigate the following problems:

\begin{enumerate}
    \item Characterizing the minimax risk, $R(\Theta(k_n,\tau_{n}),\sigma_{n})$, for different choices of sparsity level and signal-to-noise ratio. This will help us understand the intertwined roles of SNR and sparsity on signal recovery.  

\item Obtaining minimax optimal estimators in the aforementioned settings, along with evaluating the performance of some common estimators (e.g., soft thresholding).
\end{enumerate}

The solutions to the above problems will help resolve the issues we raised before about the classical minimax results.  First, we introduce two critical quantities associated with the target parameter space $\Theta(k_n,\tau_{n})$ introduced in \eqref{eq::parameter-space-SNR} under the model \eqref{model::gaussian_sequence}. Denote 
\begin{equation}
\label{key:quantity}
\epsilon_n=\frac{k_n}{n}, \quad \mu_n=\frac{\tau_n}{\sigma_n}.
\end{equation}
It is clear that $\epsilon_n$ represents the sparsity level and $\mu_n$ is a form of signal-to-noise ratio over the parameter space. We aim to study $R(\Theta(k_n,\tau_{n}),\sigma_{n})$ for different values of $(\epsilon_n,\mu_n)$. Since an explicit solution to exact minimaxity is very challenging to derive (it is not even available for $\Theta(k_n)$), we focus on obtaining asymptotic minimaxity, and consider the following regimes: as $n\rightarrow \infty$,

\begin{itemize}[label={}]
\item \textbf{Regime (\rom{1})}  Low signal-to-noise ratio: $\mu_{n}\rightarrow 0$, $\epsilon_{n}\rightarrow 0$; 
\item \textbf{Regime (\rom{2})} Moderate signal-to-noise ratio: $\mu_{n}\rightarrow \infty$,  $\epsilon_{n}\rightarrow 0$, $\mu_{n}=o(\sqrt{\log \epsilon_{n}^{-1}})$;
\item \textbf{Regime (\rom{3})} High signal-to-noise ratio: $\epsilon_{n}\rightarrow 0$, $\mu_{n}=\omega(\sqrt{\log \epsilon_{n}^{-1}})$. 
\end{itemize}

The condition $\epsilon_n\rightarrow 0$ is standard to model sparse signals. The above three regimes are classified according to the order of signal-to-noise ratio $\mu_n$. As will be shown in Section \ref{ssec:higher-order} via higher-order asymptotics, each regime exhibits unique minimaxity, and distinct minimax estimators emerge in different regimes. But before that, we first derive similar first-order asymptotic result as the classical one and reveal its limitations in the SNR-aware minimax setting.

\subsection{First order analysis of SNR-aware minimaxity and its drawbacks} \label{first:order:d}

Our first theorem generalizes Theorem \ref{thm::Thm8.21-Johnstone}, to our SNR-aware minimax framework.

\begin{theorem}\label{thm:firstorder:firstregime}
Assume model \eqref{model::gaussian_sequence} and parameter space \eqref{eq::parameter-space-SNR}. The following hold:
\begin{itemize}
\item {\rm Regime (\rom{1})}. When $\mu_{n}\rightarrow 0, \epsilon_{n}\rightarrow 0$, 
\[
R(\Theta(k_n,\tau_{n}),\sigma_{n})=(1+o(1))\cdot n \sigma_{n}^{2} \epsilon_{n} \mu_{n}^{2},
\]
and the zero estimator is asymptotically minimax optimal (up to the first order).
\item {\rm Regime (\rom{2})}. When $\mu_{n}\rightarrow \infty$,  $\epsilon_{n}\rightarrow 0$, $\mu_{n}=o(\sqrt{\log \epsilon_{n}^{-1}})$, 
\[
R(\Theta(k_n,\tau_{n}),\sigma_{n})=(1+o(1))\cdot n \sigma_{n}^{2} \epsilon_{n} \mu_{n}^{2},
\]
and the zero estimator is asymptotically minimax optimal (up to the first order).
\item {\rm Regime (\rom{3})}. When $\epsilon_{n}\rightarrow 0$, $\mu_{n}=\omega(\sqrt{\log \epsilon_{n}^{-1}})$,
\[
R(\Theta(k_n,\tau_{n}), \sigma_{n}) = (2+o(1))\cdot n \sigma_{n}^{2} \epsilon_{n} \log(\epsilon_{n}^{-1}).
\]
Furthermore, both soft and hard thresholding estimators \eqref{eq::softthresh}-\eqref{eq::hardthresh} with the tuning parameter $\lambda_{n} = \sigma_{n}\sqrt{2 \log \epsilon_{n}^{-1}}$ are asymptotically minimax optimal (up to the first order).
\end{itemize}
\end{theorem}

This theorem is covered as a special case of Theorems \ref{thm::regime_1}, \ref{thm::regime_2}, and \ref{thm::regime_4}  we present in Section \ref{ssec:higher-order}. Hence, the proof is skipped.

There are a few aspects of the above results that we would like to emphasize here:

\begin{enumerate}
    \item As is clear, first-order analysis under the new SNR-aware minimax framework already provides more information than in the previous framework. For instance, it implies that below a certain signal-to-noise-ratio, i.e. when  $\mu_n = o (\sqrt{\log \epsilon_n^{-1}})$, sparsity promoting estimators such as hard or soft thresholding do not seem to have any advantage over the zero estimator. In fact, the zero estimator is optimal up to the first order. Later in Section \ref{ssec:higher-order} we will argue that even these theorems should be interpreted carefully, and that the current interpretation is not fully accurate.

    \item If we consider the rate of $\epsilon_n$ fixed and evaluate the minimax risk as a function of $\mu_n$, we will see a phase transition happening in the first order term of the minimax risk. As long as the first order is concerned, the trivial zero estimator is minimax optimal for any $\mu_n = o (\sqrt{\log \epsilon_n^{-1}})$. Hence, it seems that unless $\mu_n = \Omega(\sqrt{\log \epsilon_n^{-1}})$, even the optimal minimax estimators will miss the signal. Once $\mu_n = \omega (\sqrt{\log \epsilon_n^{-1}})$, the first order result implies the optimality of non-trivial estimators, such as soft-thresholding. While it is challenging to provide an intuitive argument for the phase transition occurring at $\sqrt{\log\epsilon_n^{-1}}=\sqrt{\log (n/k_n)}$, the following explanation may offer some insight: Consider a $k_n$-sparse signal (with $k_n$ non-zero components) in $\mathbb{R}^n$ with Gaussian noises. \emph{On average}, there exists one non-zero signal component among $n/k_n$ locations. The maximum absolute value of the noises at the $n/k_n$ locations is on the order of $\sqrt{\log(n/k_n)}$. Consequently, from an intuitive perspective, it becomes easier to detect signals when their magnitudes exceed this threshold, but significantly more challenging when they fall below this threshold. It's important to note that heuristic arguments like the one above have their limitations and should not be solely relied upon for drawing conclusive results. This aspect will be further clarified in the next section, where we will demonstrate that minimax estimators can outperform zero estimators even when $\mu_n = o (\sqrt{\log \epsilon_n^{-1}})$.

\end{enumerate}

One of the main issues in the above theorem is that the first-order asymptotic approximation of minimax risk does not seem to always offer accurate information. %finite values of $n$ (and finite values of $\epsilon_n$ and $\mu_n$).
For example, as the signal-to-noise ratio significantly increases from Regime (\rom{1}) to Regime (\rom{2}), the first-order analysis falls short of capturing any difference and continues to generate the naive zero estimator as the optimal one. Moreover, in Regime (III), the analysis is inadequate to explain the difference between hard and soft thresholding estimators. In the next section, we push the analysis one step further to develop second-order asymptotics. This refined version of the SNR-aware minimax analysis will provide a much more accurate approximation of the minimax risk, and can provide more useful information and resolve the confusing aspects of the first-order results presented above.

\subsection{Second order analysis of SNR-aware minimaxity}\label{ssec:higher-order}

In this section, we discuss how the analysis provided in Section \ref{first:order:d} can be refined to resolve the issues we raised in Section \ref{sec::misleading-sparse-model}.

\subsubsection{Results in Regime (\rom{1})}

We start with Regime (\rom{1}). As discussed in Theorem \ref{thm:firstorder:firstregime}, as far as the first order of minimax risk is concerned, the zero estimator is asymptotically optimal in this regime, and no other estimators can outperform the zero estimator. The reason this peculiar feature arises is that since the exact expression for $R(\Theta(k_n,\tau_{n}),\sigma_{n})$ is very complicated, Theorem \ref{thm:firstorder:firstregime} resorts to an approximation that is asymptotically accurate. However, this approximation is coarse when $n$ is not too large and/or $\epsilon_{n}$ is not too small. The conclusions that are based on such first order analysis are hence not reliable. Therefore, we pursue a second-order asymptotic analysis of minimax risk to achieve better approximations. This more delicate analysis turns out to be   instructive for understanding the three regimes of varying SNRs. We first present the result in Regime (\rom{1}). Define the simple linear estimator $\hat{\eta}_L(y,\lambda)\in \mathbb{R}^n$ with coordinates:
\begin{equation}\label{def::ridge-estimator}
    [\hat{\eta}_L(y,\lambda)]_i = \frac{y_i}{1+\lambda} = \argmin_{\mu\in \mathbb{R}}~(y_i-\mu)^{2} + \lambda \mu^{2}, \quad 1\leq i \leq n.
\end{equation}

\begin{theorem}\label{thm::regime_1}
Consider model \eqref{model::gaussian_sequence} and parameter space \eqref{eq::parameter-space-SNR}.
For Regime (\rom{1}) in which $\epsilon_{n} \rightarrow 0, \mu_{n} \rightarrow 0$ as $n\rightarrow \infty$, we have
\begin{equation*}
    R(\Theta(k_n,\tau_{n}),\sigma_{n}) =  n \sigma_{n}^{2} \left(\epsilon_{n} \mu_{n}^{2} - \epsilon_{n}^{2} \mu_{n}^{4} \left(1+o(1)\right) \right).
\end{equation*}
In addition, the linear estimator $\hat{\eta}_L(y, \lambda_n)$ with tuning $\lambda_n=\left(\epsilon_{n}\mu_{n}^{2}\right)^{-1}$ is asymptotically minimax up to the second order term, i.e.
\[
 \sup_{\theta \in \Theta(k_n,\tau_{n}) } \mathbb{E}_{\theta} \norm{\hat{\eta}_L (y, \lambda_n) - \theta}_{2}^{2} = n \sigma_{n}^{2} \left(\epsilon_{n} \mu_{n}^{2} - \epsilon_{n}^{2} \mu_{n}^{4} \left(1+o(1)\right) \right).
\]
%$\lambda_{n} = \epsilon_{n}\mu_{n}^{2} / \left(\epsilon_{n}\mu_{n}^{2} + 1\right)$.
\end{theorem}

% \begin{theorem}\label{thm::regime_1}
% Consider model \eqref{model::gaussian_sequence} and parameter space \eqref{eq::parameter-space-SNR}.
% For Regime (\rom{1}) in which $\epsilon_{n} \rightarrow 0, \mu_{n} \rightarrow 0$ as $n\rightarrow \infty$, we have
% \begin{equation*}
%     R(\Theta(k_n,\tau_{n}),\sigma_{n}) =  n \sigma_{n}^{2} \left(\epsilon_{n} \mu_{n}^{2} - \epsilon_{n}^{2} \mu_{n}^{4} \left(1+o(1)\right) \right).
% \end{equation*}
% In addition, the linear estimator $\hat{\eta}_L(y, \lambda_n)$ with tuning $\lambda_n=\left(\epsilon_{n}\mu_{n}^{2}\right)^{-1}$ is asymptotically minimax up to the second order term, i.e.
% \[
%  \sup_{\theta \in \Theta(k_n,\tau_{n}) } \mathbb{E}_{\theta} \norm{\hat{\eta}_L (y, \lambda_n) - \theta}_{2}^{2} = n \sigma_{n}^{2} \left(\epsilon_{n} \mu_{n}^{2} - \epsilon_{n}^{2} \mu_{n}^{4} \left(1+o(1)\right) \right).
% \]
% %$\lambda_{n} = \epsilon_{n}\mu_{n}^{2} / \left(\epsilon_{n}\mu_{n}^{2} + 1\right)$.
% \end{theorem}

The proof of this theorem can be found in Section \ref{proof:them:regime1}. Compared with Theorem \ref{thm:firstorder:firstregime}, Theorem \ref{thm::regime_1} obtains the additional second dominating term in the minimax risk. This negative term quantifies the amount of improvement that can be possibly achieved over the trivial zero estimator (whose supremum risk exactly equals $n\sigma_n^2\epsilon_n\mu_n^2$). Indeed, the non-trivial linear estimator $\hat{\eta}_L(y, \lambda_n)$ has supremum risk matching with the minimax risk up to the second order. Therefore, through the lens of second-order asymptotics, we discover a new minimax optimal estimator that outperforms the zero estimator recommended from the first-order analysis.

The second-order optimality of the linear estimator $\hat{\eta}_L(y, \lambda_n)$ in Regime (\rom{1}) raises the following question: how do non-linear estimators compare with $\hat{\eta}_L(y, \lambda_n)$? For instance, the soft thresholding estimator $\etahatS(y,\lambda)$ in \eqref{eq::softthresh} with $\lambda=\infty$ recovers the zero estimator and is hence first-order optimal. Can $\etahatS(y,\lambda)$ with proper tuning become second-order asymptotically optimal in this regime? The following theorem shows that the answer is negative.  

\begin{prop}\label{lem::lasso-risk-regime-1}
Consider model \eqref{model::gaussian_sequence} and parameter space \eqref{eq::parameter-space-SNR}. In Regime (\rom{1}) where $\epsilon_{n} \rightarrow 0, \mu_{n} \rightarrow 0$ as $n \rightarrow \infty$, the optimally tuned soft thresholding estimator $\etahatS(y, \lambda)$ has supremum risk:
\begin{align*}
    & \inf_{\lambda} \sup_{\theta \in \Theta(k_n,\tau_{n})} \mathbb{E}_{\theta} \norm{\etahatS(y,\lambda) - \theta}_{2}^{2} \nonumber \\
    =  & n \sigma_{n}^{2} \left( \epsilon_{n}\mu_{n}^{2} - \exp\left[ -\frac{1}{2} \frac{1}{\mu_{n}^{2}} \left(\log \frac{1}{\epsilon_{n}}\right)^{2}\left(1 + o(1) \right)\right] \right).
\end{align*}
\end{prop}
The proof of this proposition can be found in Section \ref{app:proof_lasso_regime_1}.

It is straightforward to confirm that $$\exp\Big[ -\frac{1}{2} \frac{1}{\mu_{n}^{2}} \left(\log \frac{1}{\epsilon_{n}}\right)^{2}\left(1 + o(1) \right)\Big]/(\epsilon_n^2\mu_n^4)=o(1)$$ under the scaling $\epsilon_n\rightarrow 0,\mu_n\rightarrow 0$. Hence, soft thresholding $\etahatS(y,\lambda)$ is outperformed by the linear estimator $\hat{\eta}_L(y, \lambda_n)$ and is sub-optimal (up to second order). A similar result can be proved for the hard thresholding estimator as well.

\begin{prop}\label{lem::hardthreshold-risk-regime-1}
Consider model \eqref{model::gaussian_sequence} and parameter space \eqref{eq::parameter-space-SNR}. In Regime (\rom{1}) where $\epsilon_{n} \rightarrow 0, \mu_{n} \rightarrow 0$ as $n \rightarrow \infty$, the optimally tuned hard thresholding estimator $\etahatH(y, \lambda)$ has supremum risk:
\begin{align*}
     \inf_{\lambda} \sup_{\theta \in \Theta(k_n,\tau_{n})} \mathbb{E}_{\theta} \norm{\etahatH(y,\lambda) - \theta}_{2}^{2} 
    =  n \sigma_{n}^{2}  \epsilon_{n}\mu_{n}^{2}.
\end{align*}
\end{prop}
The proof of this proposition is presented in Section \ref{proof:hardthreshold:reg1}.

% It is straightforward to confirm that $\exp\Big[ -\frac{1}{2} \frac{1}{\mu_{n}^{2}} \left(\log \frac{1}{\epsilon_{n}}\right)^{2}\left(1 + o(1) \right)\Big]/(\epsilon_n^2\mu_n^4)=o(1)$ under the scaling $\epsilon_n\rightarrow 0,\mu_n\rightarrow 0$. Hence, soft thresholding $\etahatS(y,\lambda)$ is outperformed by the linear estimator $\hat{\eta}_L(y, \lambda_n)$ and is sub-optimal (up to second order). A similar result can be proved for the hard thresholding estimator as well.\\

% \textbf{new proposition}

% \begin{prop}\label{lem::hardthreshold-risk-regime-1}
% Consider model \eqref{model::gaussian_sequence} and parameter space \eqref{eq::parameter-space-SNR}. In Regime (\rom{1}) where $\epsilon_{n} \rightarrow 0, \mu_{n} \rightarrow 0$ as $n \rightarrow \infty$, the optimally tuned hard thresholding estimator $\etahatS(y, \lambda)$ has supremum risk:
% \begin{align*}
%      \inf_{\lambda} \sup_{\theta \in \Theta(k_n,\tau_{n})} \mathbb{E}_{\theta} \norm{\etahatH(y,\lambda) - \theta}_{2}^{2} 
%     =  n \sigma_{n}^{2}  \epsilon_{n}\mu_{n}^{2}.
% \end{align*}

% \end{prop}

The fact that $\hat{\eta}_L(y, \lambda_n)$ is optimal and $\etahatS(y,\lambda)$ and $\etahatH(y,\lambda)$  are sub-optimal in Regime (\rom{1}) is intriguing. It says that the former non-sparse estimator is better than the latter sparse ones for recovering sparse signals. In fact, the result further implies that any sparsity-promoting procedure cannot improve over a simple linear shrinkage for the recovery of sparse signals. A high-level explanation is that since Regime (\rom{1}) has low signal-to-noise ratio in which variance is the dominating factor of mean squared error, linear shrinkage achieves a better balance between bias and variance than those more ``aggressive" sparsity-inducing operations. These results demonstrate the practical relevance of SNR-aware minimaxity as opposed to the classical minimax approach.

\subsubsection{Results in Regime (\rom{2})}

We now move on to discuss Regime (\rom{2}) where new minimaxity results arise as the signal-to-noise ratio increases. Introduce an estimator $\hat{\eta}_E(y,\lambda,\gamma)=\frac{\hat{\eta}_S(y,\lambda)}{1+\gamma}\in \mathbb{R}^n$ with coordinates:
\begin{align}\label{def::elastic-estimator}
    & [\hat{\eta}_E(y,\lambda, \gamma)]_i = \frac{[\etahatS(y,\lambda)]_i}{1+\gamma} \nonumber \\
    = & \argmin_{u\in \mathbb{R}}~(y_i-u)^{2} + 2\lambda |u|+\gamma u^{2}, \quad 1\leq i \leq n.
\end{align}
The estimator $\hat{\eta}_E(y,\lambda,\gamma)$ is a composition of soft thresholding and linear shrinkage. It can be considered as an``interpolation" between soft thresholding estimator and linear estimator. 

\begin{theorem}\label{thm::regime_2}
Consider model \eqref{model::gaussian_sequence} and parameter space \eqref{eq::parameter-space-SNR}.
For Regime (\rom{2}) in which $\epsilon_{n} \rightarrow 0, \mu_{n} \rightarrow \infty, \mu_{n} = o(\sqrt{\log \epsilon_{n}^{-1}})$ as $n\rightarrow \infty$, we have
\begin{equation*}
    R(\Theta(k_n,\tau_{n}),\sigma_{n}) \geq n \sigma_{n}^{2} \Big(\epsilon_{n} \mu_{n}^{2} - \frac{1}{2}\epsilon_{n}^{2} \mu_{n}^{2} e^{\mu_{n}^{2}} \left(1+o(1)\right) \Big). 
\end{equation*}
In addition, based on the estimator $\hat{\eta}_E(y,\lambda_n,\gamma_n)$ with tuning parameters $\lambda_n=2\sigma_n\mu_{n}$, and $\gamma_n=(2\epsilon_{n}\mu_{n}^{2}e^{\frac{3}{2}\mu_{n}^{2}})^{-1}-1$, we have
\begin{eqnarray*}
    && R(\Theta(k_n,\tau_{n}),\sigma_{n}) \nonumber \\
    &\leq& \sup_{\theta \in \Theta(k_n,\tau_{n})} \mathbb{E}_{\theta}\norm{\hat{\eta}_E(y,\lambda_n,\gamma_n) - \theta}_{2}^{2} \\
    &=& n \sigma_{n}^{2} \left(\epsilon_{n} \mu_{n}^{2} - (\sqrt{2/\pi}+o(1))\epsilon_{n}^{2} \mu_{n} e^{\mu_{n}^{2}}  \right).
\end{eqnarray*}
\end{theorem}
The proof of this theorem can be found in Section \secThmIV.

\begin{remark}
Theorem \ref{thm::regime_2} does not provide a tight upper or lower bound for the minimax risk approximation. However, the upper bound given by $\hat{\eta}_E(y,\lambda_n,\gamma_n)$ only differs from the lower bound up to an order of $\mu_{n}$ in the second order term. Note that this difference is very small in view of the occurrence of $e^{\mu_{n}^{2}}$ in the second order term. In this sense, the estimator $\hat{\eta}_E(y,\lambda_n,\gamma_n)$ is nearly optimal in Regime (\rom{2}). In this theorem, we believe that the upper bound is not necessarily sharp. In fact, we anticipate that there may be other estimators capable of outperforming $\hat{\eta}_E(y,\lambda_n,\gamma_n)$. Our next theorem (Theorem \ref{thm::regime_2_compactness}) gives an accurate second order term for the minimax risk in Regime (\rom{2}), under a uniform boundedness condition on parameter coordinates in the parameter space. However, as will be elaborated in the proof, the technique employed to establish the upper bound on the minimax risk is not constructive and does not identify the minimax estimator.
\end{remark}

\begin{theorem}\label{thm::regime_2_compactness}
Consider model \eqref{model::gaussian_sequence} with the following parameter space:
\begin{align}\label{eq::parameter-space-compact}
    \Theta^{A}(k_n,\tau_{n}) := \Big\{ & \theta \in \mathbb{R}^{n}:
    ~\norm{\theta}_{0} \leq k_n, \nonumber \\
    & ~\norm{\theta}_{2}^{2} \leq k_n \tau_{n}^{2}, ~\norm{\theta}_{\infty} \leq A \tau_{n} \Big \}. 
\end{align}
For Regime (\rom{2}) in which $\epsilon_{n} \rightarrow 0, \mu_{n}  \rightarrow \infty, \mu_{n} = o(\sqrt{\log \epsilon_{n}^{-1}})$ as $n\rightarrow \infty$, we have that for any constant $A>1$,
\begin{equation*}
    R(\Theta^{A}(k_n,\tau_{n}),\sigma_{n}) = n \sigma_{n}^{2} \left(\epsilon_{n} \mu_{n}^{2} - \frac{1}{2}\epsilon_{n}^{2} \mu_{n}^{2} e^{\mu_{n}^{2}} \left(1+o(1)\right) \right).
\end{equation*}
\end{theorem}
The theorem is proved in Section \secThmV{}.

Now let us interpret the above results. First note that in Regime (\rom{2}), compared to Regime (\rom{1}), the magnitude of the second order term (relative to the first order term) is much larger, so that the possible improvement over the zero estimator is much more significant. This is expected as the SNR is higher compared to Regime (\rom{1}). Furthermore, the (near) optimality of $\hat{\eta}_E(y,\lambda_n,\gamma_n)$ showed in Theorem \ref{thm::regime_2} indicates that thresohlding and linear shrinkage together play an important role in estimating sparse signals in Regime (\rom{2}). To shed more light on it, the following three propositions prove that neither thresohlding estimators $\etahatS(y,\lambda), \etahatH(y,\lambda)$ nor linear estimator $\hat{\eta}_L(y, \lambda)$ alone is close to optimal.

\begin{prop}\label{lem::lasso-risk-regime-2}
Consider model \eqref{model::gaussian_sequence} and parameter space \eqref{eq::parameter-space-SNR}. In Regime (\rom{2}) where $\epsilon_{n}\rightarrow 0, \mu_{n} \rightarrow \infty$, $\mu_{n} = o(\sqrt{\log \epsilon_{n}^{-1}})$, as $n \rightarrow \infty$, the optimally tuned soft thresholding estimator has supremum risk:
\begin{eqnarray*}
    && \inf_{\lambda} \sup_{\theta \in \Theta(k_n,\tau_{n})} \mathbb{E}_{\theta} \norm{\etahatS(y,\lambda) - \theta}_{2}^{2} \nonumber \\
    &=& n \sigma_{n}^{2} \left( \epsilon_{n}\mu_{n}^{2} - \exp\left[ -\frac{1}{2} \frac{1}{\mu_{n}^{2}} \left(\log \frac{1}{\epsilon_{n}}\right)^{2}\left(1 + o(1) \right)\right] \right).
\end{eqnarray*}
\end{prop}
The proof of this proposition can be found in Section \secPropII{}.

\begin{prop}\label{lem::hardthreshold-risk-regime-2}
Consider model \eqref{model::gaussian_sequence} and parameter space \eqref{eq::parameter-space-SNR}. In Regime (\rom{2}) where $\epsilon_{n} \rightarrow 0, \mu_{n} \rightarrow \infty$, $\mu_n = o(\sqrt{\log \epsilon_n^{-1}})$ as $n \rightarrow \infty$, the optimally tuned hard thresholding estimator $\etahatH(y, \lambda)$ has supremum risk:
\begin{align*}
     \inf_{\lambda} \sup_{\theta \in \Theta(k_n,\tau_{n})} \mathbb{E}_{\theta} \norm{\etahatH(y,\lambda) - \theta}_{2}^{2} 
    =  n \sigma_{n}^{2}  \epsilon_{n}\mu_{n}^{2}.
\end{align*}
\end{prop}
The proof of this proposition is presented in Section \ref{proof:hardthreshold:reg2}.

\begin{prop}\label{lem::ridge-risk-regime-2}
Consider model \eqref{model::gaussian_sequence} and parameter space \eqref{eq::parameter-space-SNR}. In Regime (\rom{2}) where $\epsilon_{n} \rightarrow 0, \mu_{n} \rightarrow \infty, \mu_{n} = o(\sqrt{\log \epsilon_{n}^{-1}})$, as $n \rightarrow \infty$, the optimally tuned linear estimator has supremum risk:
\begin{equation*}
    \inf_{\lambda} \sup_{\theta \in \Theta(k_n,\tau_{n})} \mathbb{E}_{\theta} \norm{\hat{\eta}_L(y,\lambda) - \theta}_{2}^{2} = n \sigma_{n}^{2} \left( \epsilon_{n}\mu_{n}^{2} - \frac{\epsilon_{n}^{2}\mu_{n}^{4}}{1+\epsilon_{n}\mu_{n}^{2}} \right).
\end{equation*}
\end{prop}

The proof of this proposition can be easily followed by the discussion in Section \ref{sec::regime_1-upper-bound}.

Comparing the second order term in Theorem \ref{thm::regime_2} and Propositions \ref{lem::lasso-risk-regime-2}-\ref{lem::ridge-risk-regime-2} under the scaling condition $\epsilon_{n} \rightarrow 0, \mu_{n} \rightarrow \infty, \mu_{n} = o(\sqrt{\log \epsilon_{n}^{-1}})$, it is straightforward to verify that the supremum risk of $\hat{\eta}_E(y,\lambda_n,\gamma_n)$ is much smaller than that of optimally tuned soft thresholding, hard thresholding, and linear estimator. In light of what we have discussed in Regime (\rom{1}), the results in Regime (\rom{2}) deliver an interesting message: when SNR increases from low to moderate level, sparsity promoting operation becomes effective in estimating sparse signals; on the other hand, since SNR is not sufficiently high yet, a component of linear shrinkage towards zero still boosts the performance.

\subsubsection{Results in Regime (\rom{3})}

Finally, let us consider the high-SNR regime, i.e., Regime (\rom{3}). As shown in Theorem \ref{thm:firstorder:firstregime}, the first-order approximation of minimax risk claims that both hard and soft thresholding estimators are optimal. However, the refined second-order analysis will reveal that hard thresholding remains optimal while soft thresholding is in fact sub-optimal, up to the second order term. 

\begin{theorem}\label{thm::regime_4}
Consider model \eqref{model::gaussian_sequence} and parameter space \eqref{eq::parameter-space-SNR}. For Regime (\rom{3}) in which $\epsilon_{n} \rightarrow 0, \mu_{n} = \omega(\sqrt{\log \epsilon_{n}^{-1}})$ as $n \rightarrow \infty$, we have
\begin{eqnarray*}
    && R(\Theta(k_n,\tau_{n}), \sigma_{n}) \nonumber \\
    &=& n \sigma_{n}^{2} \left( 2 \epsilon_{n} \log \epsilon_{n}^{-1} - 2 \epsilon_{n} \nu_{n} \sqrt{2 \log \nu_{n}} \left(1 + o(1) \right) \right),
\end{eqnarray*}
 where $\nu_{n}:= \sqrt{2 \log \epsilon_{n}^{-1}}$. In addition, the hard thresholding $\etahatH(y, \lambda_n)$ with tuning $\lambda_n=\sigma_{n}\sqrt{2\log\epsilon_{n}^{-1}}$ is asymptotically minimax up to the second order term, i.e.
\begin{eqnarray*}
  && \sup_{\theta \in \Theta(k_n,\tau_{n})} \mathbb{E}_{\theta} \norm{\etahatH(y,\lambda_n) - \theta}_{2}^{2} \nonumber \\
  &=& n \sigma_{n}^{2} \left( 2 \epsilon_{n} \log \epsilon_{n}^{-1} - 2 \epsilon_{n} \nu_{n} \sqrt{2 \log \nu_{n}} \left(1 + o(1) \right) \right).
\end{eqnarray*}
\end{theorem}
The proof of this theorem can be found in Section \secThmVI. 
Before we interpret this result, let us obtain the risk of the soft thresholding estimator and linear estimator as well.

\begin{prop}\label{lem::lasso-risk-regime-4}
Consider model \eqref{model::gaussian_sequence} and parameter space \eqref{eq::parameter-space-SNR}. In Regime (\rom{3}) where $\epsilon_{n} \rightarrow 0, \mu_{n} = \omega(\sqrt{\log \epsilon_{n}^{-1}})$ as $n \rightarrow \infty$, the optimally tuned soft thresholding achieves the supremum risk:
\begin{eqnarray*}
    && \inf_{\lambda} \sup_{\theta \in \Theta(k_n,\tau_{n})} \mathbb{E}_{\theta} \norm{\etahatS(y,\lambda) - \theta}_{2}^{2} \nonumber \\
    &=& n \sigma_{n}^{2} \left( 2 \epsilon_{n} \log \epsilon_{n}^{-1} - 6 \epsilon_{n} \log \nu_{n} \left(1 + o(1) \right) \right),
\end{eqnarray*}
where $\nu_{n}= \sqrt{2 \log \epsilon_{n}^{-1}}$.
\end{prop}

The proof of the proposition can be found in Section \secPropIV{}.

\begin{prop}\label{prop:ridge:thirdregime}
Consider model \eqref{model::gaussian_sequence} and parameter space \eqref{eq::parameter-space-SNR}. In Regime (\rom{3}) where $\epsilon_{n} \rightarrow 0, \mu_{n} = \omega(\sqrt{\log \epsilon_{n}^{-1}})$ as $n \rightarrow \infty$, the optimally tuned linear estimator achieves the supremum risk:
\begin{eqnarray*}
    && \inf_{\lambda} \sup_{\theta \in \Theta(k_n,\tau_{n})} \mathbb{E}_{\theta} \norm{\etahatL(y,\lambda) - \theta}_{2}^{2} \nonumber \\
    &=& \frac{n\sigma_n^2\epsilon_n\mu_n^2}{1+\epsilon_n\mu_n^2} = \omega(n \sigma_n^2 \epsilon_n \log(\epsilon_n^{-1})). 
\end{eqnarray*}
\end{prop}
The proof of this proposition is presented in Section \ref{sec:proof:thridregime:ridge}. 

% \begin{prop}\label{lem::lasso-risk-regime-4}
% Consider model \eqref{model::gaussian_sequence} and parameter space \eqref{eq::parameter-space-SNR}. In Regime (\rom{3}) where $\epsilon_{n} \rightarrow 0, \mu_{n} = \omega(\sqrt{\log \epsilon_{n}^{-1}})$ as $n \rightarrow \infty$, the optimally tuned soft thresholding achieves the supremum risk:
% \begin{eqnarray*}
%     && \inf_{\lambda} \sup_{\theta \in \Theta(k_n,\tau_{n})} \mathbb{E}_{\theta} \norm{\etahatS(y,\lambda) - \theta}_{2}^{2} \nonumber \\
%     &=& n \sigma_{n}^{2} \left( 2 \epsilon_{n} \log \epsilon_{n}^{-1} - 6 \epsilon_{n} \log \nu_{n} \left(1 + o(1) \right) \right),
% \end{eqnarray*}
% where $\nu_{n}= \sqrt{2 \log \epsilon_{n}^{-1}}$.
% \end{prop}

% The proof of the proposition can be found in Section \secPropIV{}.

Combining the above three results, we can conclude that overall in Regime (\rom{3}) hard thresholding offers a better estimate than soft thresholding and linear shrinkage. The intuition is that Regime (\rom{3}) has a high SNR where bias becomes the dominating factor of mean squared error, therefore hard thresholding has an edge on soft thresholding and linear shrinkage by producing zero coordinates while not shrinking the above-threshold coordinates. Moreover, note that the difference between the first order and second order terms in the minimax risk is smaller than $\sqrt{\log \epsilon_n^{-1}}$. This implies that the second order term in our approximations can be relevant in a wide range of sparsity levels. 

\section{Numerical experiments}
\label{simulations}

As discussed in Section \ref{sec::misleading-sparse-model} through one simulation example, classical minimax results are inadequate for characterizing the role of signal-to-noise ratio (SNR) in the estimation of sparse signals. Hence, we developed the new SNR-aware minimax framework in Section \ref{SNR-aware:minmaxity} to overcome the limitations of the classical minimaxity. In this section, we provide more empirical results to evaluate the points we discussed above. 

We generate the signal $\theta$ in the following way: for a sample size $n$, $\theta = (\theta_{1}, \ldots, \theta_{n})$ is generated by assigning $\tau_n$ to a random choice of $k_{n}$ coordinates and setting the others to zero. Then $y=(y_{1}, \ldots, y_{n})$ and $z=(z_{1}, \ldots, z_{n})$ are generated according to Model \eqref{model::gaussian_sequence} for a certain noise level $\sigma_n$.

Given the sample size $n$, we consider three sparsity levels $k_{n} =\lfloor n^{2/3} \rfloor$, $\lfloor n^{3/4} \rfloor$, $\lfloor n^{1/2} \rfloor$, so that $\epsilon_{n} = k_{n}/n \rightarrow 0$ as $n\rightarrow\infty$. In addition, since SNR is decided by $\mu_{n} = \tau_{n}/\sigma_{n}$, without the loss of generality, we fix the value of the signal strength $\tau_{n} = 10$. We demonstrate our findings in two ways: 
\begin{enumerate}
    \item Let $\mu_{n}$ change from small to large values, and plot the mean squared error (MSE) of different estimators as a function of $\mu_n$. 
    \item Let $\sigma_{n}$ change from small to large values, and plot the MSE as a function of $\sigma_n$. 
    
\end{enumerate}
    
 In our experiments, we consider moderate sample size $n=500$ and large sample size $n=5000$. We consider the four estimators that have been extensively discussed in the previous sections: linear estimator $\etahatL$ defined in \eqref{def::ridge-estimator}, soft thresholding $\etahatS$ defined in \eqref{eq::softthresh}, hard thresholding $\etahatH$ defined in \eqref{eq::hardthresh}, and the soft-linear ``interpolation" estimator $\etahatE$ defined in \eqref{def::elastic-estimator} (since $\etahatE$ is the composition of soft thresholding and linear shrinkage, we refer to it as soft-linear ``interpolation" for convenience). We evaluate the performance of estimators using the empirical MSE scaled by the total signal strength: $\|\theta\|_{2}^{-2}\cdot\sum_{i=1}^{n}(\hat{\theta}_{i} - \theta_{i})^{2}$. The MSEs shown in Figures \ref{fig:MSEvsSigma1}-\ref{fig:MSEvsMu2} are averaged over 20 repetitions, plotted with 95\% confidence intervals from t-distribution. For each estimator, tuning parameters are chosen by grid search to obtain the minimum possible MSE.

\begin{figure}[!thbp]
\centering
\setlength\tabcolsep{-8.pt}
\begin{tabular}{c}
\includegraphics[scale=0.35]{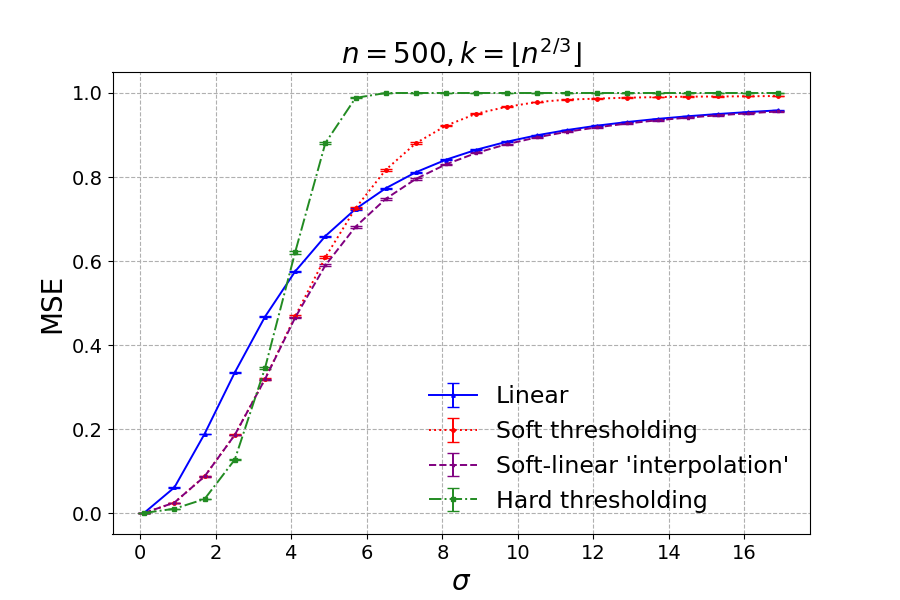} \\
\includegraphics[scale=0.35]{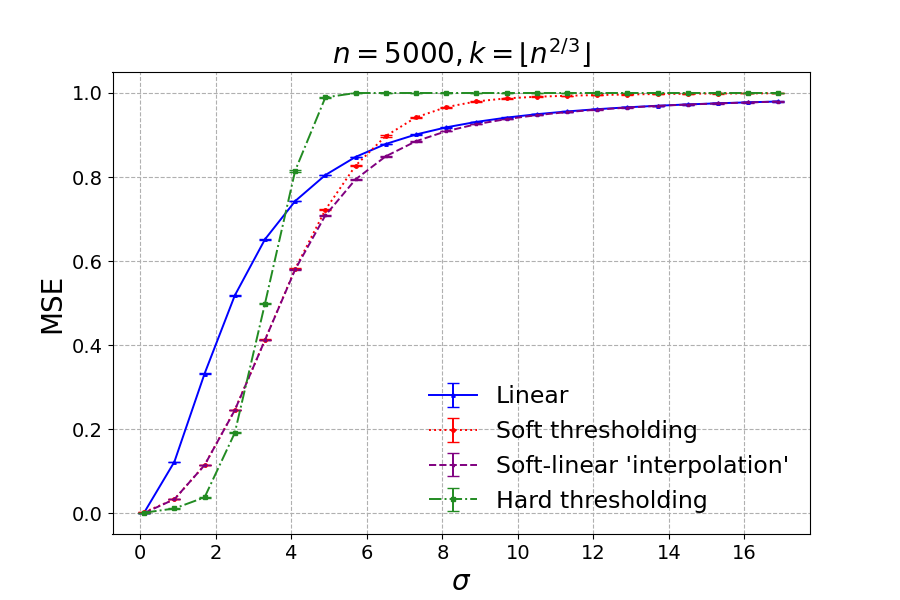} \\
\includegraphics[scale=0.35]{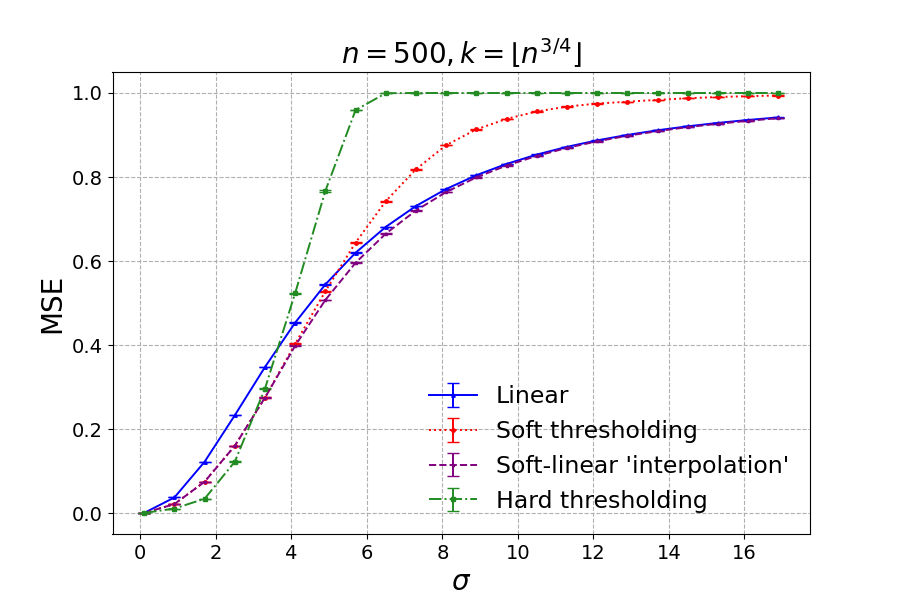} \\
\includegraphics[scale=0.35]{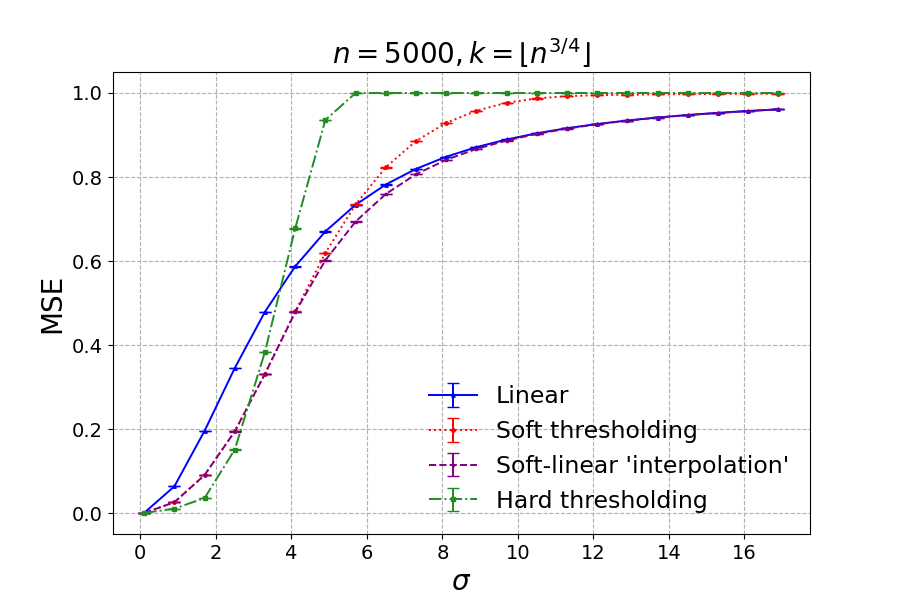}
\end{tabular}
\vspace{-0.3cm}
\caption{Mean squared error comparison at different noise levels. On each graph, the y-axis is the scaled MSE, and the x-axis is the noise standard deviation $\sigma_n$. (to be continued in Fig. \ref{fig:MSEvsSigma2})}
\label{fig:MSEvsSigma1}
\end{figure}

\begin{figure}[!thbp]
\centering
\setlength\tabcolsep{-8.pt}
\begin{tabular}{c}
\includegraphics[scale=0.35]{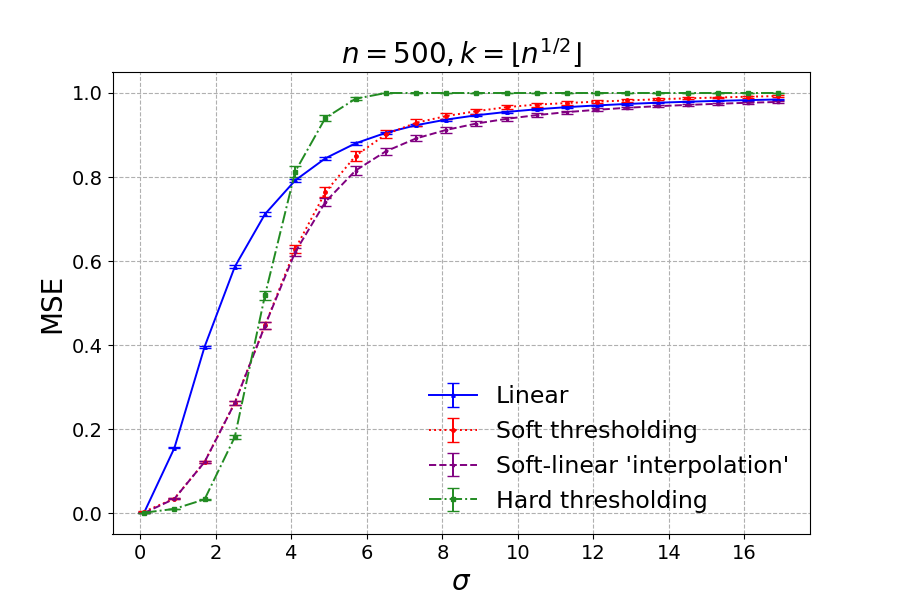} \\
\includegraphics[scale=0.35]{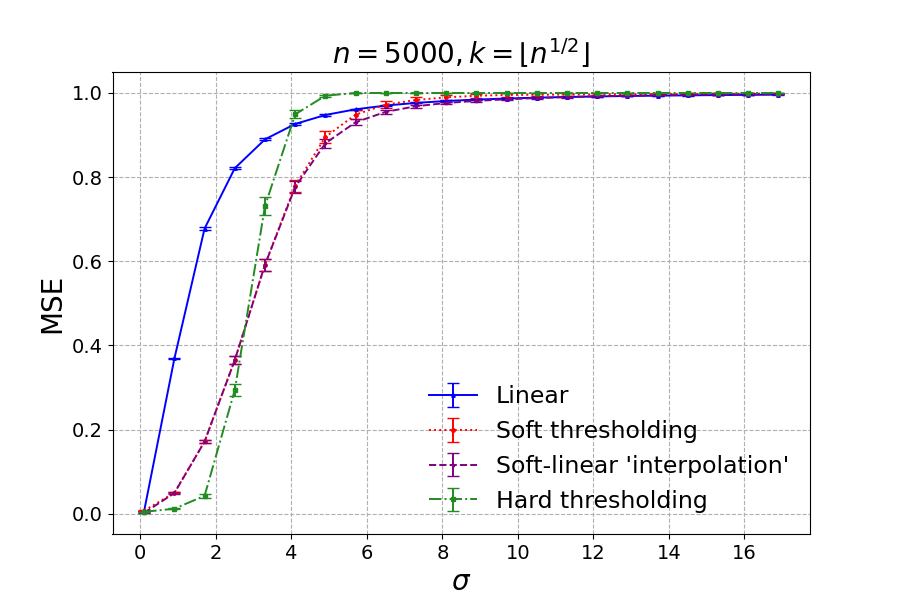}
\end{tabular}
\vspace{-0.45cm}
\caption{(continued from Fig. \ref{fig:MSEvsSigma1})}
\label{fig:MSEvsSigma2}
\end{figure}

\begin{figure}[!thbp]
\centering
\setlength\tabcolsep{-8.pt}
\begin{tabular}{c}
\includegraphics[scale=0.35]{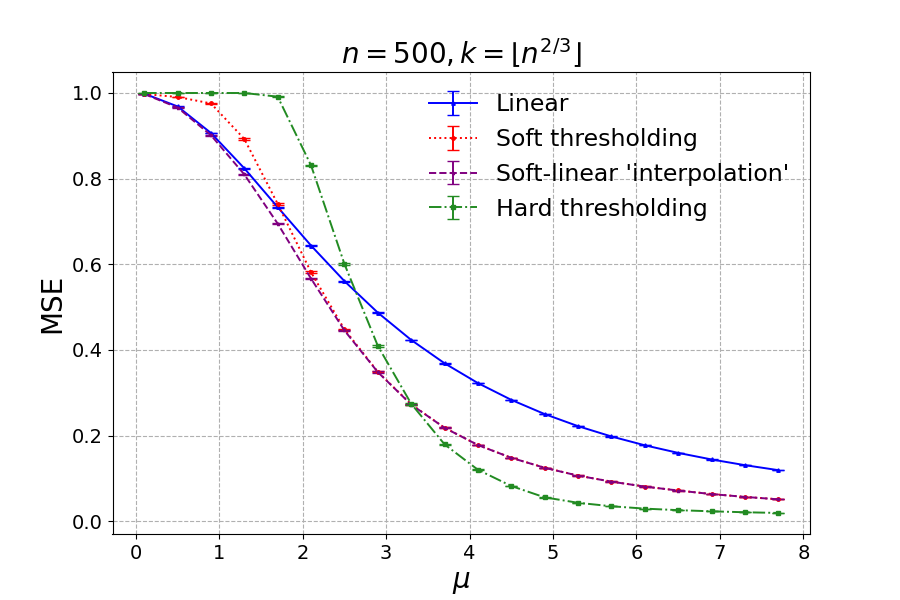} \\
\includegraphics[scale=0.35]{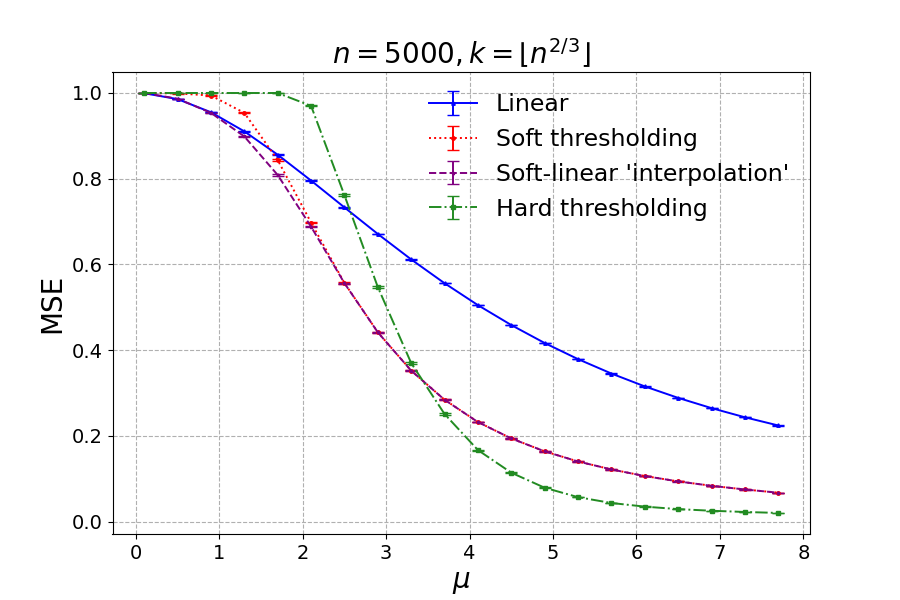}
\end{tabular}
\vspace{-0.3cm}
\caption{Mean squared error comparison at different SNR levels. On each graph, the $y$-axis is the scaled MSE, and the $x$-axis is the SNR $\mu_n$. (to be continued in Fig. \ref{fig:MSEvsMu2})}
\label{fig:MSEvsMu1}
\end{figure}

\begin{figure}[!thbp]
\centering
\setlength\tabcolsep{-8.pt}
\begin{tabular}{c}
\includegraphics[scale=0.35]{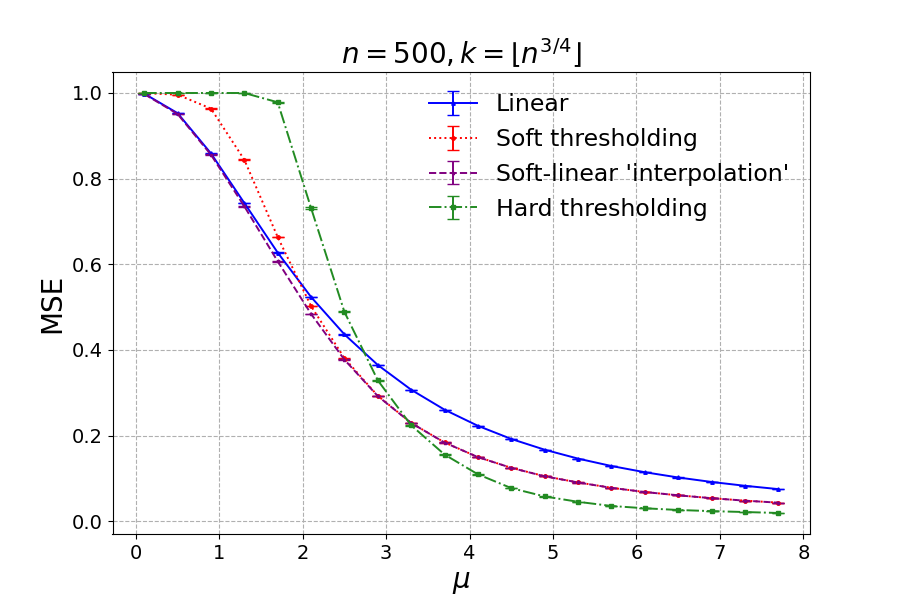} \\
\includegraphics[scale=0.35]{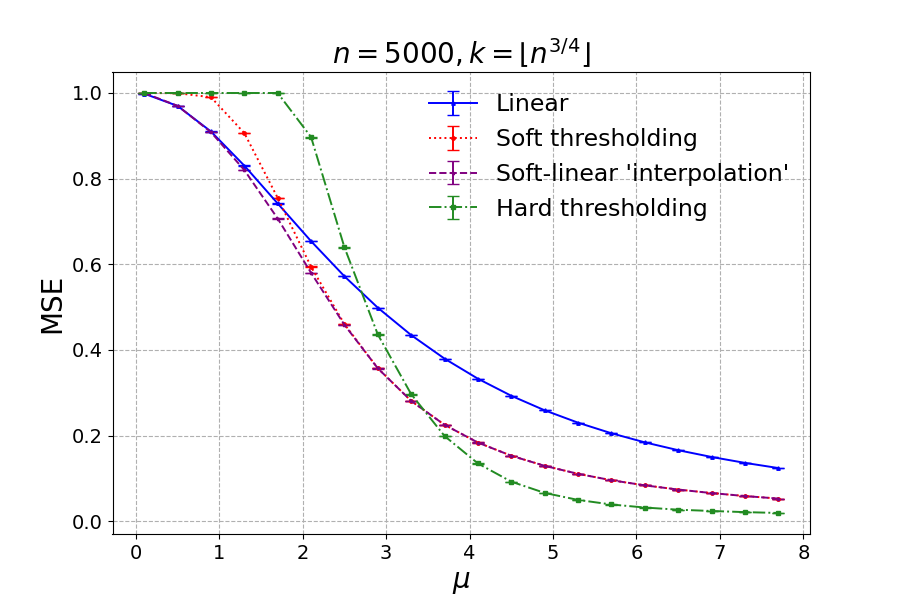} \\
\includegraphics[scale=0.35]{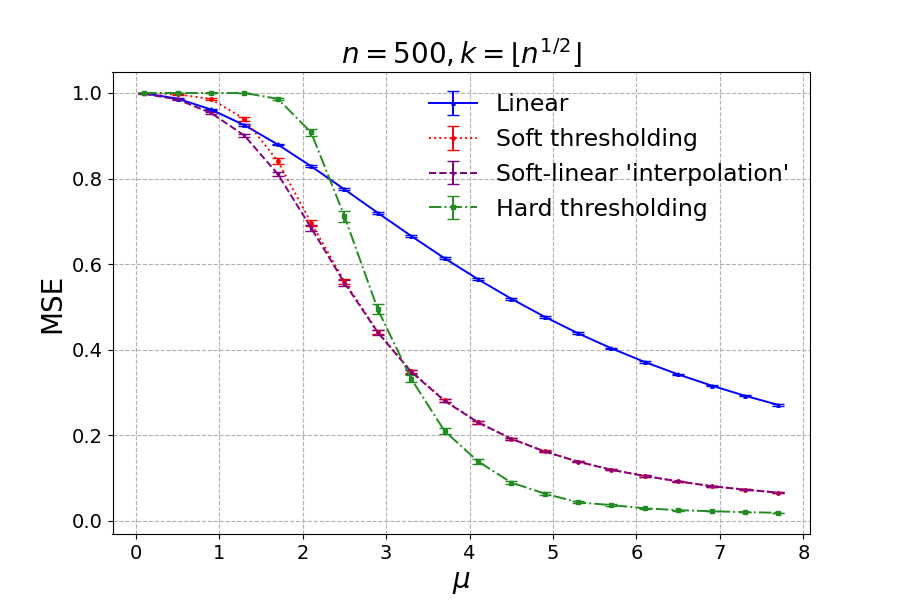} \\
\includegraphics[scale=0.35]{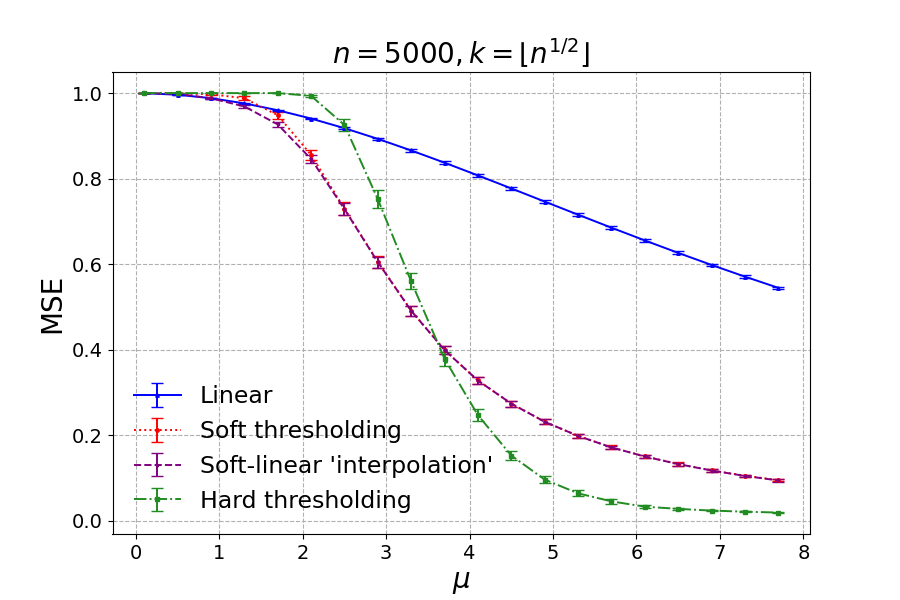}
\end{tabular}
\vspace{-0.45cm}
\caption{(continued from Fig. \ref{fig:MSEvsMu1})}
\label{fig:MSEvsMu2}
\end{figure}

From Figures \ref{fig:MSEvsSigma1}-\ref{fig:MSEvsSigma2}, when $\sigma_n$ changes from small to large values, we observed that: (1) When $\sigma_n$ is near zero, hard thresholding achieves the minimum MSE among the four estimators discussed in previous sections. This corresponds to Regime (\rom{3}) in our theory. (2)  When $\sigma_n$ is in moderate area, the soft-linear `interpolation' estimator $\etahatE$ has the minimum empirical MSE. This corresponds to Regime (\rom{2}) in our theory. (3) When $\sigma_n$ becomes large, the linear estimator $\etahatL$ as well as the optimally tuned $\etahatE$ (since $\etahatE$ can achieve $\etahatL$ when optimally tuned) have the minimum empirical MSE. Our theory in Regime (\rom{1}) states that when SNR is small, $\etahatL$ becomes asymptotically minimax optimal. The empirical studies align well with our current theory. 

Figures \ref{fig:MSEvsMu1}-\ref{fig:MSEvsMu2} offer similar conclusions as the ones we mentioned above. The main difference is that instead of revealing MSE as a function of the noise level, we view it as a function of SNR. Due to this difference, the leftmost part of each graph corresponds to Regime (\rom{1}). As $\mu_n$ increases, the curves will correspond to Regime (\rom{2}) and Regime (\rom{3}). In particular, when $\mu_n$ is large, it corresponds with the area of $\sigma_n$ near zero in Figures \ref{fig:MSEvsSigma1}-\ref{fig:MSEvsSigma2}. Here, it is shown more clearly that in the large SNR regime, hard thresholding has the minimum empirical MSE among all the estimators.

\section{Discussions}
\label{discuss}

\subsection{Summary}

We introduced two notions that can make the minimax results more meaningful and appealing for practical purposes: (i) signal-to-noise-ratio aware minimaxity, (ii) second-order asymptotic approximation of minimax risk. We showed that these two notions can alleviate the major drawbacks of the classical minimax results. For instance, while the classical results prove that the hard and soft thresholding estimators are minimax optimal, the new results reveal that in a wide range of low signal-to-noise ratios the two estimators are in fact sub-optimal. Even when the signal-to-noise ratio is high, only hard thresholding is optimal and soft thresholding remains sub-optimal. Furthermore, our refined minimax analysis identified three optimal (or nearly optimal) estimators in three regimes with varying SNR: hard thresholding $\hat{\eta}_H(y,\lambda)$ of \eqref{eq::hardthresh} in high SNR; $\hat{\eta}_E(y,\lambda,\gamma)$ of \eqref{def::elastic-estimator} in moderate SNR; linear estimator $\hat{\eta}_L(y,\lambda)$ of \eqref{def::ridge-estimator} in low SNR. As is clear from the definition of the three estimators, they are induced by $\ell_0$-regularization, elastic net regularization \cite{zou2005regularization} and  $\ell_2$-regularization, respectively. These regularization techniques have been widely used in statistics and machine learning \cite{hastie2009elements}.

The concepts of signal-to-noise ratio aware minimaxity and higher-order asymptotic approximations introduced in this paper may open up new venues for investigating various estimation problems. We have recently used the same framework to revisit the sparse estimation problem in high-dimensional linear regression and obtained new insights. That being said, it is important to acknowledge that the additional insights gained from this framework come with increased mathematical complexity when computing minimax estimators. Therefore, one direction we plan to explore in the future is the development of simpler and more general techniques for obtaining higher-order approximations of minimax risk or the supremum risk of well-established estimators. 

\subsection{Related works}

%To the best of our knowledge, our work is the first theoretical one that investigates the role of SNR in the context of exact sparsity learning from the minimax perspective. 

 There are some recent works on the significance of SNR for sparse learning. The extensive simulations conducted in the linear regression setting by \cite{hastie2020best} demonstrated that best subset selection ($\ell_0$-regularization) performs better than the lasso ($\ell_1$-regularization) in very high SNR, while the lasso outperforms best subset selection in low SNR regimes. \cite{hazimeh2020fast,mazumder2022subset} developed new variants of subset selection that can perform consistently well in various levels of SNR. Some authors of the current paper (with their collaborators) established sharp theoretical characterizations of $\ell_q$-regularization under varying SNR regimes in high-dimensional sparse regression and variable selection problems \cite{wang2020bridge, weng2018overcoming, zheng2017does}. In particular, their results revealed that among the $\ell_q$-regularization for $q \in [0,2]$, as SNR
decreases from high to low levels, the optimal value of $q$ for parameter estimation and variable selection will move from 0 towards 2. All the aforementioned works studied the impact of SNR on several or a family of popular estimators. Hence their comparison conclusions are only applicable to a restricted set of estimators. In contrast, our work focused on minimax analysis that led to stronger optimality-type conclusions. For example, the preceding works showed that $\ell_2$-regularization outperforms other $\ell_q$-regularization when SNR is low. We obtained a stronger result that $\ell_2$-regularization is in fact (minimax) optimal among all the estimators in low SNR. 

In a separate work, the first order minimax optimality is also proved for other estimators, such as empirical Bayes estimators \cite{10.1214/08-AOS638}. However, as we discussed before, first order minimax analysis is inherently incapable of evaluating the impact of the SNR on the performance of different estimators.

The second-order analysis of the minimax risk of the Gaussian sequence model under the sparsity constraint has been discussed in  \cite{johnstone1994minimax}. To compare this paper with our work, we have to mention the following points: (1)
Such analysis still suffers from the fact that it disregards the effect of the signal-to-noise ratio. By restricting the signal-to-noise ratio, our SNR-aware minimax framework provides much more refined information about the minimax estimators. (2) In terms of the theoretical analysis, the SNR-aware minimax analysis requires much more delicate analysis compared to the classical settings where there is no constraint on the SNR. In particular, constructing and proving the least favorable distributions is more complicated in our settings compared to the classical setting. As a result, all the following steps of the proof become more complicated too.  

  We should also emphasize that minimax analysis over classes of $\ell_p$ balls (i.e., $\Theta=\{\theta:\|\theta\|_p\leq C_n\}$) for $p>0$ under Gaussian sequence model has been performed in \cite{donoho1994minimax, johnstone19, zhang2012minimax}. These works revealed that a notion of SNR involving $C_n$ and $\sigma_n$ plays a critical role in characterizing the asymptotic minimax risk and the optimality of linear or thresholding estimators. Finally, see \cite{butucea2018variable, collier2017minimax} for non-asymptotic minimax rate analysis of variable selection and functional estimation on sparse Gaussian sequence models.

\subsection{Future research}
Several important directions are left open for future research:
\begin{itemize}
    \item The paper considered estimating signals with sparsity $k_n/n\rightarrow 0$. The other denser regime where $k_n/n\rightarrow c>0$ is also important to study. This will provide complementary asymptotic insights into the estimation of signals with varying sparsity. There exists classical minimax analysis along this line (see Chapter 8 in \cite{johnstone19}). A generalization of SNR-aware minimaxity to this regime is an interesting future work.
     \item The obtained minimax optimal estimators involve tuning parameters that depend on unknown quantities such as sparsity $k_n$ and signal strength $\tau_n$ from the parameter space. It is important to develop fully data-driven estimators that retain optimality for practical use. Hence, adaptive minimaxity is the next step, and classical adaptivity results (e.g., \cite{johnstone19}) may be helpful for the development.
   \item In this paper, we have focused on the parameter spaces that imposed the exact sparsity on $\theta$. Sparsity promoting denoisers such as hard thresholding and soft thresholding have  been also used over other structured parameter spaces such as Sobolev ellipsoids and Besov bodies. These parameter spaces usually characterize the smoothness properties of functions in terms of their Fourier or wavelet coefficients. We refer to \cite{gine2021mathematical, johnstone19, Tsybakov:2008:INE:1522486} and references therein for a systematic treatment of this topic. An interesting future research would be to explore the implications of the SNR-aware minimaixty and higher-order approximation of the minimax risk for such spaces.  
   \item The current work focused on the classical sparse Gaussian sequence model. It would be interesting to pursue a generalization to high-dimensional sparse linear regressions. Existing works (see \cite{buhlmann2011statistics,wainwright2019high} and references there) established minimax rate optimality (with loose constants) which is not adequate to accurately capture the impact of SNR. Instead, the goal is to derive asymptotic approximations with sharp constants as we did for Gaussian sequence models. We believe that this is generally a very challenging problem without imposing specific constraint on the design matrix. A good starting point is to consider the ``compressed sensing" model whose design rows follow independent isotropic Gaussian distribution. We have made some major progress along this line and look forward to further development.
 \end{itemize}

\section{Proofs}\label{sec::proofs}
\subsection{Preliminaries}\label{sec:prelim}

\subsubsection{Scale invariance} \label{sec:scaling}

The minimax risk defined in \eqref{eq::minimax-risk} has the following scale invariance property
\[
R(\Theta(k_n,\tau_{n}),\sigma_{n}) = \sigma_{n}^{2} \cdot R(\Theta(k_n,\mu_{n}), 1),
\]
where we recall that $\mu_n=\tau_n/\sigma_n$. This can be easily verified by rescaling the Gaussian sequence model to have unit variance. Moreover, similar invariance holds for the four estimators considered in the paper. We state it without proof in the following: $\forall \sigma >0$,
\begin{align*}
&\sigma\cdot \hat{\eta}_S(y,\lambda)=\hat{\eta}_S(\sigma y,\sigma \lambda), \quad \sigma\cdot \hat{\eta}_H(y,\lambda)=\hat{\eta}_S(\sigma y,\sigma \lambda), \\
&\sigma\cdot \hat{\eta}_L(y,\lambda)=\hat{\eta}_L(\sigma y,\lambda), \quad \sigma\cdot \hat{\eta}_E(y,\lambda,\gamma)=\hat{\eta}_E(\sigma y,\sigma \lambda, \gamma).
\end{align*}
These invariance properties will be frequently used in the proof to reduce a problem to a simpler one under unit variance.

\subsubsection{Gaussian tail bound}\label{ssec:GTB}

Recall the notation that $\phi, \Phi$ denote the probability density function and cumulative distribution function of a standard normal random variable, respectively. The following Gaussian tail bound will be extensively used in the proof. 

\begin{lemma}[Exercise 8.1 in \cite{johnstone19}]\label{lem::gaussian-tail-mills-ratio}
Define
\begin{equation*}
    \Tilde{\Phi}_{l}(\lambda) := \lambda^{-1} \phi(\lambda) \sum_{k=0}^{l}\frac{(-1)^{k}}{k!} \frac{\Gamma(2k+1)}{2^{k}\lambda^{2k}},
\end{equation*}
where $\Gamma(\cdot)$ is the gamma function. Then, for each $k\geq 0$ and all $\lambda>0$:
\begin{equation*}
    \Tilde{\Phi}_{2k+1}(\lambda) \leq 1-\Phi(\lambda) \leq \Tilde{\Phi}_{2k}(\lambda).
\end{equation*}
\end{lemma}

\subsubsection{The minimax theorem}\label{app:minmax}

Consider the Gaussian sequence model:
\begin{eqnarray}
\label{gsm:one:more}
    y_i=\theta_i+\sigma z_i, \quad i=1,2,\ldots, n,
\end{eqnarray}
where $z_1,z_2,\ldots, z_n \overset{i.i.d.}{\sim}\mathcal{N}(0,1)$. If $\pi$ is a prior distribution of $\theta\in \mathbb{R}^n$, the integrated risk of an estimator $\hat{\theta}$ (with squared error loss) is
$B(\hat{\theta},\pi)=\E_{\pi}\E_{\theta}\|\hat{\theta}-\theta\|_2^2$, and the Bayes risk of $\pi$ is $B(\pi)=\inf_{\hat{\theta}}B(\hat{\theta},\pi)$. We state a version of minimax theorem suited to the Gaussian sequence model. The theorem allows to evaluate minimax risk by calculating the maximum Bayes risk over a class of prior distributions. 
 
 \begin{theorem}[Theorem 4.12 in \cite{johnstone19}]\label{thm:minimax}
 Consider the Gaussian sequence model \eqref{gsm:one:more}. Let $\mathcal{P}$ be a convex set of probability measures on $\mathbb{R}^n$. Then
 \begin{equation*}
     \inf_{\hat{\theta}}\sup_{\pi \in \mathcal{P}} B(\hat{\theta}, \pi) = \sup_{\pi\in\mathcal{P}} \inf_{\hat{\theta}} B(\hat{\theta}, \pi) = \sup_{\pi\in \mathcal{P}}B(\pi).
 \end{equation*}
 A maximising $\pi$ is called a least favorable distribution (with respect to $\mathcal{P}$).
 \end{theorem}

\subsubsection{Independence is less favorable}

We present a useful result that can often help find the least favorable distributions. Let $\pi$ be an arbitrary prior, so that the $\theta_{j}$ are not necessarily independent. Denote by $\pi_{j}$ the marginal distribution of $\theta_{j}$. Build a new prior $\Bar{\pi}$ by making the $\theta_{j}$ independent: $\Bar{\pi} = \prod_{j}\pi_{j}$. This product prior has a larger Bayes risk. 

\begin{theorem}[Lemma 4.15 in \cite{johnstone19}]\label{thm:indep_less_favor}
$B(\Bar{\pi})\geq B(\pi)$.
\end{theorem}
\subsubsection{A machinery for obtaining lower bounds for the minimax risk}
In our results, we are often interested in finding lower bounds for the minimax risk. The following elementary result taken from Chapter 4.3 of \cite{johnstone19} will be useful in those cases. 
\begin{theorem}
\label{thm::minimax-lower-bound}
Consider the minimax risk of a risk function $r(\cdot,\cdot)$ over a parameter set $\Theta$:
\begin{equation*}
    R(\Theta) = \inf_{\hat{\theta}}\sup_{\theta \in \Theta} r(\hat{\theta},\theta). 
\end{equation*}
Recall that $B(\pi)$ is the Bayes risk of prior $\pi$: $B(\pi)=\inf_{\hat{\theta}}\int r(\hat{\theta},\theta) \pi(d \theta)$. Let $\mathcal{P}$ denote a collection of probability measure, and $\supp  \mathcal{P}$ denote the union of all $\supp \pi$ for $\pi$ in $\mathcal{P}$. If 
\begin{equation*}
    B(\mathcal{P}) = \sup_{\pi \in \mathcal{P}} B(\pi),
\end{equation*}
then 
\begin{equation*}
    \supp  \mathcal{P} \subset \Theta \quad \Rightarrow \quad R(\Theta) \geq B(\mathcal{P}).
\end{equation*}
\end{theorem}

\subsection{Proof of Theorem \ref{thm::regime_1}}\label{proof:them:regime1}
% This section contains the proof of Theorem \ref{thm::regime_1}, where the roadmap is shared by the proof of several other theorems in the Supplementary Material. 
% \subsubsection{Roadmap of the proof}

To calculate the minimax risk $R(\Theta(k_n,\tau_{n}), \sigma_n)$, we first obtain an upper bound by computing the supremum risk of the linear estimator $\hat{\eta}_L(y,\lambda_n)$,
\[
R(\Theta(k_n,\tau_{n}), \sigma_n)\leq \sup_{\theta \in \Theta(k_n,\tau_{n})}\mathbb{E}_{\theta}\|\hat{\eta}_L(y,\lambda_n)-\theta\|_2^2.
\]
We then derive a matching lower bound based on Theorem \ref{thm::minimax-lower-bound}. In particular, we construct a particular prior supported on $\Theta(k_n,\tau_{n})$ (that is the least favorable prior at the level of approximation we require), and its corresponding Bayes risk leads to a sharp lower bound for the minimax risk. The detailed derivation of the upper and lower bounds is presented below.

\subsubsection{Upper bound} \label{sec::regime_1-upper-bound} 

Thanks to the simple form of the linear estimator $\hat{\eta}_L(y,\lambda_n)$, its supremum risk under tuning $\lambda_n=(\epsilon_n\mu_n^2)^{-1}$ can be computed in a straightforward way: for all $\theta \in \Theta(k_n,\tau_n)$,
\begin{eqnarray*}
    && \mathbb{E}_{\theta}\|\hat{\eta}_L(y,\lambda_n)-\theta\|_2^2=\E_{\theta} \sum_{i=1}^{n} \left(\frac{1}{1+\lambda_{n}} y_{i} - \theta_{i}\right)^{2} \\
    &=&  \sum_{i=1}^{n} \left[ \left(\frac{\lambda_{n}}{1+\lambda_{n}}\right)^{2} \theta_{i}^{2} + \left(\frac{1}{1+\lambda_{n}}\right)^{2}\sigma_n^2 \right] \nonumber \\
    &\leq & \frac{\lambda^2_{n}k_n\tau_n^2+n\sigma_n^2}{(1+\lambda_{n})^2} = \frac{n\sigma_n^2\epsilon_n\mu_n^2}{1+\epsilon_n\mu_n^2} \nonumber \\
    &=& n\sigma_n^2\epsilon_n\mu_n^2\cdot \Big(1-\epsilon_n\mu_n^2(1+\epsilon_n\mu_n^2)^{-1}\Big) \nonumber \\
    &=& n\sigma_n^2\epsilon_n\mu_n^2\cdot \Big(1-\epsilon_n\mu_n^2(1+o(1))\Big),
       \end{eqnarray*}
where we have used the assumption $\epsilon_n=k_n/n\rightarrow 0,\mu_n=\tau_n/\sigma_n\rightarrow 0$, and the constraint $\|\theta\|_2^2\leq k_n\tau_n^2, \forall \theta \in \Theta(k_n,\tau_n)$.
As a result, 
\begin{eqnarray*}
    && R(\Theta(k_n,\tau_{n}),\sigma_n) \leq \sup_{\theta \in \Theta(k_n,\tau_{n})} \mathbb{E}_{\theta}\|\hat{\eta}_L(y,\lambda_n)-\theta\|_2^2 \\
    &=& n\sigma_n^2\epsilon_n\mu_n^2\cdot \Big(1-\epsilon_n\mu_n^2(1+o(1))\Big).
\end{eqnarray*}

\subsubsection{Lower bound} \label{sec::lower-bound-regime-1}

First, due to the scale invariance property $R(\Theta(k_n,\tau_{n}),\sigma_n)=\sigma_n^2 \cdot R(\Theta(k_n,\mu_{n}),1)$  (see Section \ref{sec:scaling}), 
it is sufficient to obtain lower bound for $R(\Theta(k_n,\mu_{n}),1)$, i.e., the minimax risk under Gaussian sequence model: $y_i=\theta_i+z_i, 1\leq i \leq n$, with $z_i\overset{i.i.d.}{\sim} \mathcal{N}(0,1)$. A general strategy for finding lower bounds of minimax risk in sparse Gaussian sequence model, is to employ i.i.d. univariate spike prior as the (asymptotically) least favorable prior. Although such product prior served as a suitable tool to establish a sharp lower bound for proving Theorem \ref{thm::Thm8.21-Johnstone}, we have since recognized its inadequacy in providing a sufficiently sharp lower bound for obtaining the second-order approximation of the minimax risk. Hence, in order to use Theorem \ref{thm::minimax-lower-bound}, we utilize the family of \textit{independent block priors} \cite{donoho1997universal, johnstone19}. The specific independent block prior $\pi^{IB}(\theta)$ on $\Theta(k_n,\mu_{n})$ for our problem is constructed in the following steps:
\begin{enumerate}
\item Divide $\theta\in \mathbb{R}^n$ into $k_n$ disjoint blocks of dimension $m=n/k_n$\footnote{For simplicity, here we assume $n/k_n$ is an integer. In the case when it is not, we can slightly adjust the block size to obtain the same lower bound.}: \[\theta=(\theta^{(1)},\theta^{(2)},\ldots, \theta^{(k_n)}).\]
\item Sample each block $\theta^{(j)}\in \mathbb{R}^m$ from the symmetric spike prior $\pi_S^{\mu, m}$: for $ 1\leq i \leq m$, 
\[\pi_S^{\mu, m}\Big(\theta^{(j)}=\mu e_i\Big)=\pi_S^{\mu, m}\Big(\theta^{(j)}=-\mu e_i\Big)=\frac{1}{2m},\]
where $\mu \in (0, \mu_n]$ is a location parameter.
\item Combine independent blocks: \[\pi^{IB}(\theta)=\prod_{j=1}^{k_n}\pi_S^{\mu, m}(\theta^{(j)})\].
\end{enumerate}
In other words, the independent block prior $\pi^{IB}$ picks a single spike (from $2m$ possible locations) in each of $k_n$ non-overlapping blocks of $\theta$, with the spike location within each block being independent and uniform. As is clear from the construction, $\supp (\pi^{IB}) \subseteq \Theta(k_n,\mu_n)$ so that 
\begin{align}
\label{reduce:1d}
R(\Theta(k_n,\mu_{n}),1) \geq B(\pi^{IB})=k_n \cdot B(\pi_S^{\mu,m}).
\end{align}
Here, the last equation holds because when the prior has block independence and the loss function is additive, the Bayes risk can be decomposed into the sum of Bayes risk of prior for each block (see Chapter 4.5 in \cite{johnstone19}). 

As a result, the main goal of the rest of this section is to obtain a sharp lower bound ({\it up to the second order}) for the Bayes risk $B(\pi_S^{\mu,m})$, i.e., the risk of the posterior mean under the spike prior $\pi_S^{\mu,m}$. The following two lemmas are instrumental in obtaining such a sharp lower bound.

\begin{lemma}
\label{lem::spike-bayes-risk-decompose}
Consider the Gaussian sequence model: $y_i=\theta_i+z_i, 1\leq i \leq m$, with $z_i\overset{i.i.d.}{\sim} \mathcal{N}(0,1)$. The Bayes risk of $\pi_S^{\mu,m}$ takes the form
\begin{align*}
B(\pi_S^{\mu,m})=\mathbb{E}_{\mu e_1}(\hat{\theta}_1-\mu)^2+(m-1)\mathbb{E}_{\mu e_2}\hat{\theta}^2_1,
\end{align*}
where $\mathbb{E}_{\mu e_1}(\cdot)$ is taken with respect to $y\sim \mathcal{N}(\mu e_1,I)$ and $\mathbb{E}_{\mu e_2}(\cdot)$ for $y\sim \mathcal{N}(\mu e_2,I)$; $\hat{\theta}_1$ is the posterior mean for the first coordinate having the expression
\[
\hat{\theta}_1=\frac{\mu(e^{\mu y_1}-e^{-\mu y_1})}{\sum_{i=1}^m(e^{\mu y_i}+e^{-\mu y_i})}.
\]
\end{lemma}

\begin{proof}
Let the posterior mean be $\hat{\theta}=\mathbb{E}[\theta|y]$. Using Bayes' Theorem we obtain

\begin{eqnarray*}
\hat{\theta}_1&=&\mu \mathbb{P}(\theta=\mu e_1\mid y)-\mu \mathbb{P}(\theta=-\mu e_1\mid y) \\
&=&\frac{\mu [\mathbb{P}(y\mid \theta=\mu e_1)- \mathbb{P}(y\mid \theta=-\mu e_1)]}{\sum_{i=1}^m[\mathbb{P}(y\mid \theta=\mu e_i)+\mathbb{P}(y\mid \theta=-\mu e_i)]} \\
&=&\frac{\mu(e^{\mu y_1}-e^{-\mu y_1})}{\sum_{i=1}^m(e^{\mu y_i}+e^{-\mu y_i})}.
\end{eqnarray*}
Moreover, since both $\theta_i$'s (under the prior) and $z_i$'s are exchangeable, the pairs $\{(\hat{\theta}_i,\theta_i)\}_{i=1}^m$ are exchangeable as well. As a result,
\begin{align*}
     & B(\pi_S^{\mu,m}) = \E \sum_{i=1}^{m} (\hat{\theta}_{i} - \theta_{i})^{2} = m \E (\hat{\theta}_{1} - \theta_{1})^{2} \nonumber\\
    =& m \left[\frac{1}{2m} \E_{\mu e_{1}} (\hat{\theta}_{1} - \mu)^{2} + \frac{1}{2m} \E_{-\mu e_{1}} (\hat{\theta}_{1} + \mu)^{2} \right. \nonumber \\
    & \left. + \frac{1}{2m} \sum_{i=2}^{m} \left( \E_{\mu e_{i}} \hat{\theta}_{1}^{2} + \E_{-\mu e_{i}} \hat{\theta}_{1}^{2} \right) \right] \nonumber \\
    =& \frac{1}{2} \left[ \E_{\mu e_{1}} (\hat{\theta}_{1} - \mu)^{2} + \E_{-\mu e_{1}} (\hat{\theta}_{1} + \mu)^{2} \right] \nonumber \\
    & + \frac{1}{2} \sum_{i=2}^{m} \left[ \E_{\mu e_{i}} \hat{\theta}_{1}^{2} + \E_{-\mu e_{i}} \hat{\theta}_{1}^{2} \right] \\
    =& \E_{\mu e_{1}} (\hat{\theta}_{1} - \mu)^{2} + (m-1) \E_{\mu e_{2}} \hat{\theta}_{1}^{2},
\end{align*}
where in the last equation we have used the facts that the distribution of $\hat{\theta}_1$ under $\theta=\mu e_1$ equals that of $-\hat{\theta}_1$ under $\theta=-\mu e_1$, and $\hat{\theta}_1$ has the same distribution when $\theta=\pm \mu e_i, i=2,\ldots, m$.
\end{proof}

\begin{lemma}\label{lem::spike-bayes-risk-regime-1}
As $\mu \rightarrow 0, m\rightarrow \infty$, The Bayes risk of $\pi_S^{\mu,m}$ has the lower bound
\begin{equation*}
   B(\pi_S^{\mu,m}) \geq \mu^{2} -\frac{\mu^{4}}{m} (1+o(1)).
\end{equation*}
\end{lemma}

\begin{proof}
Denote $p_m=\frac{e^{\mu y_1}-e^{-\mu y_1}}{\sum_{i=1}^m(e^{\mu y_i}+e^{-\mu y_i})}$. According to Lemma \ref{lem::spike-bayes-risk-decompose}, the Bayes risk can be lower bounded in the following way:
\begin{align*}
B(\pi_S^{\mu,m})\geq \mu^2\cdot \Big[1-2\E_{\mu e_{1}}p_m+(m-1)\E_{\mu e_{2}}p_m^2\Big].
\end{align*}
It is thus sufficient to prove that $\E_{\mu e_{1}} p_m \leq \frac{\mu^{2}}{m} (1+o(1))$ and $(m-1) \E_{\mu e_{2}} p_{m}^{2} \geq \frac{\mu^{2}}{m}(1+o(1))$. We first prove the former one. We have

\begin{align*}
    \E_{\mu e_{1}} p_{m} & = \E\left[ \frac{e^{\mu (\mu +z_{1})} - e^{-\mu (\mu +z_{1})}}{\sum_{j\neq 1} [e^{\mu z_{j}} + e^{-\mu z_{j}}] + e^{\mu (\mu +z_{1})} + e^{-\mu (\mu +z_{1})} } \right] \\
    & = \E\left[ \frac{(e^{\mu ^{2}} - 1)e^{\mu z_{1}}}{\sum_{j\neq 1} [e^{\mu z_{j}} + e^{-\mu z_{j}}] + e^{\mu (\mu +z_{1})} + e^{-\mu (\mu +z_{1})}} \right] \\
    & + \E\left[ \frac{(1-e^{-\mu ^{2}})e^{-\mu z_{1}}}{\sum_{j\neq 1} [e^{\mu z_{j}} + e^{-\mu z_{j}}] + e^{\mu (\mu +z_{1})} + e^{-\mu (\mu +z_{1})}} \right] \\
    & + \E\left[ \frac{e^{\mu z_{1}} - e^{-\mu z_{1}}}{\sum_{j\neq 1} [e^{\mu z_{j}} + e^{-\mu z_{j}}] + e^{\mu (\mu +z_{1})} + e^{-\mu (\mu +z_{1})}} \right] \\
    & =: E_{1} + E_{2} + E_{3}.
\end{align*}
We study $E_1, E_2$ and $E_3$ separately. For $E_{1}$, given that the numerator inside the expectation is positive, we apply the basic inequality $a+b\geq 2\sqrt{ab},\forall a,b\geq 0$ to the denominator to obtain  
\begin{equation*}
    E_{1}
    \leq \frac{e^{\mu^{2}} - 1}{2m}  \E e^{\mu z_{1}} =\frac{\mu^2}{2m}\cdot \frac{(e^{\mu^2}-1)e^{\mu^2/2}}{\mu^2} =\frac{\mu^2(1+o(1))}{2m}.
\end{equation*}
Similarly, for $E_{2}$ we have
\begin{eqnarray*}
     E_{2}
    &\leq& \frac{1-e^{-\mu^{2}}}{2m}  \E e^{-\mu z_{1}} \\
    &=&\frac{\mu^2}{2m}\cdot \frac{(1-e^{-\mu^2})e^{\mu^2/2}}{\mu^2} =\frac{\mu^2(1+o(1))}{2m}.
\end{eqnarray*}
To study $E_{3}$, define 
\begin{align*}
    & A := \sum_{j\neq 1} [e^{\mu z_{j}} + e^{-\mu z_{j}}] + e^{\mu (\mu +z_{1})} + e^{-\mu (\mu +z_{1})}, \\
    & B := \sum_{j\neq 1} [e^{\mu z_{j}} + e^{-\mu z_{j}}] + e^{\mu (\mu -z_{1})} + e^{-\mu (\mu -z_{1})}.
\end{align*}
The basic inequality $a+b\geq 2\sqrt{ab}$ implies that $A\geq 2m, B\geq 2m$. This together with the symmetry of standard normal distribution yields
\begin{align*}
    E_3&=\E\frac{e^{\mu z_1}}{A}-\mathbb{E}\frac{e^{-\mu z_1}}{A}=\E\frac{e^{\mu z_1}}{A}-\E\frac{e^{\mu z_1}}{B} \\
    &=\E \left[ \frac{(e^{\mu^{2}} - e^{-\mu^{2}}) (e^{-\mu z_{1}} - e^{\mu z_{1}}) e^{\mu z_{1}} }{AB} \right] \\
   &\leq \E \left[ \frac{(e^{\mu^{2}} - e^{-\mu^{2}}) (1 - e^{2\mu z_{1}}) I_{(z_{1}\leq 0)} }{AB} \right] \\
   &\leq \frac{e^{\mu^{2}} - e^{-\mu^{2}}}{4m^2} \E \left[ (1-e^{2\mu z_{1}}) \mathbbm{1}_{(z_{1}\leq 0)} \right]=O\Big(\frac{\mu^2}{m^2}\Big)
\end{align*}

It remains to prove $(m-1) \E_{\mu e_{2}} p_{m}^{2} \geq \frac{\mu^{2}}{m}(1+o(1))$. Denote
\begin{equation*}
    C:= \Big[ e^{\mu b} + e^{-\mu b} + 2(m-2) e^{\frac{\mu ^{2}}{2}} + e^{\frac{3}{2}\mu ^{2}} + e^{-\frac{\mu ^{2}}{2}}\Big]^{2},
\end{equation*}
where $b>0$ is a scalar to be specified later. Then
\begin{eqnarray*}
     && \E_{\mu e_{2}} p_m^{2} \\
     &=&   \E \Bigg[ \frac{(e^{\mu z_{1}} - e^{-\mu z_{1}})^{2}}{\big [\sum_{j\neq 2} (e^{\mu z_{j}} + e^{-\mu z_{j}}) + e^{\mu (\mu +z_{2})} + e^{-\mu (\mu +z_{2})}\big ]^{2}} \Bigg] \\
    & \labelrel\geq{eq::jensen-conditional-on-z1}&   \E \Bigg[ \frac{(e^{\mu z_{1}} - e^{-\mu z_{1}})^{2}}{[ e^{\mu z_{1}} + e^{-\mu z_{1}} + 2(m-2)e^{\frac{\mu ^{2}}{2}} + e^{\frac{3}{2} \mu ^{2}} + e^{-\frac{\mu ^{2}}{2}} ]^{2}} \Bigg] \\
    & \geq &  \E \left[\frac{(e^{\mu z_{1}} - e^{-\mu z_{1}})^{2} I_{(|z_{1}|\leq b)}}{ [e^{\mu b} + e^{-\mu b} + 2(m-2) e^{\frac{\mu ^{2}}{2}} + e^{\frac{3}{2}\mu ^{2}} + e^{-\frac{\mu ^{2}}{2}} ]^{2} } \right]\\
    & = & \frac{2}{C} \left[  \E e^{2\mu z_{1}} I_{(|z_{1}|\leq b)} -  \mathbb{P}(|z_{1}|\leq b) \right] \\
     & = & \frac{2}{C} \left[ e^{2\mu ^{2}} \int_{-b-2\mu }^{b-2\mu } \phi(z) d z - \int_{-b}^{b} \phi(z) d z  \right] \\
    & = & \frac{2}{C} \left[  (e^{2\mu ^{2}} - 1)\int_{-b-2\mu }^{b-2\mu }\phi(z) d z \right. \\
    && \left. - \int_{b-2\mu  }^{b}\phi(z) d z + \int_{-b-2\mu }^{-b}\phi(z) d z \right]  \\
    & \labelrel={eq::select-b} &~ \frac{2}{C} \left[ 2\mu ^{2} (1+o(1)) + o(\mu ^{2}) + o(\mu ^{2}) \right] \\
    & \labelrel\geq{eq::upper-bound-C} & ~\frac{2}{4m^{2} e^{2\sqrt{\mu }}} \cdot 2\mu ^{2} (1+o(1))
    = \frac{\mu ^{2}}{m^{2}} (1+o(1)).
\end{eqnarray*}
Inequality \eqref{eq::jensen-conditional-on-z1} is obtained by conditioning on $z_{1}$ and applying Jensen's inequality on the convex function $1/(x+c)^{2}$ for $x>0$. Equality \eqref{eq::select-b} holds by setting $b = 1/\sqrt{\mu}$, for the purpose of matching the asymptotic order $\frac{\mu^{2}}{m}(1+o(1))$. Finally, inequality \eqref{eq::upper-bound-C} is because $C \leq 4m^{2} e^{2\sqrt{\mu}}$ when $\mu$ is sufficiently small.
\end{proof}

We are in the position to derive the matching lower bound for the minimax risk. Recall that in the block prior we have $m=n/k_n,\mu\in (0,\mu_n]$. Set $\mu=\mu_n$. The assumption $\epsilon_n=k_n/n\rightarrow 0,\mu_n\rightarrow 0$ guarantees that the condition $m\rightarrow \infty,\mu\rightarrow 0 $ in Lemma \ref{lem::spike-bayes-risk-regime-1} is satisfied. 
We therefore combine Lemma \ref{lem::spike-bayes-risk-regime-1} and \eqref{reduce:1d} to obtain
\begin{align*}
 R(\Theta(k_n,\tau_{n}),\sigma_n)&=\sigma_n^2 \cdot R(\Theta(k_n,\mu_{n}),1) \geq k_n\sigma_n^2 \cdot B(\pi_S^{\mu,m}) \\
 &\geq k_n\sigma_n^2\cdot \Big[\mu_n^2-\frac{\mu_n^4k_n}{n}(1+o(1))\Big]\\
&=n\sigma^2_n\cdot\Big(\epsilon_n\mu_n^2-\epsilon_n^2\mu_n^4(1+o(1))\Big).
\end{align*}

\subsection{Proof of Proposition \ref{lem::lasso-risk-regime-1}}\label{app:proof_lasso_regime_1}
% This section contains the proof of Proposition \ref{lem::lasso-risk-regime-1}. The main idea of this proof is shared by the proof of Propositions \ref{lem::lasso-risk-regime-2} and \ref{lem::lasso-risk-regime-4} in the Supplementary Material. 

Define the supremum risk of optimally tuned soft thresholding estimator as
\[
R_s (\Theta (k_n, \tau_n), \sigma_n) = \inf_{\lambda>0} \sup_{\theta \in \Theta (k_n, \tau_n)} \E_{\theta} \norm{\etahatS(y,\lambda) - \theta}_{2}^{2},
\]
where $y_i = \theta_i + \sigma_n z_i$, with $z_i \overset{i.i.d.}{\sim}\mathcal{N}(0,1)$. It is straightforward to verify that
\begin{equation}
   R_s (\Theta (k_n, \tau_n), \sigma_n) = \sigma_n^2\cdot  R_s (\Theta (k_n, \mu_n), 1). \label{eq::soft-thresholding-risk-scalability}
\end{equation}
Hence, without loss of generality, in the rest of the proof we will assume that $\sigma_n =1$. 

Since $\hat{\eta}_S(y,\lambda)$ is the special case of $\hat{\eta}_E(y,\lambda,\gamma)$ with $\gamma=0$, the supremum risk result stated in Equation \eqref{eq::maximum-elastic-net-risk} for $\hat{\eta}_E(y,\lambda,\gamma)$ applies to $\hat{\eta}_S(y,\lambda)$ as well. It shows that the supremum risk of $\etahatS(y,\lambda)$ is attained on a particular boundary of the parameter space:
\begin{eqnarray}
     && \sup_{\theta \in \Theta(k_n,\mu_{n})} \E \sum_{i=1}^{n} |\etahatS(y_{i},\lambda) - \theta_{i}|_{2}^{2} \nonumber \\
     &=& (n-k_n) r_{S}(\lambda,0) + k_n r_{S}(\lambda,\mu_{n}) \nonumber\\
    & =&  n\left[ (1-\epsilon_{n})r_{S}(\lambda,0) + \epsilon_{n}r_{S}(\lambda,\mu_{n}) \right], \label{eq::soft-thresholding-sup-risk}
\end{eqnarray}
with $\epsilon_{n} = k_n/n$ and $r_{S}(\lambda,\mu)$ defined as
\begin{equation}
    r_{S}(\lambda,\mu) = \E(\etahatS(\mu+z,\lambda)-\mu)^{2},~~z\sim \mathcal{N}(0,1). \label{eq::soft-thresholding-risk-definition0}
\end{equation}
To prove Proposition \ref{lem::lasso-risk-regime-1}, we need to find the optimal $\lambda$ that minimizes the supremum risk in \eqref{eq::soft-thresholding-sup-risk}, or equivalently, the function
\begin{equation}
    F(\lambda) := (1-\epsilon_n) r_{S}(\lambda, 0) + \epsilon_n r_{S}(\lambda, \mu_n). \label{eq::soft-thresholding-risk-function}
\end{equation}

\begin{lemma}
\label{pre:order:tuning}
Denote the optimal tuning by $\lambdas=\argmin_{\lambda \geq 0} F(\lambda)$. It holds that 
\begin{equation}\label{eq:label:optlambda-regime_1}
    \log 2 \epsilon_n^{-1} + \frac{\mu_{n}^{2}}{2} -2 \log \log \frac{2}{\epsilon_{n}} < \lambdas \mu_n <  \log{2\epsilon_n^{-1}} + \frac{\mu_n^2}{2},
\end{equation}
when $n$ is sufficiently large.
\end{lemma}

\begin{proof}
Using integration by parts, we first obtain a more explicit expression for $F(\lambda)$:
\begin{align}
    F(\lambda) =& (1-\epsilon_n)\E\hat{\eta}^2_S(z,\lambda) + \epsilon_n \mu_{n}^{2} \nonumber \\
    -& 2\epsilon_{n} \mu_{n} \E \etahatS(\mu_{n}+z, \lambda)
    + \epsilon_{n} \E\etahatS^{2}(\mu_{n}+z, \lambda), \label{eq::soft-thresholding-risk-of-lambda}
\end{align}
where the three expectations take the form
\begin{align}
   \E\hat{\eta}^2_S(z,\lambda) =  &2(1+\lambda^{2}) \int_{\lambda}^{\infty} \phi(z) d z - 2 \lambda \phi(\lambda)   \label{risk:at:zero} \\
   \E \etahatS(\mu_{n}+z, \lambda) =& \phi(\lambda-\mu_{n}) + (\mu_{n}-\lambda) \int_{\lambda-\mu_{n}}^{\infty} \phi(z) d z \nonumber \\
   -& \phi(\lambda+\mu_{n}) + (\mu_{n}+\lambda) \int_{\lambda + \mu_{n}}^{\infty} \phi(z) d z\label{eq::soft-thresholding-first-moment} \\
       \E \etahatS^{2}(\mu_{n}+z, \lambda)  =& \bigg[ \left( 1+(\lambda-\mu_{n})^{2}\right) \int_{\lambda-\mu_{n}}^{\infty} \phi(z) d z \nonumber \\
       -& (\lambda - \mu_{n}) \phi(\lambda-\mu_{n})  \bigg] \nonumber\\
       +& \bigg[ \left( 1+(\lambda+\mu_{n})^{2} \right) \int_{\lambda+\mu_{n}}^{\infty} \phi(z) d z \nonumber \\
       -& (\lambda + \mu_{n}) \phi(\lambda+\mu_{n})  \bigg]. \label{eq::soft-thresholding-second-moment}
\end{align} 
Therefore, $F(\lambda)$ is a differentiable function of $\lambda$, and as long as the infimum of $F(\lambda)$ is not achieved at $0$ or $+\infty$, $\lambdas$ will satisfy $F'(\lambdas)=0$. From Equations \eqref{eq::soft-thresholding-risk-of-lambda}-\eqref{eq::soft-thresholding-second-moment}, it is direct to compute $F(0) = 1>F(+\infty)=\epsilon_{n}\mu_{n}^{2}$ for large $n$. Moreover, as we will show in the end of the proof, $F(\lambda)$ is increasing when $\lambda$ is above a threshold. Hence, the optimal tuning $\lambdas \in (0,\infty)$, and we can characterize it through the derivative equation:
\begin{align}
    0 & = F{'}(\lambdas)  = (1-\epsilon_{n}) \left[ 4\lambdas \int_{\lambdas}^{\infty} \phi(z) d z - 4 \phi(\lambdas) \right] \nonumber\\
    & + \epsilon_{n} \Bigg[ -2\phi(\lambdas-\mu_{n}) - 2\phi(\lambdas + \mu_{n})  \nonumber \\
    & + 2\lambdas \left( \int_{\lambdas-\mu_{n}}^{\infty} \phi(z) d z + \int_{\lambda^* + \mu_{n}}^{\infty} \phi(z) d z \right) \Bigg]. \label{eq::soft-thresholding-zero-derivative}
\end{align}

First, we show that $\lambdas \rightarrow \infty$. Suppose this is not true. Then $\lambdas\leq C$ for some constant $C>0$ (take a subsequence if necessary). From  \eqref{eq::soft-thresholding-risk-of-lambda}, we have
\begin{align*}
    F(\lambdas) \geq &  (1-\epsilon_{n}) r_{S}(C,0) \nonumber \\
    =& 2(1-\epsilon_{n}) \left[(1+C^{2})\int_{C}^{\infty}\phi(z)dz - C\phi(C)\right] \nonumber \\
    >& \epsilon_n\mu_n^2=F(+\infty),
\end{align*}
when $n$ is large. This contradicts with the optimality of $\lambdas$.

Second, we prove that $\lambdas \mu_n \rightarrow \infty$. Otherwise, $\lambda_*\mu_n=O(1)$ (take a subsequence if necessary). We will show that it leads to a contradiction in \eqref{eq::soft-thresholding-zero-derivative}. Using the Gaussian tail bound $\int_{t}^{\infty}\phi(z)dz=(\frac{1}{t}-\frac{1+o(1)}{t^3})\phi(t)$ as $t\rightarrow \infty$ from Section \ref{ssec:GTB}, since $\lambdas \rightarrow \infty,\mu_n\rightarrow 0, \lambda_*\mu_n=O(1)$, we obtain 
\begin{align}
-\lambdas \int_{\lambdas}^{\infty} \phi(z) d z + \phi(\lambdas) &= (1+o(1))\cdot  \lambdas^{-2}\phi(\lambdas), \label{lasso:regime_1-lo2}\\
-\phi(\lambdas+\mu_{n}) + \lambdas\int_{\lambdas+\mu_{n}}^{\infty} \phi(z) d z&=O(\lambdas^{-2}\phi(\lambdas)), \label{eq:limit:case1:last2} \\
-\phi(\lambdas-\mu_{n}) + \lambdas\int_{\lambdas-\mu_{n}}^{\infty} \phi(z) d z&=O(\lambdas^{-2}\phi(\lambdas)). \label{eq:limit:case1:last1}
\end{align}
Given that $\epsilon_n\rightarrow 0$, combining the above results with $\eqref{eq::soft-thresholding-zero-derivative}$ implies that $0=F'(\lambdas)\cdot \lambdas^{2}\phi^{-1}(\lambdas)=-4+o(1)$, which is a contradiction. 

Third, we show that $\lambdas\mu_{n} < \log \frac{2}{\epsilon_{n}} + \frac{\mu_{n}^{2}}{2}$ for large $n$. Now that we have proved $\lambdas \mu_n \rightarrow \infty$, results in \eqref{eq:limit:case1:last2}-\eqref{eq:limit:case1:last1} can be strengthened:
\begin{align}
    & -\phi(\lambdas+\mu_{n}) + \lambdas\int_{\lambdas+\mu_{n}}^{\infty} \phi(z) d z =o(\mu_n\lambdas^{-1}\phi(\lambdas-\mu_n)), \label{eq:limit:case1:last2:new} \\
& -\phi(\lambdas-\mu_{n}) + \lambdas\int_{\lambdas-\mu_{n}}^{\infty} \phi(z) d z
=(1+o(1)) \nonumber \\
& \cdot \mu_n\lambdas^{-1} \phi(\lambdas-\mu_n). \label{eq:limit:case1:last1:new}
\end{align}
Plugging \eqref{lasso:regime_1-lo2} and \eqref{eq:limit:case1:last2:new}-\eqref{eq:limit:case1:last1:new} into \eqref{eq::soft-thresholding-zero-derivative} gives $(4+o(1))\cdot \lambdas^{-2}\phi(\lambdas)=(2+o(1))\cdot \epsilon_n\mu_n\lambdas^{-1}\phi(\lambdas-\mu_n)$,
which can be further simplified as 
\begin{align}
\label{a:key:master}
    2+o(1)=\epsilon_n\mu_n \lambdas \exp(\lambdas \mu_n-\mu_n^2/2).
\end{align}
The above equation implies that $\lambdas\mu_{n} < \log \frac{2}{\epsilon_{n}} + \frac{\mu_{n}^{2}}{2}$ for large $n$. Otherwise, the right-hand side will be no smaller than $2\mu_n\lambda_n\rightarrow \infty$ contradicting with the left-hand side term. 

Fourth, we prove that $\lambdas \mu_{n} > \log \frac{2}{\epsilon_{n}} + \frac{\mu_{n}^{2}}{2} - 2 \log \log \frac{2}{\epsilon_{n}}$ when $n$ is large. Otherwise, suppose $\lambdas \mu_{n} \leq \log \frac{2}{\epsilon_{n}} + \frac{\mu_{n}^{2}}{2} - 2 \log \log \frac{2}{\epsilon_{n}}$ (take a subsequence if necessary). This leads to
\begin{align*}
& 0\leq \epsilon_n\mu_n \lambdas \exp(\lambdas \mu_n-\mu_n^2/2)\leq \frac{2\mu_n\lambdas}{(\log \frac{2}{\epsilon_n})^2} \nonumber \\
<& \frac{2\log \frac{2}{2\epsilon_{n}} + \mu_{n}^{2}}{(\log \frac{2}{\epsilon_n})^2}=o(1),
\end{align*}
where we have used the upper bound $\lambdas \mu_{n}<\log \frac{2}{\epsilon_{n}}+ \frac{\mu_{n}^{2}}{2}$ derived earlier. The obtained result contradicts with \eqref{a:key:master}. 

Finally, as mentioned earlier in the proof, we need to show that $\lambdas \neq +\infty$ for large $n$. It is sufficient to prove that $F'(\lambda)>0, \forall \lambda \in [\frac{2}{\mu_n}\log\frac{1}{\epsilon_n},\infty)$, when $n$ is large. To this end, using the Gaussian tail bound $\int_{t}^{\infty}\phi(z)dz\geq (\frac{1}{t}-\frac{1}{t^3})\phi(t), \forall t>0$ and the derivative expression \eqref{eq::soft-thresholding-zero-derivative}, we have
\begin{alignat*}{2}
F'(\lambda)&\geq \frac{\phi(\lambda)}{\lambda^2}\cdot \Big[&& -4+4\epsilon_n +\frac{\mu_n(\lambda-\mu_n)^2-\lambda}{(\lambda-\mu_n)^3\lambda^{-2}}2\epsilon_ne^{\lambda\mu_n-\mu_n^2/2} \nonumber \\
& &&+\frac{-\mu_n(\lambda+\mu_n)^2-\lambda}{(\lambda+\mu_n)^3\lambda^{-2}}2\epsilon_ne^{-\lambda\mu_n-\mu_n^2/2}\Big] \\
&\geq \frac{\phi(\lambda)}{\lambda^2}\cdot \Big[&& -4 +4\epsilon_n +(2+o(1))\cdot \epsilon_n e^{-\mu_n^2/2}\lambda \mu_ne^{\lambda \mu_n}\Big],
\end{alignat*}
where we used that $\lambda \geq \frac{2}{\mu_{n}}\log\frac{1}{\epsilon_{n}}$ implies $\lambda\mu_{n}=\omega(1)$. Note that the above asymptotic notion $o(\cdot)$ is uniform for all $\lambda \geq \frac{2}{\mu_n}\log\frac{1}{\epsilon_n}$ when $n$ is large. Since $\lambda \mu_n\geq 2\log \frac{1}{\epsilon_n}$, we can easily continue from the above inequality to obtain $F'(\lambda)>0$ for sufficiently large $n$.
\end{proof}

The next lemma turns $F(\lambdas)$ into a form that is more amenable to asymptotic analysis.

\begin{lemma}
\label{refine:order}
Define
\begin{eqnarray*}
\mathcal{A}&=&-\mu_n(\lambdas-\mu_n)+1+\frac{\mu_n(\lambdas-\mu_n)^3e^{-2\lambdas\mu_n}}{(\lambdas+\mu_n)^2} \\
&+& \frac{(\lambdas-\mu_n)^3e^{-2\lambdas\mu_n}}{(\lambdas+\mu_n)^3}+O\Big(\frac{\mu_n}{\lambdas}\Big), \\
\mathcal{B}&=&\mu_n(\lambdas-\mu_n)^2-\lambdas+(3+o(1))\lambdas^{-1} \\
&+& \frac{[-\mu_n\lambdas^2-\lambdas(1+2\mu_n^2(1+o(1)))]}{(\lambdas+\mu_n)^3} \\
&\cdot& (\lambdas-\mu_n)^3e^{-2\lambdas\mu_n}.
\end{eqnarray*}
As $\epsilon_n\rightarrow 0, \mu_n\rightarrow 0$, it holds that
\begin{eqnarray*}
F(\lambdas)&=&\epsilon_n\mu_n^2+\frac{4(1-\epsilon_n)\phi(\lambdas)}{\lambdas^3}\cdot \\
&& \Big[1-6\lambdas^{-2}+O(\lambdas^{-4})+\Big(\lambdas-\frac{3+o(1)}{\lambdas}\Big)\frac{\mathcal{A}}{\mathcal{B}}\Big].
\end{eqnarray*}
\end{lemma}

\begin{proof}
We use Gaussian tail bounds to evaluate the three expectations \eqref{risk:at:zero}-\eqref{eq::soft-thresholding-second-moment} in the expression of $F(\lambdas)$ in \eqref{eq::soft-thresholding-risk-of-lambda}. Note that as shown in Lemma \ref{pre:order:tuning}, $\lambdas\mu_n=\Uptheta(\log2\epsilon_n^{-1})$. The first expectation is
\begin{equation}\label{eq::soft-threshold-regime_1-risk-at-zero}
    \E\hat{\eta}^2_S(z,\lambdas) = 2 \phi(\lambdas)\left[2\lambdas^{-3} -12 \lambdas^{-5} + O(\lambdas^{-7})\right].
\end{equation}
Regarding the second one, we obtain
\begin{eqnarray*}
    && \phi(\lambdas-\mu_{n}) - (\lambdas-\mu_{n})\int_{\lambdas-\mu_{n}}^{\infty} \phi(z) d z \\
    &=& \left[ (\lambdas-\mu_{n})^{-2} + O\left( (\lambdas-\mu_{n})^{-4} \right) \right] \phi(\lambdas-\mu_{n}),
\end{eqnarray*}
% \begin{eqnarray*}
%     &&\phi(\lambdas-\mu_{n}) - (\lambdas-\mu_{n})\int_{\lambdas-\mu_{n}}^{\infty} \phi(z) d z \\
%     &=& \left[ (\lambdas-\mu_{n})^{-2} + O\left( (\lambdas-\mu_{n})^{-4} \right) \right] \phi(\lambdas-\mu_{n}),
% \end{eqnarray*}
and 
\begin{eqnarray*}
    && \phi(\lambdas+\mu_{n}) - (\lambdas+\mu_{n})\int_{\lambdas+\mu_{n}}^{\infty} \phi(z) d z \\
    &=& \left[ (\lambdas+\mu_{n})^{-2}e^{-2\lambdas\mu_{n}} + O\left((\lambdas+\mu_{n})^{-4}e^{-2\lambdas\mu_{n}}\right) \right]\cdot \\
    && \phi(\lambdas-\mu_{n}).
\end{eqnarray*}
Therefore,
\begin{align}\label{eq::soft-threshold-regime_1-risk-first-moment}
    \E \etahatS(\mu_{n}+z, \lambdas) =& \left[ (\lambdas-\mu_{n})^{-2} - (\lambdas+\mu_{n})^{-2}e^{-2\lambdas\mu_{n}} \right. \nonumber \\
    +& \left. 
    O\left( (\lambdas-\mu_{n})^{-4} \right) \right] \phi(\lambdas-\mu_{n}).
\end{align}
For the third expectation, we first have
\begin{align*}
   & \left(1+(\lambdas-\mu_{n})^{2}\right)\int_{\lambdas-\mu_{n}}^{\infty} \phi(z) d z - (\lambdas-\mu_{n}) \phi(\lambdas-\mu_{n}) \\
   =& \left[ 2(\lambdas-\mu_{n})^{-3} + O\left((\lambdas-\mu_{n})^{-5}\right) \right] \phi(\lambdas-\mu_{n}), \\
    &\left(1+(\lambdas+\mu_{n})^{2}\right)\int_{\lambdas+\mu_{n}}^{\infty} \phi(z) d z - (\lambdas+\mu_{n}) \phi(\lambdas+\mu_{n}) \\
    =& \left[ 2(\lambdas+\mu_{n})^{-3} + O\left((\lambdas+\mu_{n})^{-5}\right) \right]\phi(\lambdas+\mu_{n}).
\end{align*}
Thus,
\begin{align}\label{eq::soft-threshold-regime_1-risk-second-moment}
    \E\etahatS^{2}(\mu_{n}+z, \lambdas) =& \Big[ 2(\lambdas-\mu_{n})^{-3} + 2(\lambdas+\mu_{n})^{-3} e^{-2\lambdas\mu_{n}}  \nonumber \\
    +&  O\left((\lambdas-\mu_{n})^{-5}\right) \Big]\phi(\lambdas-\mu_{n}).
\end{align}
Plugging \eqref{eq::soft-threshold-regime_1-risk-at-zero}-\eqref{eq::soft-threshold-regime_1-risk-second-moment} into \eqref{eq::soft-thresholding-risk-of-lambda}, we have
\begin{align}
    F(\lambdas) =&~ 2(1-\epsilon_{n}) \phi(\lambdas)\left[2\lambdas^{-3} -12 \lambdas^{-5}\right.  \nonumber \\
    +&  \left. O(\lambdas^{-7})\right] + \epsilon_{n} \mu_{n}^{2} \nonumber \\
    -& 2\epsilon_{n}\mu_{n} \Big[ (\lambdas-\mu_{n})^{-2} - (\lambdas+\mu_{n})^{-2}e^{-2\lambdas\mu_{n}} \nonumber \\
    +& O\left( (\lambdas-\mu_{n})^{-4} \right) \Big] \phi(\lambdas-\mu_{n}) \nonumber \\
    +& \epsilon_{n} \Big[ 2(\lambdas-\mu_{n})^{-3} + 2(\lambdas+\mu_{n})^{-3} e^{-2\lambdas\mu_{n}}  \nonumber \\
    +&  O\left((\lambdas-\mu_{n})^{-5}\right) \Big]\phi(\lambdas-\mu_{n}) \nonumber \\
    =& ~\epsilon_{n}\mu_{n}^{2} + 2(1-\epsilon_{n}) \phi(\lambdas)\left[2\lambdas^{-3} -12 \lambdas^{-5} + O(\lambdas^{-7})\right] \nonumber \\
    +& \frac{2\epsilon_{n}\mathcal{A}\phi(\lambdas-\mu_{n})}{(\lambdas-\mu_n)^3}. \label{towards:one:step}
\end{align}
Next, we utilize the derivative equation \eqref{eq::soft-thresholding-zero-derivative} to further simplify \eqref{towards:one:step}. We first list the asymptotic approximations needed:
\begin{align*}
& -\lambdas \int_{\lambdas}^{\infty} \phi(z) d z + \phi(\lambdas)= \\
& (1-(3+o(1))\lambdas^{-2})\cdot  \lambdas^{-2}\phi(\lambdas), \\
 & -\phi(\lambdas+\mu_{n}) + \lambdas\int_{\lambdas+\mu_{n}}^{\infty} \phi(z) d z= \\
 & \frac{[-\mu_n\lambdas^2-\lambdas(1+2\mu_n^2(1+o(1)))]e^{-2\lambdas\mu_n}}{(\lambdas+\mu_n)^3}\phi(
 \lambdas-\mu_n), \\
&-\phi(\lambdas-\mu_{n}) + \lambdas\int_{\lambdas-\mu_{n}}^{\infty} \phi(z) d z= \\
& \frac{\mu_n(\lambdas-\mu_n)^2-\lambdas[1-(3+o(1))\lambdas^{-2}]}{(\lambdas-\mu_n)^3}\phi(\lambdas-\mu_n).
\end{align*}
Plugging them into \eqref{eq::soft-thresholding-zero-derivative} yields
\begin{align*}
4(1-\epsilon_n)\Big[\frac{1}{\lambdas^2}-\frac{3+o(1)}{\lambdas^4}\Big]\phi(\lambdas) =2\epsilon_n \frac{\mathcal{B}\phi(\lambdas-\mu_n)}{(\lambdas-\mu_n)^3}.
\end{align*}
Obtaining the expression for $\frac{\phi(\lambdas-\mu_n)}{(\lambdas-\mu_n)^3}$ from the above equation and plugging it into \eqref{towards:one:step} completes the proof. 
\end{proof}

We now apply Lemmas \ref{pre:order:tuning} and \ref{refine:order} to obtain the final form of $F(\lambdas)$. Referring to the expression of $F(\lambdas)$ in Lemma \ref{refine:order}, the key term to compute is $1+\Big(\lambdas-\frac{3+o(1)}{\lambdas}\Big)\frac{\mathcal{A}}{\mathcal{B}}$. Using the fact that $\lambdas\mu_n\rightarrow \infty$, some direct calculations enable us to obtain 
\begin{align*}
\Big(\lambdas-\frac{3+o(1)}{\lambdas}\Big)\mathcal{A}+\mathcal{B} & =(-1+o(1))\lambdas\mu_n^2, \\
\mathcal{B} & =\mu_n\lambdas^2(1+o(1)). 
\end{align*}
Therefore, the expression $F(\lambdas)$ in Lemma \ref{refine:order} can be simplified to 
\begin{align*}
F(\lambdas)=&\epsilon_n\mu_n^2+\frac{4(1-\epsilon_n)\phi(\lambdas)}{\lambdas^3} \\
\cdot& \Big[-6\lambdas^{-2}+O(\lambdas^{-4})-\frac{\mu_n}{\lambdas}(1+o(1))\Big] \\
=&\epsilon_n\mu_n^2-\frac{(4+o(1))\mu_n\phi(\lambdas)}{\lambdas^4}.
\end{align*}
Finally, Lemma \ref{pre:order:tuning} implies that $\lambdas=(1+o(1))\frac{\log \epsilon^{-1}_n}{\mu_n}$. Replacing $\lambdas$ by this rate in the above equation gives us the result in Proposition \ref{lem::lasso-risk-regime-1}.  
%----------------------
%----------------------
%----------------------

\subsection{Proof of Proposition \ref{lem::hardthreshold-risk-regime-1}}\label{proof:hardthreshold:reg1}
Define the supremum risk of optimally tuned hard thresholding estimator as
\[
R_H (\Theta (k_n, \tau_n), \sigma_n) = \inf_{\lambda>0} \sup_{\theta \in \Theta (k_n, \tau_n)} \E_{\theta} \norm{\etahatH(y,\lambda) - \theta}_{2}^{2},
\]
where $y_i = \theta_i + \sigma_n z_i$, with $z_i \overset{i.i.d.}{\sim}\mathcal{N}(0,1)$. It is straightforward to verify that
\begin{equation*}
   R_H (\Theta (k_n, \tau_n), \sigma_n) = \sigma_n^2\cdot  R_H (\Theta (k_n, \mu_n), 1). 
\end{equation*}
   Without loss of generality, let $\sigma_n = 1$ in the model. We first obtain the lower bound, by calculating the risk at a specific value of $\underline{\theta}$ such that $\underline{\theta}_{i} = \mu_{n}$ for $i \in \{ 1, 2, \ldots, k_n\}$ and $\underline{\theta}_i = 0$ for $i > k_n $:
   \begin{align}
   \label{lower:bound:first:hard}
    R_H (\Theta (k_n, \mu_n), 1) \geq \inf_{\lambda>0}\E_{\underline{\theta}}\| \etahatH(y,\lambda) -  \underline{\theta}\|_2^{2}.
   \end{align}
   Denote the one-dimensional risk:
\begin{align*}
    & r_{H}(\lambda,\mu) := \E \left( \etahatH(\mu + z,\lambda) - \mu \right)^{2}, \\
    & z \sim \mathcal{N}(0,1), \quad \forall \mu\in \mathbb{R},\lambda\geq 0.
\end{align*}
It is then direct to confirm that
\begin{align} \label{eq:hard:decomp}
    & \E_{\underline{\theta}}\| \etahatH(y,\lambda) -  \underline{\theta}\|_2^{2} | \nonumber \\
   &= n \Big[ (1-\epsilon_{n})r_{H}(\lambda, 0) + \epsilon_{n} r_{H}(\lambda,\mu_{n}) \Big].
\end{align} 
Let $\lambda_n^*$ be the optimal choice of $\lambda$ in $\E_{\underline{\theta}}\| \etahatH(y,\lambda) -  \underline{\theta}\|^{2}$ so that 
\[
\inf_{\lambda>0}\E_{\underline{\theta}}\| \etahatH(y,\lambda) -  \underline{\theta}\|_2^{2}=\E_{\underline{\theta}}\| \etahatH(y,\lambda_n^*) -  \underline{\theta}\|_2^{2}.
\]
To evaluate the lower bound in \eqref{lower:bound:first:hard}, we consider two scenarios for the optimal choice $\lambda_n^*$ and in each one we obtain a lower bound for $\E_{\underline{\theta}}\| \etahatH(y,\lambda^*_{n}) -  \underline{\theta}\|^{2}$. But before considering these two cases, we use the integration by part to find the following more explicit forms for $r_H(\lambda, 0)$ and $r_H(\lambda, \mu)$:
\begin{align}\label{eq:riskexp:hard}
r_H(\lambda, 0) &= 2 \int_{\lambda}^{\infty} z^2 \phi(z) dz = 2 \lambda \phi(\lambda)+2 (1- \Phi(\lambda)), \nonumber \\
r_H(\lambda, \mu) &= \mu^2 \int_{-\lambda- \mu}^{\lambda-\mu} \phi(z)dz + \int_{-\infty}^{-\lambda- \mu} z^2 \phi(z)dz \nonumber \nonumber \\
&+ \int_{\lambda-\mu}^{\infty} z^2 \phi(z)dz \nonumber \\
&= (\mu^{2} -1) \Big[ \Phi(\lambda - \mu) - \Phi(-\lambda -\mu) \Big] + 1 \nonumber \\
&+ (\lambda -\mu) \phi(\lambda -\mu) + (\lambda +\mu) \phi(\lambda+\mu),
\end{align}
where we recall that $\phi(\cdot)$ and $\Phi(\cdot)$ denote the density and CDF of $\calN(0,1)$ respectively. Now we consider two cases for the optimal choice $\lambda^*_n$ and in each case find a lower bound for the risk. 
\begin{itemize}
    \item \textbf{Case \rom{1} $\lambda^*_{n} = O(1)$:} we have $\lambda^*_{n}\leq c$ for some constant $c >0$. Hence, from \eqref{eq:hard:decomp} we obtain
    \begin{align*}
    &\inf_{\lambda>0} \E_{\underline{\theta}}\| \etahatH(y,\lambda) -  \underline{\theta}\|_2^{2} =\E_{\underline{\theta}}\| \etahatH(y,\lambda_n^*) -  \underline{\theta}\|_2^{2} \\
   % &= n \Big[ (1-\epsilon_{n})r_{H}(\lambda^*_{n}, 0) + \epsilon_{n} r_{H}(\lambda^*_{n},\mu_{n}) \Big] \nonumber \\
        \geq ~& n(1-\epsilon_{n}) r_{H}(\lambda^*_{n} ,0) \\
        =~ &  n (1-\epsilon_{n}) \Big[ 2\lambda^*_{n}\phi(\lambda^*_{n}) + 2(1-\Phi(\lambda^*_{n})) \Big] \\
        \geq ~& n(1-\epsilon_{n}) \Big[ 2(1-\Phi(\lambda^*_{n}))  \Big] \\
        \geq ~& n(1-\epsilon_{n}) \Big[ 2(1-\Phi(c))   \Big]\geq n\epsilon_n \mu_{n}^2.
    \end{align*}
    The last inequality is because $\epsilon_n \mu_{n}^2=o(1)$ and $(1-\epsilon_{n}) [ 2(1-\Phi(c))] = \Uptheta(1)$.  
    \item \textbf{Case \rom{2} $\lambda^*_{n}=\omega(1)$:} then $\lambda^*_n \rightarrow \infty$ as $n \rightarrow \infty$. From \eqref{eq:hard:decomp} and \eqref{eq:riskexp:hard}, we have
    \begin{align*}
       & \inf_{\lambda>0}\E_{\underline{\theta}}\| \etahatH(y,\lambda) -  \underline{\theta}\|_2^{2} \\
       =~&\E_{\underline{\theta}}\| \etahatH(y,\lambda_n^*) -  \underline{\theta}\|_2^{2}
       \geq k_n r_{H}(\lambda^*_{n},\mu_{n})\\
        =~& k_n(\mu_{n}^{2} -1) \bigg[ 1-\int_{\lambda^*_{n}-\mu_{n}}^{\infty}\phi(z)dz -\int_{\lambda^*_{n}+\mu_{n}}^{\infty}\phi(z)dz\bigg] + k_n \nonumber \\
        &+k_n(\lambda^*_{n} - \mu_{n}) \phi(\lambda^*_{n}-\mu_{n}) + k_n(\lambda^*_{n} + \mu_{n})\phi(\lambda^*_{n} + \mu_{n}) \\
        \overset{(a)}{=}~& k_n\mu_{n}^{2} +k_n(\lambda_n^*-\mu_n+o(1)) \phi(\lambda^*_{n} - \mu_{n})\\
        & +k_n (\lambda^*_{n} + \mu_{n}+o(1)) \phi(\lambda^*_{n} + \mu_{n})  \\
        \geq~& k_n \mu_{n}^{2}=n\epsilon_n\mu_n^2,
    \end{align*}
    where to obtain (a), we have used the Gaussian tail bound in Lemma \ref{lem::gaussian-tail-mills-ratio} under the scaling $\lambda^*_{n} \rightarrow \infty$ and $\mu_n \rightarrow 0$.
\end{itemize}
Note that since the two cases we have discussed above cover all the ranges of $\lambda_n^*$, we conclude that
\begin{eqnarray*}
R_H (\Theta (k_n, \mu_n), 1) \geq \inf_{\lambda>0}\E_{\underline{\theta}}\| \etahatH(y,\lambda) -  \underline{\theta}\|_2^{2} \geq n \epsilon_n \mu_n^2,
\end{eqnarray*}
for all sufficiently large $n$. To obtain the matching upper bound, we have
\begin{align}
&R_H (\Theta (k_n, \mu_n), 1) \nonumber \\
=& \inf_{\lambda>0} \sup_{\theta \in \Theta (k_n, \mu_n)} \E_{\theta} \norm{\etahatH(y,\lambda) - \theta}_{2}^{2} \nonumber \\
\leq &\lim_{\lambda\rightarrow \infty} \sup_{\theta \in \Theta (k_n, \mu_n)} \E_{\theta} \norm{\etahatH(y,\lambda) - \theta}_{2}^{2} \nonumber \nonumber \\
\leq & \lim_{\lambda\rightarrow \infty}  \Big(\sup_{\theta \in \Theta (k_n, \mu_n)}\E_{\theta}\|\etahatH(y,\lambda)\|_2^2 \nonumber\\
&+\sup_{\theta \in \Theta (k_n, \mu_n)}\E_{\theta}\langle -2\etahatH(y,\lambda),\theta \rangle+\sup_{\theta \in \Theta (k_n, \mu_n)}\|\theta\|_2^2\Big)  \nonumber \\
\leq& n\epsilon_n\mu_n^2+\lim_{\lambda\rightarrow \infty} \Big(\sup_{\theta \in \Theta (k_n, \mu_n)}\E_{\theta}\|\etahatH(y,\lambda)\|_2^2 \nonumber \\
&+2\sqrt{n\epsilon_n\mu_n^2}\sqrt{\sup_{\theta \in \Theta (k_n, \mu_n)}\E_{\theta}\|\etahatH(y,\lambda)\|_2^2}\Big).\label{upper:hard}
\end{align}
To obtain the last inequality, we have used Cauchy–Schwarz inequality and $\sup_{\theta \in \Theta (k_n, \mu_n)}\|\theta\|_2^2=k_n\mu_n^2$. From \eqref{upper:hard}, to show $R_H (\Theta (k_n, \mu_n), 1)\leq n\epsilon_n\mu_n^2$, it is sufficient to prove 
\[
\lim_{\lambda\rightarrow \infty}\sup_{\theta \in \Theta (k_n, \mu_n)}\E_{\theta}\|\etahatH(y,\lambda)\|_2^2=0.
\]
Define $f_{\lambda}(\mu):=\E|\hat{\eta}_H(\mu+z,\lambda)|^2, z\sim \mathcal{N}(0,1)$. It is not hard to verify that $f_{\lambda}(\mu)$, as a function of $\mu$, is symmetric around zero and increasing over $[0,\infty)$ for all $\lambda>0$. As a result, 
\begin{align*}
&\lim_{\lambda\rightarrow \infty}\sup_{\theta \in \Theta (k_n, \mu_n)}\E_{\theta}\|\etahatH(y,\lambda)\|_2^2\\
\leq& \lim_{\lambda\rightarrow \infty}\big[(n-k_n)f_{\lambda}(0)+k_nf_{\lambda}(\sqrt{k_n}\mu_n)\big] \\
=&(n-k_n)\lim_{\lambda\rightarrow \infty}f_{\lambda}(0)+k_n \lim_{\lambda\rightarrow \infty}f_{\lambda}(\sqrt{k_n}\mu_n) \\
=&0+0=0.
\end{align*}
The last line holds because 
% \textcolor{teal}{
% \begin{align*}
%     &\E [ \etahatH(\sqrt{k_{n}}\mu_{n}+z, \lambda) ]^{2} \\
%     =~& \int_{|\sqrt{k_{n}}\mu_{n} +z|>\lambda} (\sqrt{k_{n}}\mu_{n} + z)^{2} \phi(z) d z \\
%     =~& (k_{n}\mu_{n}^{2}+ 1) [ \tilde{\Phi}(\lambda - \sqrt{k_{n}}\mu_{n}) + \Tilde{\Phi}(\lambda + \sqrt{k_{n}}\mu_{n}) ] \\
%     +~& (\lambda+\sqrt{k_{n}}\mu_{n}) \phi(\lambda-\sqrt{k_{n}}\mu_{n}) + (\lambda-\sqrt{k_{n}}\mu_{n}) \phi(\lambda+\sqrt{k_{n}}\mu_{n}),
% \end{align*}
% where we used integration by parts to calculate the expectation. Thus, as $\lambda\rightarrow \infty$, $f_{\lambda}(\sqrt{k_{n}}\mu_{n})\rightarrow 0$.
% }
$\lim_{\lambda\rightarrow\infty}f_{\lambda}(\mu)=0, \forall \mu\in \mathbb{R}$ from dominated convergence theorem. The dominated convergence theorem can be used since $|\hat{\eta}_H(\mu+z,\lambda)|^2\leq |\mu+z|^2$ and $\lim_{\lambda\rightarrow\infty}|\hat{\eta}_H(\mu+z,\lambda)|^2=0$.

%------------------------
%------------------------
%-------------------------

\subsection{Proof of Theorem \ref{thm::regime_2}}
\label{proof:thm4:rg1}

Recall the scale invariance property in \secScaling: $R(\Theta(k_n,\tau_{n}),\sigma_{n}) = \sigma_{n}^{2} \cdot  R(\Theta(k_n,\mu_{n}), 1)$, where $\mu_{n}= \tau_{n}/\sigma_{n}$. Moreover, the estimator $\hat{\eta}_E(y,\lambda,\gamma):=\frac{1}{1+\gamma}\hat{\eta}_S(y,\lambda)$ defined in Equation \eqref{def::elastic-estimator} also preserves an invariance: $t\cdot \hat{\eta}_E(y,\lambda,\gamma)=\hat{\eta}_E(t y,t\lambda,\gamma),\forall t\geq 0$. Therefore, to prove both the upper and lower bounds, in this section, it is sufficient to consider the simpler unit-variance model: 
\begin{align}
    y_{i} = \theta_{i} + z_{i}, \quad i = 1, \ldots, n,\label{model::noise-level-one} 
\end{align}
where $(z_{i}) \simiid \mathcal{N}(0,1)$. We find an upper bound for the minimax risk by calculating the supremum risk of $\eta_E(y,\lambda,\gamma)$ with proper tuning. The lower bound is obtained by using Theorem \thmMinimaxLower\ and considering the independent block prior again. Both steps are more challenging than the corresponding steps in the proof of Theorem \ref{thm::regime_1}. 

\subsubsection{Upper bound}\label{sec:proof:thm:mediumrange}

To analyze the supremum risk of $\hat{\eta}_E(y,\lambda,\gamma)$, it is important to understand its risk in one dimension. Define the one-dimensional risk function as:
\begin{align}
\label{1d:risk:f}
r_e(\mu;\lambda,\gamma) = \E \Big(\frac{1}{1+\gamma} \etahatS (\mu + z, \lambda ) - \mu\Big)^{2},\quad z \sim \mathcal{N}(0,1).
\end{align}

The following property of the risk function plays a pivotal role in our analysis. 

\begin{lemma}
\label{lem::elastic-net-risk}
For any given tuning parameters $\lambda >0 $, $\gamma \in [0,+\infty]$, it holds that 
\begin{enumerate}[label=(\roman*)]
    \item $r_e(\mu;\lambda,\gamma)$, as a function of $\mu$, is symmetric, and increasing over $\mu \in [0,+\infty)$.
    \item $\max_{(x,y):x^{2}+y^{2} = c^{2}} [r_e(x;\lambda,\gamma) + r_e(y;\lambda,\gamma)] = 2r_e(c/\sqrt{2};\lambda,\gamma)$, \quad $\forall c>0.$
\end{enumerate}
\end{lemma}
\begin{proof}
\textit{(i)} Proving the symmetry of $r_e(\mu;\lambda,\gamma)$ is straightforward and is hence skipped. To prove the monotonicity of $r_e(\mu;\lambda,\gamma)$, we will calculate its derivative and show that it is positive for all $\mu>0$. To this end, we first decompose $r_e(\mu;\lambda,\gamma)$ into three terms: 
\begin{align*}
    r_e(\mu;\lambda,\gamma) =& \frac{1}{(1+\gamma)^{2}} \E(\etahatS(\mu+z,\lambda)-\mu)^{2} \\
    +& \frac{\gamma^2\mu^{2}}{(1+\gamma)^2} + \frac{2\gamma \mu}{(1+\gamma)^2} \E(\mu - \etahatS(\mu+z,\lambda)).
\end{align*}
Accordingly, the derivative of $r_e(\mu;\lambda,\gamma)$ takes the form:
\begin{align}
    \frac{\partial r_e(\mu;\lambda,\gamma)}{\partial \mu}  =& \frac{1}{(1+\gamma)^{2}} \frac{\partial \E(\etahatS(\mu+z,\lambda)-\mu)^{2}}{\partial\mu} +\frac{2\gamma^2\mu}{(1+\gamma)^2} \nonumber \\
    -& \frac{2\gamma}{(1+\gamma)^2}\bigg[\mu \frac{\partial \E(\etahatS(\mu + z,\lambda) - \mu)}{\partial \mu} \nonumber \\
    +&  \E(\etahatS(\mu + z,\lambda) - \mu)\bigg] . \label{dev:general}
    \end{align}
Using the explicit expression $\hat{\eta}_S(\mu+z,\lambda)={\rm sign}(\mu+z)(|\mu+z|-\lambda)_+$, we can calculate
    \begin{align*}
    & \frac{\partial  \E(\etahatS(\mu + z,\lambda) - \mu)}{\partial \mu} = \frac{\partial}{\partial \mu} \E\Big[(-\mu)I_{(|\mu+z|\leq \lambda)}  \\
    & + (z-\lambda)I_{(z+\mu > \lambda)} + (z+\lambda)I_{(z+\mu < -\lambda)}\Big] \\
     =& - \mathbb{P}(|z+\mu|\leq \lambda)  
     -\mu [-\phi(\lambda-\mu)+\phi(-\lambda-\mu)] \\
     -& \mu\phi(\lambda-\mu) + \mu\phi(-\lambda-\mu) 
     = - \mathbb{P}(|z+\mu|\leq \lambda),  
    \end{align*}
    and
    \begin{align*}
    &\frac{\partial \E(\etahatS(\mu+z,\lambda)-\mu)^{2}}{\mu} \\
    =~&\frac{\partial}{\partial \mu} \E \left[\mu^{2}I_{(|\mu+z|\leq \lambda)} + (z-\lambda)^{2}I_{(z+\mu > \lambda)} +(z+\lambda)^{2}I_{(z+\mu < -\lambda)}\right] \\
     =~ &2\mu \mathbb{P}(|\mu + z|\leq \lambda) + \mu^{2} [-\phi(\lambda-\mu) +\phi(-\lambda-\mu)] \\
     & + \mu^{2} \phi(\lambda-\mu) - \mu^{2} \phi(-\lambda-\mu) \\
     = ~&2\mu \mathbb{P}(|\mu + z|\leq \lambda).
\end{align*}
Putting the above two results into \eqref{dev:general}, we obtain, $\forall \mu > 0$,
\begin{align}
  & \frac{\partial r_e(\mu;\lambda,\gamma)}{\partial \mu} = \frac{2\mu}{(1+\gamma)} \mathbb{P}(|z+\mu|\leq \lambda) + \frac{2\gamma^2\mu}{(1+\gamma)^2}  \nonumber\\
  & + \frac{2\gamma}{(1+\gamma)^2}  \E(\mu - \etahatS(\mu + z,\lambda))  > 0,\label{derivative:1d:risk}
\end{align}
where the derivative is positive as all the terms on the right-hand side are non-negative and at least one of them is positive for all $\mu>0$. To verify this, all others are obvious and only the last term $\E(\mu - \etahatS(\mu + z,\lambda))$ needs be checked: this term is positive because it is an odd function and has positive derivative.

\textit{(ii)} Since the case where $\gamma=+\infty$ is trivial, we consider $\gamma\in [0,\infty)$ in the rest of the proof. Let $H(x) := r_e(x;\lambda,\gamma) +r_e(\sqrt{c^{2} - x^{2}};\lambda,\gamma)$ and consider $\max_{0\leq x \leq c} H(x)$. Since $H(x)$ is continuous over $[0,c]$, we find the maximum by evaluating the derivative of $H(x)$ over $(0,c)$. Using the derivative calculation \eqref{derivative:1d:risk}, we have
\begin{align*}
    H^{'}(x)  & = r_{e}^{'}(x;\lambda,\gamma) - \frac{x}{\sqrt{c^{2} - x^{2}}} r_{e}^{'}(\sqrt{c^{2} - x^{2}};\lambda,\gamma) \\
    & = \frac{2x}{1+\gamma} f_{1}(x) + \frac{2\gamma x}{(1+\gamma)^2} f_{2}(x),
    \end{align*}
    where
    \begin{align*}
    f_{1}(x) & :=  \mathbb{P}(|x+z|\leq \lambda) - \mathbb{P}(|\sqrt{c^{2}-x^{2}}+z|\leq \lambda), \\
    f_{2}(x) & :=  \frac{1}{\sqrt{c^{2}-x^{2}}} \E \etahatS(\sqrt{c^{2}-x^{2}}+z,\lambda) - \frac{1}{x} \E \etahatS(x+z,\lambda).
\end{align*}
We now show that $H^{'}(x)>0$ for $x\in (0,\frac{c}{\sqrt{2}})$, $H^{'}(\frac{c}{\sqrt{2}}) = 0 $, and $H^{'}(x) < 0$ for $x\in (\frac{c}{\sqrt{2}},c)$. It is straightforward to confirm that $H^{'}(\frac{c}{\sqrt{2}}) = 0$. Hence, it is sufficient to show both $f_{1}(x)$ and $f_{2}(x)$ are positive over $(0,\frac{c}{\sqrt{2}})$ and negative over $(\frac{c}{\sqrt{2}},c)$. This can be proved if we show that both $f_{1}(x)$ and $f_{2}(x)$ are strictly decreasing over $(0,c)$, given that $f_1(c/\sqrt{2})=f_2(c/\sqrt{2})=0$. 

Regarding $f_1(x)$, it is direct to verify that $\mathbb{P}(|x+z|\leq \lambda)$ is strictly decreasing over $(0,c)$, and accordingly $\mathbb{P}(|\sqrt{c^{2}-x^{2}}+z|\leq \lambda)$ is strictly increasing over $(0,c)$. Hence $f_1(x)$ is strictly decreasing over $(0,c)$. It remains to prove the monotonicity of $f_2(x)$. By the structure of $f_{2}(x)$, it is sufficient to show $\E [ \frac{1}{x}\etahatS(x+z,\lambda)]$ is a strictly increasing function of $x$ for $x>0$. We compute the derivative:
\begin{align*}
    & \frac{\partial \E \frac{1}{x}\etahatS(x+z,\lambda)}{\partial x} = -\frac{1}{x^{2}} \E \etahatS(x+z,\lambda) + \frac{1}{x} \mathbb{P}(|x+z|>\lambda) \\
    & = -\frac{1}{x^{2}} \bigg( \E \left[ (x+z-\lambda)I_{(x+z>\lambda)} + (x+z+\lambda) I_{(x+z<-\lambda)}  \right] \\
    & - x\int_{\lambda-x}^{\infty} \phi(z) d z - x\int_{\lambda+x}^{\infty} \phi(z) d z \bigg) \\
    & =- \frac{1}{x^{2}}  \bigg[\phi(\lambda - x ) - \lambda \int_{\lambda-x}^{\infty} \phi(z) d z
    + \lambda \int_{\lambda+x}^{\infty} \phi(z) d z  \\
    & - \phi(\lambda+x)\bigg]:= -\frac{1}{x^2}h(x).
\end{align*}
Therefore, for $x>0$, $\frac{\partial \E \frac{1}{x}\etahatS(x+z,\lambda)}{\partial x} >0$ if and only if $h(x)<0$. In fact,
\begin{align*}
    h^{'}(x) & = (\lambda-x) \phi(\lambda-x) + (\lambda+x) \phi(\lambda+x) \\
    & - \lambda \phi(\lambda-x) - \lambda \phi(\lambda+x) \\
    & = x(\phi(\lambda+x) - \phi(\lambda-x)) < 0, \quad \forall x>0.
\end{align*}
Also, it is straightforward to confirm that $h(0) = 0$. Thus $h(x)<0$ for $x>0$. 
\end{proof}

The one-dimensional risk function properties in Lemma \ref{lem::elastic-net-risk} will enable us to locate the parameter value at which the supremum risk of $\hat{\eta}_E(y,\lambda,\gamma)$ over the parameter space $\Theta(k_n,\mu_n)$ is achieved. The following lemma provides the detailed supremum risk calculation for a carefully-picked choice of the tuning. 

\begin{lemma}\label{lem:optimally-tuned_elastic}
Consider model \eqref{model::noise-level-one}. Suppose $\epsilon_{n}=k_n/n\rightarrow 0$, $\mu_{n}\rightarrow \infty$, and $\mu_{n} = o(\sqrt{\log \epsilon_{n}^{-1}})$, as $n\rightarrow \infty$. Then the estimator $\hat{\eta}_E(y,\lambda_n,\gamma_n)=\frac{1}{1+\gamma_n}\hat{\eta}_S(y,\lambda_n)$, with $\gamma_{n} = (2\epsilon_{n}\mu_{n}^{2}e^{\frac{3}{2}\mu_{n}^{2}})^{-1}-1$ and $\lambda_{n} = 2\mu_{n}$, has supremum risk:
\begin{align*}
    & \sup_{\theta \in \Theta(k_n,\mu_{n})} \E_{\theta}\|\hat{\eta}_E(y,\lambda_n,\gamma_n)-\theta\|_{2}^{2} \nonumber \\
    =&  k_n \mu_{n}^{2} - (\sqrt{2/\pi}+o(1))\cdot \frac{k_n^{2}}{n}\mu_{n}e^{\mu_{n}^{2}}.
\end{align*}
\end{lemma}
\begin{proof}
Using the one-dimensional risk function in \eqref{1d:risk:f}, we can write:
\begin{align*}
&\sup_{\theta \in \Theta(k_n,\mu_{n})} \E_{\theta}\|\hat{\eta}_E(y,\lambda_n,\gamma_n)-\theta\|_{2}^{2} \\
=& \sup_{\theta \in \Theta(k_n,\mu_{n})}\sum_{i=1}^nr_e(\theta_i;\lambda_n,\gamma_n).
\end{align*}
According to the properties proved in Lemma \ref{lem::elastic-net-risk}, it is clear that the above supremum is attained at the parameter vector $\theta$ in which there are $k_n$ non-zero components and they are all equal to $\mu_n$ (it occurs at a particular boundary of the parameter space $\Theta(k_n,\mu_{n})$). Therefore, the supremum risk can be simplified to
\begin{align}
\label{eq::maximum-elastic-net-risk}
   & \sup_{\theta \in \Theta(k_n,\mu_{n})} \E_{\theta}\|\hat{\eta}_E(y,\lambda_n,\gamma_n)-\theta\|_{2}^{2} 
   \nonumber\\
   =& n\Big[(1-\epsilon_n)r_e(0;\lambda_n,\gamma_n) + \epsilon_n r_e(\mu_n;\lambda_n,\gamma_n)\Big]  \nonumber\\
 =& n\bigg[\frac{1-\epsilon_n}{(1+\gamma_n)^2}\E \etahatS^{2}(z,\lambda_n)+\frac{\epsilon_n}{(1+\gamma_n)^2}\E \etahatS^{2}(\mu_n+z,\lambda_n) \nonumber \\
 -&\frac{2\epsilon_n\mu_n}{1+\gamma_n}\E \etahatS(\mu_n+z,\lambda_n)+\epsilon_n\mu_n^2\bigg]. 
\end{align}
To further calculate the supremum risk, we evaluate the three expectations in the above expression, using the Gaussian tail bound $\int_{t}^{\infty} \phi(z) d z = \left( \frac{1}{t} -\frac{1}{t^{3}} + \frac{3+o(1)}{t^{5}} \right) \phi(t)$ as $t\rightarrow \infty$. For the particular choice $\lambda_n= 2\mu_{n}\rightarrow \infty$, we have

\begin{align}
     \E \etahatS^{2}(z,\lambda_n) & = 2 \bigg[ (1+\lambda_n^{2})\int_{\lambda_n}^{\infty} \phi(z) d z - \lambda_n \phi(\lambda_n) \bigg] \nonumber \\
     & = \frac{1+o(1)}{2\mu_{n}^{3}}\phi(2\mu_n). \label{eq::regime_2-elastic_net-A_1}
\end{align}
Furthermore,
    \begin{align}
    & \E \etahatS^{2}(\mu_{n}+z,\lambda_n) \nonumber\\
    = &\bigg[(1+(\mu_{n}-\lambda_n)^{2})\int_{\lambda_n-\mu_{n}}^{\infty} \phi(z) d z
    -(\lambda_n- \mu_{n})\phi(\lambda_n- \mu_{n})\bigg] \nonumber\\
    +& \bigg[ (1+(\mu_{n}+\lambda_n)^{2})\int_{\lambda_n+\mu_{n}}^{\infty} \phi(z) d z - (\lambda_n+ \mu_{n})\phi(\lambda_n+ \mu_{n}) \bigg] \nonumber \\
    =& \frac{2+o(1)}{(\lambda_n-\mu_{n})^{3}} \phi(\lambda_n-\mu_{n}) + \frac{2+o(1)}{(\lambda_n+\mu_{n})^{3}}\phi(\lambda_n+\mu_{n}) \nonumber \\
    =&\frac{2+o(1)}{\mu_{n}^{3}} \phi(\mu_{n}), \label{eq::regime_2-elastic_net-A_2}
\end{align}
and
\begin{align}
     & \E\etahatS(\mu_{n}+z,\lambda_n) =  \phi(\lambda_n-\mu_{n}) - (\lambda_n-\mu_{n}) \nonumber \\
     \cdot & \int_{\lambda_n-\mu_{n}}^{\infty} \phi(z) d z - \phi(\lambda_n+\mu_{n}) + (\mu_{n}+\lambda_n)\int_{\lambda_n+\mu_{n}}^{\infty}\phi(z) d z  \nonumber\\
    =& \frac{1+o(1)}{(\lambda_n-\mu_{n})^{2}}\phi(\lambda_n-\mu_{n}) - \frac{1+o(1)}{(\lambda_n+\mu_{n})^{2}}\phi(\lambda_n+\mu_{n}) \nonumber \\
    =& \frac{1+o(1)}{\mu_{n}^{2}} \phi(\mu_{n}). \label{eq::regime_2-elastic_net-A_3}
\end{align}
Plugging \eqref{eq::regime_2-elastic_net-A_1}-\eqref{eq::regime_2-elastic_net-A_3} into \eqref{eq::maximum-elastic-net-risk} with the particular choice $\gamma_n=(2\epsilon_n\mu_n^2e^{\frac{3}{2}\mu_n^2})^{-1}-1$ considered in the lemma, we obtain
\begin{align*}
& \sup_{\theta \in \Theta(k_n,\mu_{n})} \E_{\theta}\|\hat{\eta}_E(y,\lambda_n,\gamma_n)-\theta\|_{2}^{2}~ \\
=&~ k_n\mu_n^2+(2+o(1))\cdot n\epsilon_n^2\mu_ne^{\frac{3}{2}\mu_n^2}\phi(\mu_n)
+ (8+o(1)) \\
&\cdot n\epsilon_n^3\mu_ne^{3\mu_n^2}\phi(\mu_n)-(4+o(1))\cdot n\epsilon_n^2\mu_ne^{\frac{3}{2}\mu_n^2}\phi(\mu_n) \\
=&~k_n\mu_n^2-(2+o(1))\cdot n\epsilon_n^2\mu_n e^{\frac{3}{2}\mu_n^2}\phi(\mu_n).
\end{align*}
The last equation holds because $\mu_n=o(\sqrt{\log\epsilon_n^{-1}})$ implies $\epsilon_n e^{\frac{3}{2}\mu_n^2}=o(1)$, so that the third term on the right-hand side of the first equation is negligible. 
\end{proof}

Now we can combine the preceding results we proved to obtain an upper bound for the minimax risk: with $\gamma_{n} = (2\epsilon_{n}\mu_{n}^{2}e^{\frac{3}{2}\mu_{n}^{2}})^{-1}-1$ and $\lambda_{n} = 2\mu_{n}$, it holds that
\begin{align*}
   & R(\Theta(k_n,\tau_{n}),\sigma_{n})  = \sigma_{n}^{2}\cdot  R(\Theta(k_n,\mu_{n}), 1) \\
     \leq & ~\sigma_{n}^{2} \cdot \sup_{\theta \in \Theta(k_n,\mu_{n})} \E_{\theta}\|\hat{\eta}_E(y,\lambda_n,\gamma_n)-\theta\|_{2}^{2} \\
     = &~ \sup_{\theta \in \Theta(k_n,\tau_{n})} \E_{\theta}\|\hat{\eta}_E(y,\sigma_n\lambda_n,\gamma_n)-\theta\|_{2}^{2} \\
    = &~ \sigma_{n}^{2} \left( k_n \mu_{n}^{2} - (\sqrt{2/\pi}+o(1))\cdot \frac{k_n^{2}}{n}\mu_{n}e^{\mu_{n}^{2}} \right) \\
    = &~  n\sigma_n^2\Big(\epsilon_n\mu_n^2-(\sqrt{2/\pi}+o(1))\epsilon_n^2\mu_ne^{\mu_n^2}\Big).
\end{align*}

\subsubsection{Lower bound}\label{sec::lower-bound-regime-2}

The derivation of the lower bound follows the same roadmap of proof for the lower bound in Theorem \ref{thm::regime_1}. It relies on the independent block prior constructed in \secRegimeILower{}. According to \EqBlockPriorLower, the key step is to calculate the Bayes risk $B(\pi_S^{\mu,m})$ of the symmetric spike prior $\mu_S^{\mu,m}$ for $(\mu \in (0,\mu_n])$, in the regime $m=n/k_n \rightarrow \infty, \mu_n\rightarrow \infty, \mu_n=o(\sqrt{\log \epsilon_n^{-1}})$. It turns out that setting $\mu=\mu_n$ will lead to a sharp lower bound. We summarize the result in the next lemma.

\begin{lemma}
\label{lem::single-spike-case}
As $m=n/k_n\rightarrow \infty,\mu_n\rightarrow \infty, \mu_n=o(\sqrt{\log \epsilon_n^{-1}})$, the Bayes risk $B(\pi_S^{\mu_n,m})$ satisfies
\begin{equation*}
    B(\pi_S^{\mu_n,m}) \geq \mu_{n}^{2} \bigg[ 1 - \frac{e^{\mu_{n}^{2}}}{2m} (1+o(1)) \bigg].
\end{equation*}
\end{lemma}

\begin{proof}
The result is an analog of Lemma \lemSpikeBayes\ in Regime (\rom{2}). Adopt the same notation from the proof of Lemma \lemSpikeBayes: $p_m=\frac{e^{\mu y_1}-e^{-\mu y_1}}{\sum_{i=1}^m(e^{\mu y_i}+e^{-\mu y_i})}$. In light of Lemma \lemSpikeBayesDecompose, it is sufficient to show that
\begin{itemize}
\item[(i)] $\E_{\mu_{n}e_{1}} (p_{m} - 1)^{2} \geq   1 - \frac{1}{m}e^{\mu_{n}^{2}} (1+o(1)) $, 
\item[(ii)] $(m-1) \E_{\mu_{n}e_{2}} p_{m}^{2} \geq \frac{1}{2m} e^{\mu_{n}^{2}} (1+o(1))$.
\end{itemize}
Regarding Part (i), we have
\begin{align*}
    &  \E_{\mu_{n}e_{1}} [p_m - 1]^{2} \geq 1 - 2\cdot \\
    & \E \left( \frac{e^{\mu_{n} (\mu_{n} + z_{1})} - e^{-\mu_{n} (\mu_{n}+z_{1})} }{\sum_{j\neq 1}  (e^{\mu_{n}z_{j}} + e^{-\mu_{n}z_{j}}) + e^{\mu_{n}(\mu_{n}+ z_{1})} + e^{-\mu_{n}(\mu_{n}+ z_{1})}}\right)  \\
    & \geq    1 - 2\cdot \\
    & \E \left( \frac{e^{\mu_{n}(\mu_{n}+z_{1})}}{\sum_{j\neq 1} (e^{\mu_{n}z_{j}} + e^{-\mu_{n}z_{j}}) + e^{\mu_{n}(\mu_{n}+z_{1})} + e^{-\mu_{n}(\mu_{n}+z_{1})} } \right).
\end{align*}
Thus, (i) will be proved by showing that
\begin{align*}
    & \E \left(\frac{e^{\mu_{n}(\mu_{n}+z_{1})}}{\sum_{j\neq 1} (e^{\mu_{n}z_{j}} + e^{-\mu_{n}z_{j}}) + e^{\mu_{n}(\mu_{n}+z_{1})} + e^{-\mu_{n}(\mu_{n}+z_{1})} } \right) \\
    & \leq \frac{1}{2m}e^{\mu_{n}^{2}}(1+o(1)).
\end{align*}
The expectation on the left-hand side of the above can be splitted into a summation of two truncated expectations according to the following condition:
\begin{align*}
    & e^{\mu_{n}(\mu_{n}+z_{1})} + e^{-\mu_{n}(\mu_{n}+z_{1})} \geq e^{\mu_{n}z_{1}} + e^{-\mu_{n}z_{1}} \\
    \Leftrightarrow & ~(e^{\mu_{n}^{2}} - 1) \left( e^{\mu_{n}z_{1}} - e^{-\mu_{n}z_{1} - \mu_{n}^{2}} \right) \geq 0 \\
    \Leftrightarrow &~ \mu_{n}z_{1} \geq - \mu_{n}z_{1} - \mu_{n}^{2} \\
    \Leftrightarrow &~ z_{1} \geq - \frac{1}{2} \mu_{n}.
\end{align*}
In the first case, 
\begin{align*}
    & \E \left(\frac{e^{\mu_{n}(\mu_{n}+z_{1})} I_{(z_{1}\geq -\frac{1}{2}\mu_{n})}}{\sum_{j\neq 1} (e^{\mu_{n}z_{j}} + e^{-\mu_{n}z_{j}}) + e^{\mu_{n}(\mu_{n}+z_{1})} + e^{-\mu_{n}(\mu_{n}+z_{1})} }\right) \\
    \leq ~& \E \left(\frac{e^{\mu_{n}(\mu_{n}+z_{1})} I_{(z_{1}\geq -\frac{1}{2}\mu_{n})}}{\sum_{j =1}^m (e^{\mu_{n}z_{j}} + e^{-\mu_{n}z_{j}})}\right) \\
    \leq ~& e^{\mu_n^2} \E\left( \frac{e^{\mu_{n}z_{1}}}{\sum_{j=1}^{m}(e^{\mu_{n}z_{j}} + e^{-\mu_{n}z_{j}})} \right) \\
    = ~& \frac{e^{\mu_n^2}}{2} \E \left( \frac{e^{\mu_{n}z_{1}} + e^{-\mu_{n}z_{1}}}{\sum_{j=1}^{m} (e^{\mu_{n}z_{j}} + e^{-\mu_{n}z_{j}})} \right)= \frac{e^{\mu_n^2}}{2m},
\end{align*}
where in the last two equations we have used the symmetry and exchangeability of i.i.d. standard normal variables $\{z_i\}_{j=1}^m$. In the second case, 
\begin{align}
    &  \E \left(  \frac{e^{\mu_{n}(\mu_{n}+z_{1})}  I_{(z_{1}\leq -\frac{1}{2} \mu_{n})}}{\sum_{j\neq 1} (e^{\mu_{n}z_{j}} + e^{-\mu_{n}z_{j}}) + e^{\mu_{n}(\mu_{n}+z_{1})} + e^{-\mu_{n}(\mu_{n}+z_{1})} }  \right) \nonumber \\
    \leq ~ & e^{\mu_n^2} \E\left(  \frac{e^{\mu_{n}z_{1}} I_{(z_{1}\leq -\frac{1}{2}\mu_{n})}}{\sum_{j=1}^{m}e^{\mu_{n}z_{j}}} \right) \nonumber \\
    = ~ & \frac{e^{\mu_n^2}}{m} \E\left( \frac{\sum_{j=1}^{m} e^{\mu_{n}z_{j}} I_{(z_{j}\leq -\frac{1}{2}\mu_{n})}}{\sum_{j=1}^{m}e^{\mu_{n}z_{j}}} \right), \label{second:case:mid}
\end{align}
where the last equality is again due to exchangeability of $\{z_j\}_{j=1}^m$.
Denoting
\begin{equation*}
   Y_{n}:= \frac{1}{m e^{\frac{1}{2}\mu_{n}^{2}}}  \sum_{j=1}^{m}e^{\mu_{n}z_{j}}, \quad  Z_{n} := \frac{1}{m e^{\frac{1}{2}\mu_{n}^{2}}} \sum_{j=1}^{m} e^{\mu_{n}z_{j}} I_{(z_{j}\leq -\frac{1}{2}\mu_{n})},
\end{equation*}
then the last expectation in \eqref{second:case:mid} can be written as $\E(Z_n/Y_n)$, and it remains to show $\E(Z_n/Y_n)=o(1)$. It is straightforward to check that $\E Y_{n} = 1$. Furthermore, since $\mu_{n} = o(\sqrt{\log \epsilon_{n}^{-1}})$, it is direct to verify that $\Var(Y_{n}) \leq \frac{m}{m^{2}e^{\mu_{n}^{2}}} e^{2\mu_{n}^{2}} = o(1)$. Hence, $Y_{n} \rightarrow 1$ in probability. In addition, 
\begin{align*}
     \E(Z_{n}) =~& \E \left( e^{\mu_{n}z_1} I_{z_1\leq -\frac{1}{2}\mu_{n}} \cdot e^{-\frac{1}{2} \mu_{n}^{2}}\right) \\
     =~&  \int_{-\infty}^{-\frac{1}{2}\mu_{n}} \frac{1}{\sqrt{2\pi}} e^{-\frac{z^{2}}{2} + \mu_{n}z - \frac{1}{2}\mu_{n}^{2}} d z \\
    = ~& \int_{-\infty}^{-\frac{\mu_{n}}{2}} \frac{1}{\sqrt{2\pi}} e^{-\frac{1}{2} (z - \mu_{n})^{2}} d z \\
    =~ & \int_{-\infty}^{-\frac{3}{2}\mu_{n}} \frac{1}{\sqrt{2\pi}} e^{-\frac{1}{2}z^{2}} d z = o(1).
\end{align*}
Thus, $Z_{n} \rightarrow 0$ in probability.  As a result, $Z_{n}/Y_n \rightarrow 0$ in probability. Since $|Z_{n}/Y_n|\leq 1$, dominated convergence theorem guarantees that $\E(Z_{n}/Y_n) \rightarrow 0$.

To prove Part (ii), it is equivalent to prove 
\begin{align*}
     & \E \frac{(e^{\mu_{n}z_{1}} - e^{-\mu_{n}z_{1}})^{2}}{[\sum_{j\neq 2} (e^{\mu_{n}z_{j}} + e^{-\mu_{n}z_{j}}) + e^{\mu_{n}(\mu_{n}+z_{2})} + e^{-\mu_{n}(\mu_{n}+z_{2})}]^{2}} \\
     & \geq \frac{1}{2m^2} e^{\mu_{n}^{2}} (1+o(1)).
\end{align*} 
Towards this goal, we have
\begin{align*}
    & \E \frac{(e^{\mu_{n}z_{1}} - e^{-\mu_{n}z_{1}})^{2}}{[\sum_{j\neq 2} (e^{\mu_{n}z_{j}} + e^{-\mu_{n}z_{j}}) + e^{\mu_{n}(\mu_{n}+z_{2})} + e^{-\mu_{n}(\mu_{n}+z_{2})}]^{2}} \\
    \labelrel\geq{eq::jensen-lower-bound-regime-2}~ &  \E \frac{(e^{\mu_{n}z_{1}} - e^{-\mu_{n}z_{1}})^{2}}{[2(m-2) e^{\frac{\mu_{n}^{2}}{2}} + e^{\frac{3}{2}\mu_{n}^{2}} + e^{-\frac{\mu_{n}^{2}}{2}} + e^{\mu_{n}z_{1}} + e^{-\mu_{n}z_{1}} ]^{2}} \\
    \labelrel\geq{eq::condition-regime-2} ~& \E \frac{(e^{\mu_{n}z_{1}} - e^{-\mu_{n}z_{1}})^{2} I_{(|z_{1}|\leq 3\mu_{n})} }{[2(m-2)e^{\frac{\mu_{n}^{2}}{2}} +4\sqrt{m} e^{\frac{\mu_{n}^{2}}{2}} ]^{2}} \\
    \labelrel={eq::symmetry-regime-2}~ & \frac{2}{e^{\mu_{n}^{2}}(2m-4 + 4\sqrt{m})^{2}} \cdot \left[ \E e^{2\mu_{n}z_{1}} I_{|z_{1}|\leq 3\mu_{n}} - \mathbb{P}(|z_{1}|\leq 3\mu_{n}) \right] \\
    = ~& \frac{2}{e^{\mu_{n}^{2}}(2m-4 + 4\sqrt{m})^{2}} \\
    & \cdot \left( e^{2\mu_{n}^{2}} \int_{-5\mu_{n}}^{\mu_{n}} \phi(z) d z - \int_{-3\mu_{n}}^{3\mu_{n}} \phi(z) d z  \right) \\
    = ~&  \frac{1}{2m^2} e^{\mu_{n}^{2}} (1+o(1)).
\end{align*}
Here, Inequality \eqref{eq::jensen-lower-bound-regime-2} is by applying the Jensen's inequality with respect to $z_{2}, \ldots, z_{m}$ (conditioned on $z_{1}$), as $1/(x+c)^{2}~(c>0)$ is a convex function of $x>0$. Inequality \eqref{eq::condition-regime-2} holds because $e^{\frac{3}{2}\mu_{n}^{2}} + e^{-\frac{\mu_{n}^{2}}{2}} + e^{\mu_{n}z_{1}} + e^{-\mu_{n}z_{1}}\leq 4e^{3\mu_n^2}$ when $|z_1|\leq 3\mu_n$, and $e^{3\mu_n^2}\leq \sqrt{m}e^{\frac{\mu_n^2}{2}}$ (for large $n$) under the condition $\mu_n=o(\sqrt{\log m})$. Equality \eqref{eq::symmetry-regime-2} is due to the symmetry of $z_{1} \sim \mathcal{N}(0,1)$. 
\end{proof}

Our goal now is to use Lemma \ref{lem::single-spike-case} to finish the proof of the lower bound in Theorem \ref{thm::regime_2}:
\begin{align*}
& R(\Theta(k_n,\tau_{n}),\sigma_n))\\
= ~&\sigma_n^2 \cdot R(\Theta(k_n,\mu_{n}),1) \geq k_n\sigma_n^2 \cdot  B(\pi_S^{\mu_n,m})\\
\geq~ & k_n\sigma^2_n \mu_{n}^{2} \bigg[ 1 - \frac{e^{\mu_{n}^{2}}}{2m} (1+o(1)) \bigg] \\
= ~& n\sigma_n^2 \Big[\epsilon_n\mu_n^2-\frac{1}{2}\epsilon_n^2\mu_n^2e^{\mu_n^2}(1+o(1))\Big].
\end{align*}

\subsection{Proof of Theorem \ref{thm::regime_2_compactness}}
\label{thm5:proof:sec}

Like in the proof of Theorems \ref{thm::regime_1} and \ref{thm::regime_2}, we calculate the minimax risk by deriving matching upper and lower bounds. However, a notable difference of the proof of Theorem \ref{thm::regime_2_compactness} is that the tight upper bound is obtained not by analyzing the supremum risk of a given estimator, but rather by a Bayesian approach. In this approach, we establish a uniform upper bound for the Bayes risk of an arbitrary distribution supported \emph{on average} on the parameter space, and use the minimax theorem (i.e. Theorem \thmMinimax) to connect the result to the matching upper bound of the minimax risk. We present the details of the upper and lower bounds in Sections \ref{thm5:upper} and \ref{thm5:lower}, respectively.

\subsubsection{Upper bound}
\label{thm5:upper}

Consider the univariate Gaussian model:
\begin{equation}\label{model::univariate}
    Y = \theta + Z, 
\end{equation}
where $\theta \in \mathbb{R}$ and $Z \sim \mathcal{N}(0,1)$.
For a given constant $A > 1$, define a class of priors for $\theta$:
\begin{align}\label{eq::compact-priors}
    \Gamma^{A}(\epsilon, \mu) := \Big\{&\pi \in \mathcal{P}(\mathbb{R}): \pi(\{0\}) \geq 1-\epsilon,~\E_{\pi}\theta^{2} \leq \epsilon \mu^{2}, \nonumber \\
    & ~\supp(\pi) \in [-A\mu, A\mu] \Big\},
\end{align}
where $\mathcal{P}(\mathbb{R})$ denotes the class of all probability measures defined on $\mathbb{R}$, and $\epsilon \in [0,1], \mu>0$.
Note that $\pi \in \Gamma^{A}(\epsilon, \mu)$ implies that $\pi = (1-\epsilon)\delta_{0} + \epsilon G$, for some distribution $G$ satisfying $\E_G\theta^2 \leq \mu^{2}$ and $\supp(G) \subseteq [-A\mu, A\mu]$. The worst-case Bayes risk (i.e., the one of the least favorable distribution), under this univariate Gaussian model with squared error loss, is defined as 
\begin{equation}\label{eq::univariate-bayes-risk}
    B^{A}(\epsilon, \mu, 1) := \sup \Big\{B(\pi): \pi \in \Gamma^{A}(\epsilon,\mu)\Big\},
\end{equation}
where
\[
B(\pi) = \mathbb{E} (\mathbb{E}(\theta|Y)-\theta)^2,~~\theta\sim \pi, ~Y\mid \theta \sim \mathcal{N}(\theta, 1).
\]

The following lemma allows us to obtain an upper bound for $R(\Theta^{A}(k_n,\tau_{n}), \sigma_{n})$ in terms of $B^{A}(\epsilon, \mu, 1)$.

\begin{lemma}\label{lem::reduce_to_univariate_upper_bound}
The minimax risk satisfies the following inequality:
\begin{equation*}
    R(\Theta^{A}(k_n,\tau_{n}), \sigma_{n}) \leq n \sigma_{n}^{2} \cdot B^{A}(\epsilon_{n}, \mu_{n}, 1).
\end{equation*}
\end{lemma}

\begin{proof}
 The proof closely follows the arguments in the proof of Theorem 8.21 of \cite{johnstone19}. However, since the parameter space we consider is different, we cover a full proof here for completeness. For notational simplicity, let $\Theta_{n} := \Theta^{A}(k_n,\tau_{n})$. Consider the class of priors
\begin{align*}
    \calM_{n} & := \calM(k_n, \tau_{n}, A) = \Big \{ \pi \in \calP(\R^{n}): \E_{\pi}\norm{\theta}_{0} \leq k_n, \\
    &~ \E_{\pi}\norm{\theta}_{2}^{2} \leq k_n \tau_{n}^{2}, ~\supp(\pi) \subseteq [-A\tau_{n}, A\tau_{n}]^{n}\Big \},
\end{align*}
where $\calP(\R^{n})$ denotes the set of all probability measures on $\mathbb{R}^n$.
Let $\calM_{n}^{e} := \calM^{e}(k_n,\tau_{n}, A) \subseteq \calM(k_n, \tau_{n}, A)$ be its exchangeable subclass, consisting of the distributions $\pi \in \calM_{n}$ that are permutation invariant over the $n$ coordinates. Using notation $B(\pi, \calM) := \sup_{\pi \in \calM} B(\pi)$, we will show that
\begin{equation}
\label{eq::upper_bound_reduce_to_one_dim}
    R(\Theta_{n},\sigma_{n}) \leq B(\pi, \calM_{n}) = B(\pi, \calM_{n}^{e}) \leq n \sigma^{2}_{n} \cdot B^{A}(\epsilon_{n}, \mu_{n},1).
\end{equation}

We start with equality in \eqref{eq::upper_bound_reduce_to_one_dim}. 
\begin{align*}
 & R(\Theta_{n},\sigma_{n}) = \inf_{\hat{\theta}} \sup_{\theta \in \Theta_n} \mathbb{E} \|\hat{\theta}- \theta\|_2^2 \overset{(a)}{\leq} \inf_{\hat{\theta}} \sup_{\pi \in \calM_n} \mathbb{E}_\pi \|\hat{\theta}- \theta\|_2^2 \\
& \overset{(b)}{=} \sup_{\pi \in \calM_n}  \inf_{\hat{\theta}} \mathbb{E}_\pi \|\hat{\theta}- \theta\|_2^2 = B(\pi, M_n). 
\end{align*}
Inequality (a) is due to the fact that $\calM_{n}$ contains all point mass priors $\delta_{\theta}$, for every $\theta \in \Theta_{n}$. To obtain Equality (b) we have used the minimax theorem, i.e. Theorem \thmMinimax, as $\calM_{n}$ is a convex set of probability measures. To prove the second inequality in \eqref{eq::upper_bound_reduce_to_one_dim}, note that for any $\pi \in \calM_{n}$, we can construct a corresponding prior:
\[
\pi^e = \frac{1}{n!} \sum_{\sigma: [n]\rightarrow [n]} \pi \circ \sigma,
\]
where $\sigma$ denotes a permutation of the coordinates of $\theta$, and $\pi \circ \sigma$ is the distribution after permutation. In other words, $\pi^e$ is the distribution averaged over all the permutations, thus $\pi^e \in \calM^e_n$. Given that $B(\pi)$ is a concave function (it is the infimum of linear functions), we have $B(\pi, \calM_{n}) \leq B(\pi, \calM_{n}^{e})$ which implies $B(\pi, \calM_{n}) = B(\pi, \calM_{n}^{e})$ since $\calM_{n}^e \subseteq \calM_{n}$.

To show the last inequality in \eqref{eq::upper_bound_reduce_to_one_dim}, for any exchangeable prior $\pi \in \calM_{n}^{e}$, let $\pi_{1}$ be its univariate marginal distribution. Using the constraints on $\pi$ from $\calM_{n}$ and the fact that $\pi$ is symmetric over its $n$ coordinates, we have
\begin{equation*}
    \pi_{1}(\theta_{1}= 0)\geq 1- \epsilon_{n}, ~ \E_{\pi_{1}} \theta_{1}^{2} \leq \epsilon_{n} \tau_{n}^{2} , ~ \supp \pi_1 \subseteq [-A\tau_{n}, A\tau_{n}]
\end{equation*}
 Hence $\pi_{1}\in \Gamma^{A}(\epsilon_{n}, \tau_{n})$ defined in \eqref{eq::compact-priors}. Furthermore, according to Theorem \thmIndepLessFavor, the product prior $\pi_{1}^{n}$
is less favorable than $\pi^{e}$, namely, $B(\pi) \leq B(\pi_{1}^{n})=nB(\pi_1)$. Rescaling the noise level to one and maximizing over $\pi_{1} \in \Gamma^{A}(\epsilon_{n}, \mu_{n})$ completes the proof.
\end{proof}

Lemma \ref{lem::reduce_to_univariate_upper_bound} reduces the problem of obtaining the upper bound for frequentist minimax risk (under Gaussian sequence model) to the problem of upper bounding the worst-case Bayes risk (under a univariate Gaussian model). Our next goal is to find an upper bound for $B^{A}(\epsilon_{n}, \mu_{n},1)$. Towards this end, we first state a useful lemma.

\begin{lemma}\label{lem::second-moment-under-pi}
Under model \eqref{model::univariate}, consider prior $\pi = (1-\epsilon) \delta_{0} + \epsilon G \in \Gamma^{A}(\epsilon, \mu)$, as defined in \eqref{eq::compact-priors}. Then,
\begin{equation*}
    \E(\E(\theta|Y))^2 = \int \frac{\epsilon^2 (\int t e^{t z - \frac{t^{2}}{2}} d G(t) )^2}{1-\epsilon + \epsilon \int e^{t z - \frac{t^{2}}{2}} d G (t) } \phi(z) d z,
\end{equation*}
where $\phi(\cdot)$ denotes the density function of standard normal random variable.
\end{lemma}
\begin{proof}
Given the prior $\pi = (1-\epsilon) \delta_{0} + \epsilon G$, the posterior mean of $\theta$ is given by
\begin{equation*}
    \E (\theta|Y=y) = \frac{\epsilon \int \theta \phi(y-\theta) dG(\theta)}{(1-\epsilon)\phi(y)+\epsilon \int \phi(y-\theta)dG(\theta)}.
\end{equation*}
Thus,
\begin{align*}
     & \E(\E(\theta|Y))^2 \\
     =~ & (1-\epsilon) \int \left[ \frac{\epsilon \int t \phi(z-t) dG(t)}{(1-\epsilon)\phi(z)+\epsilon \int \phi(z-t)dG(t)} \right]^{2} \phi(z) d z \\
     & +\epsilon \iint \left[ \frac{\epsilon \int t \phi(\theta + \Tilde{z} -t) dG(t)}{(1-\epsilon)\phi(\theta + \Tilde{z})+\epsilon \int \phi(\theta + \Tilde{z} -t)dG(t)} \right]^{2} \\
     & \cdot \phi(\Tilde{z}) d \Tilde{z} d G(\theta) \\
    = ~& \int \left[ \frac{\epsilon \int t \phi(z-t) dG(t)}{(1-\epsilon)\phi(z)+\epsilon \int \phi(z-t)dG(t)} \right]^{2} \\
    & \cdot \left[ (1-\epsilon) \phi(z) + \epsilon \int \phi(z-\theta) d G(\theta) \right] d z \\
    = ~& \int \left[ \frac{ \epsilon \int t e^{tz-\frac{t^{2}}{2}} d G(t) }{(1-\epsilon) + \epsilon \int e^{tz - \frac{t^{2}}{2}} d G(t)} \right]^{2} \\
    & \cdot \left[ (1-\epsilon) + \epsilon \int e^{tz - \frac{t^{2}}{2}} d G(t) \right] \phi(z) d z \\
    = ~& \int \frac{\epsilon^2 (\int t e^{t z - \frac{t^{2}}{2}} d G(t) )^2}{1-\epsilon + \epsilon \int e^{t z - \frac{t^{2}}{2}} d G (t) } \phi(z) d z,
\end{align*}
where the second equality is by a simple change of variable.
\end{proof}

We can now obtain a sharp upper bound for $B^{A}(\epsilon_{n}, \mu_{n},1)$. 

\begin{lemma}\label{lem::univariate-Bayes}
Consider $\epsilon_n\rightarrow 0,\mu_n\rightarrow \infty,\mu_n=o(\sqrt{\log\epsilon_n^{-1}})$. Under model \eqref{model::univariate}, the worst-case Bayes risk $B^{A}(\epsilon_n, \mu_n, 1)$ defined in \eqref{eq::univariate-bayes-risk} satisfies that for any $A>1$,
\begin{equation*}
    B^{A}(\epsilon_n, \mu_n, 1) \leq \epsilon_n \mu^2_n -\frac{1+o(1)}{2}\epsilon_n^2\mu_n^2e^{\mu_{n}^{2}}.
\end{equation*}
\end{lemma}
\begin{proof}
For prior $\pi \in \Gamma^{A}(\epsilon, \mu)$, using the law of total expectation, 
\begin{align}\label{eq::MSE-decomposition}
    \E (\E(\theta|Y) - \theta)^{2} 
    & = \E \theta^{2} - \E(\E(\theta|Y))^{2}.
\end{align}
We first obtain a lower bound for the term $\E(\E(\theta|Y))^{2}$. We start with the expression derived in Lemma \ref{lem::second-moment-under-pi} and develop a series of lower bounds, 
\begin{align}
      & \E(\E(\theta|Y))^2 = \int \frac{\epsilon^2 (\int t e^{t z - \frac{t^{2}}{2}} d G(t) )^2}{1-\epsilon + \epsilon \int e^{t z - \frac{t^{2}}{2}} d G (t) }\phi(z) d z  \nonumber \\
      \geq ~& \int_{|z|\leq \sqrt{\log 1/\epsilon}} \frac{\epsilon^2 (\int t e^{tz - \frac{t^2}{2}} d G(t) )^2}{1-\epsilon + \epsilon \int e^{tz - \frac{t^2}{2}}dG(t)}\phi(z)d z \nonumber \\
    \labelrel\geq{eq::dominate-denominator} ~& \frac{\epsilon^2}{1-\epsilon + \epsilon^{\frac{1}{2}}} \int_{|z|\leq \sqrt{\log 1/\epsilon}} \left(\int t e^{tz - \frac{t^2}{2}} d G(t) \right)^2 \phi(z) d z \nonumber \\
    \labelrel={eq::GGprime} ~& \frac{\epsilon^2}{1-\epsilon + \epsilon^{\frac{1}{2}}} \iint \bigg[t t' e^{t t' }  \int_{-\sqrt{\log 1/\epsilon}- (t+t')}^{\sqrt{\log 1/\epsilon} -(t+t')} \phi(z) d z\bigg]d G(t) d G(t')\nonumber \\
    = ~& \frac{\epsilon^2}{1-\epsilon + \epsilon^{\frac{1}{2}}} \iint_{tt'\geq 0} \bigg[t t' e^{t t' }  \int_{-\sqrt{\log 1/\epsilon}- (t+t')}^{\sqrt{\log 1/\epsilon} -(t+t')} \phi(z) d z\bigg]d G(t) d G(t') \nonumber \\
      + ~& \frac{\epsilon^2}{1-\epsilon + \epsilon^{\frac{1}{2}}} \iint_{tt'<0} \bigg[t t' e^{t t' }  \int_{-\sqrt{\log 1/\epsilon}- (t+t')}^{\sqrt{\log 1/\epsilon} -(t+t')} \phi(z) d z\bigg] d G(t) d G(t') \nonumber \\
    \labelrel\geq{eq::negative-exponent}~ & \frac{ \epsilon^2}{1-\epsilon + \epsilon^{\frac{1}{2}}} \Bigg( \iint_{tt'\geq 0} \bigg[t t' e^{t t' }  \int_{-\sqrt{\log 1/\epsilon}- (t+t')}^{\sqrt{\log 1/\epsilon} -(t+t')} \phi(z) d z\bigg]  \nonumber \\
    & \cdot  d G(t) d G(t') - |A\mu|^2 \Bigg) \nonumber \\
\labelrel\geq{eq::inequality-compact} ~& \frac{ \epsilon^2}{1-\epsilon + \epsilon^{\frac{1}{2}}} \left(\int_{-\sqrt{\log 1/\epsilon}- 2A\mu}^{\sqrt{\log 1/\epsilon} -2A\mu} \phi(z) d z \right. \nonumber \\
& \left. \cdot \iint_{tt'\geq 0} t t' e^{t t' } d G(t) d G(t') - |A\mu|^2 \right).\label{square:term:bayes}
\end{align}
Inequality \eqref{eq::dominate-denominator} holds because for $|z|\leq \sqrt{\log 1/\epsilon}$, 
\begin{equation*}
    \epsilon \int e^{tz - \frac{t^2}{2}} d G(t) = \epsilon e^{\frac{1}{2}z^2}\int e^{-\frac{1}{2}(z-t)^2} d G(t) \leq \epsilon e^{\frac{1}{2}z^2}\leq \epsilon^{\frac{1}{2}}.
\end{equation*}
To obtain Equality \eqref{eq::GGprime} we do the following simple calculations:
\begin{align*}
    & \int_{|z|\leq \sqrt{\log 1/\epsilon}} \left(\int t e^{tz - \frac{t^2}{2}} d G(t) \right)^2 \phi(z) d z \\
    = ~& \int_{|z|\leq \sqrt{\log 1/\epsilon}} \bigg[\iint  tt' e^{zt - t^2 /2} e^{zt'-{t'}^2 /2}d G(t) d G(t')\bigg]  \phi(z) d z \\
    = ~& \iint \bigg[t t' e^{tt'}  \int_{|z| \leq \sqrt{\log 1/\epsilon}} \frac{1}{\sqrt{2\pi}} e^{-\frac{1}{2}(z-(t+t'))^2} d z\bigg]d G(t) d G(t') \\
    = ~& \iint \bigg[t t' e^{tt'}  \int_{ -\sqrt{\log 1/\epsilon}-(t+t')}^{\sqrt{\log 1/\epsilon}-(t+t')} \frac{1}{\sqrt{2\pi}} e^{-\frac{1}{2}z^2} d z\bigg]d G(t) d G(t')
\end{align*}

Inequality \eqref{eq::negative-exponent} holds because $e^{-|tt'|}\leq 1$ and $\supp G \subseteq [-A\mu, A\mu]$. Inequality \eqref{eq::inequality-compact} is due to the fact that $\supp G \subseteq [-A\mu, A\mu]$ and $\int_{-\sqrt{\log 1/\epsilon}-a}^{\sqrt{\log 1/\epsilon}-a}\phi(z)dz$ (as a function of $a$) is symmetric and decreasing over $[0,\infty)$. To continue from \eqref{square:term:bayes}, we further lower bound $ \iint_{tt'\geq 0} tt' e^{ tt'} d G(t) dG(t')$. To simplify notation, define two random variables $t,t'\overset{i.i.d.}{\sim} G$. We have
\begin{align*}
    & \iint_{tt'\geq 0} tt' e^{ tt'} d G(t) dG(t') = \E[tt'e^{tt'}I_{(tt'\geq 0)}] \\
    = ~& \sum_{k=0}^{\infty} \E  \frac{1}{k!}(tt')^{k+1} (I_{(t> 0,t'> 0)}  + I_{(t< 0,t'< 0)})  \\
     = ~& \sum_{k=0}^{\infty} \frac{1}{k!} \left(\E[t^{k+1}I_{(t> 0)}]  \cdot \E [(t')^{k+1} I_{(t'> 0)}] \right. \\
     + ~& \left. \E [t^{k+1}I_{(t< 0)}] \cdot \E [(t')^{k+1} I_{(t'< 0)}]\right)\\
     = ~&\sum_{k=0}^{\infty} \frac{1}{k!} \left((\E t^{k+1}I_{(t > 0)})^2 + (\E t^{k+1} I_{(t< 0)})^2 \right)  \\
     = ~& \sum_{k=0}^{\infty} \frac{1}{k!} \left((\E |t|^{k+1} I_{(t> 0)})^2 + (\E |t|^{k+1} I_{(t<0)})^2 \right) \\
     \labelrel\geq{eq::basic} ~& \sum_{k=0}^{\infty} \frac{1}{k!} \frac{1}{2} \left( \E |t|^{k+1} I_{(t> 0)} +  \E |t|^{k+1} I_{(t< 0)} \right)^{2} \\
     = ~& \frac{1}{2} \sum_{k=0}^{\infty} \frac{1}{k!} \left( \E |t|^{k+1} \right)^{2}
     \labelrel\geq{eq::jensen-inequality} \frac{1}{2}(\E|t|)^2 \\
     +~& \frac{1}{2}\sum_{k=1}^{\infty} \frac{1}{k!} \left(\E |t|^{2} \right)^{k+1} \geq  \frac{1}{2} \Big(\E |t|^2 e^{\E |t|^2}-\E|t|^2 \Big),
\end{align*}
where \eqref{eq::basic} is due to the basic inequality $2(x^2+y^2)\geq (x+y)^2$, and \eqref{eq::jensen-inequality} is by H\"{o}lder's inequality $(\E|t|^2)^{k+1}\leq(\E|t|^{k+1})^2, k\geq 1$. Combining the above inequality with \eqref{eq::MSE-decomposition} and \eqref{square:term:bayes} gives
\begin{align}
\label{key:form:nonasym}
 & B^{A}(\epsilon_n, \mu_n, 1)=\sup_{\pi\in \Gamma^A(\epsilon_n,\mu_n)} \E (\E(\theta|Y) - \theta)^{2} \nonumber \\
 \leq &\sup_{\E|t|^2\leq \mu_n^2}\Big(\epsilon_n+\frac{\epsilon_n^2\Delta_n }{2(1-\epsilon_n+\sqrt{\epsilon_n})}\Big)\E|t|^2 \nonumber \\
 - & \frac{\epsilon_n^2\Delta_n}{2(1-\epsilon_n+\sqrt{\epsilon_n})}\E|t|^2e^{\E|t|^2}+\frac{\epsilon^2A^2\mu_n^2}{1-\epsilon+\sqrt{\epsilon}},
\end{align}
where $\Delta_n=\int_{-\sqrt{\log 1/\epsilon_n}- 2\mu_n A}^{\sqrt{\log 1/\epsilon_n} - 2\mu_n A} \phi(z) dz$. The results we obtained so far are non-asymptotic. We now make use of the conditions $\epsilon_n\rightarrow 0,\mu_n\rightarrow \infty, \mu_n=o(\sqrt{\log\epsilon_n^{-1}})$ to derive the final asymptotic result. Under such scaling conditions, it is straightforward to confirm that the expression on the right-hand side of \eqref{key:form:nonasym} is increasing in $\E|t|^2$ when $n$ is sufficiently large (by calculating its derivative). As a result, 
\begin{align*}
& B^{A}(\epsilon_n, \mu_n, 1) \leq \Big(\epsilon_n+\frac{\epsilon_n^2\Delta_n }{2(1-\epsilon_n+\sqrt{\epsilon_n})}\Big)\mu_n^2 \\
& - \frac{\epsilon_n^2\Delta_n}{2(1-\epsilon_n+\sqrt{\epsilon_n})}\mu_n^2e^{\mu_n^2}+\frac{\epsilon^2A^2\mu_n^2}{1-\epsilon+\sqrt{\epsilon}} \\
&=\epsilon_n\mu_n^2+\frac{1+o(1)}{2}\epsilon_n^2\mu_n^2-\frac{1+o(1)}{2}\epsilon_n^2\mu_n^2e^{\mu_n^2}+O(\epsilon_n^2\mu_n^2) \\
&=\epsilon_n\mu_n^2-\frac{1}{2}\epsilon_n^2\mu_n^2e^{\mu_n^2}(1+o(1)).
\end{align*}
\end{proof}

Combing Lemmas \ref{lem::reduce_to_univariate_upper_bound} and \ref{lem::univariate-Bayes} provides the upper bound for the minimax risk:
\[
 R(\Theta^{A}(k_n,\tau_{n}), \sigma_{n})\leq n\sigma_n^2\Big(\epsilon_n\mu_n^2-\frac{1}{2}\epsilon_n^2\mu_n^2e^{\mu_n^2}(1+o(1))\Big).
\]

\subsubsection{Lower bound}
\label{thm5:lower}

Recall that in the lower bound derivation for Theorem \ref{thm::regime_2}, in Section \ref{sec::lower-bound-regime-2}, the proof is based on the independent block prior $\pi^{IB}$ with single spike distribution $\pi_S^{\mu_n,m}$ which is first introduced in \secRegimeILower{}. Since the spike locations are at $\pm \mu_{n}$, which are contained in $[-A \mu_{n}, A \mu_{n}]$ for any $A>1$,  this implies that $\supp \pi^{IB} \subseteq \Theta^{A}(k_n,\mu_{n})$ as well. As a result, the proof in Section \ref{sec::lower-bound-regime-2} also works for the new parameter space $\Theta^{A}(k_n,\mu_{n})$ and it yields the same lower bound:
\[
R(\Theta^{A}(k_n,\tau_{n}), \sigma_{n})\geq n\sigma_n^2\Big(\epsilon_n\mu_n^2-\frac{1}{2}\epsilon_n^2\mu_n^2e^{\mu_n^2}(1+o(1))\Big).
\]
\subsection{Proof of Proposition \ref{lem::lasso-risk-regime-2}}\label{app:proof_prop2}

Comparing the results in Propositions \ref{lem::lasso-risk-regime-1} and \ref{lem::lasso-risk-regime-2}, we can see that the supremum risk of optimally tuned soft thresholding has the same second-order asymptotic approximation in Regimes (\rom{1}) and (\rom{2}). Thus, the proof of Proposition \ref{lem::lasso-risk-regime-2} shares a lot of similarity with that of Proposition \ref{lem::lasso-risk-regime-1}. For simplicity we will not repeat every detail. Referring to the proof of Proposition \ref{lem::lasso-risk-regime-1} in \secPropI, the key is to obtain the accurate order of the optimal tuning $\lambdas$ and evaluate the function value $F(\lambdas)$, where we recall the definitions: $\lambdas=\argmin_{\lambda\geq 0}F(\lambda)$, $z\sim \mathcal{N}(0,1)$ and
\begin{align*}
    F(\lambda)=(1-\epsilon_n)\E\hat{\eta}^2_S(z,\lambda)+\epsilon_n \E(\hat{\eta}_S(\mu_n+z,\lambda)-\mu_n)^2.
\end{align*}

We first address the order of $\lambdas$.

\begin{lemma}
\label{optimal:tuing:regime:2}
Consider $\epsilon_n\rightarrow 0,\mu_n\rightarrow \infty,\mu_n=o\left(\sqrt{\log \epsilon_n^{-1}}\right)$, as $n\rightarrow \infty$. It holds that 
\begin{equation}\label{eq:label:optlambda}
    \log 2 \epsilon_n^{-1} + \frac{\mu_{n}^{2}}{2} -2 \log \log \frac{2}{\epsilon_{n}} < \lambdas \mu_n <  \log{2\epsilon_n^{-1}} + \frac{\mu_n^2}{2},
\end{equation}
for sufficiently large $n$.
\end{lemma}

\begin{proof}
This lemma is an analog of Lemma \lemPropItuning\ (comparing \eqLemPropItuning~with \eqref{eq:label:optlambda}). The proof is thus similar too. We will skip equivalent calculations and only highlight the differences. 

First, we show that $\lambdas\mu_n^{-1}\rightarrow \infty$. Otherwise, $\lambdas\mu_n^{-1}\leq C$ for some constant $C>0$ (take a subsequence if necessary). Then when $n$ is large,
\begin{align*}
     & F(\lambdas) \geq (1-\epsilon_n)\E\hat{\eta}^2_S(z,\lambdas) \geq (1-\epsilon_n)\E\hat{\eta}^2_S(z,C\mu_n) \\
    & = 2(1-\epsilon_n) \left[ (1+(C\mu_{n})^2) \int_{C\mu_{n}}^{\infty} \phi(z) d z - C\mu_{n} \phi(C\mu_{n})\right] \\
    & \overset{(a)}{=} \frac{4+o(1)}{\mu_{n}^{3}}\phi(C\mu_{n}) \overset{(b)}{>} \epsilon_n\mu_n^2=F(+\infty),
\end{align*}
where (a) is by the Gaussian tail bound, and (b) is due to $\mu_{n} = o\left(\sqrt{\log \epsilon_{n}^{-1}}\right)$. The result $F(\lambdas)>F(+\infty)$ contradicts with the optimality of $\lambdas$.

Second, we utilize the derivative equation $F'(\lambdas)=0$ in \eqSoftThresholdZeroDerivative\ to obtain more accurate order information of $\lambdas$. The results $\mu_n\rightarrow \infty,\lambdas\mu_n^{-1}\rightarrow \infty$ imply that $\lambdas\rightarrow\infty,\lambdas-\mu_n\rightarrow\infty,\lambdas\mu_n\rightarrow\infty$. This is all needed to obtain \eqLassoRegimeILoII\ and \eqLimitCaseILastNew{}. As a result, \eqAKeyMaster\ holds here as well:
\begin{align}
\label{first:order:equation}
 2+o(1)=\epsilon_n\mu_n \lambdas \exp(\lambdas \mu_n-\mu_n^2/2).
\end{align}
To reach \eqref{eq:label:optlambda} under the scaling $\mu_n=o(\sqrt{\log\epsilon_n^{-1}})$, the rest of the argument is exactly the same as the one in the proof of Lemma \lemPropItuning\ .  
\end{proof}

The next lemma characterizes $F(\lambdas)$.

\begin{lemma}
\label{function:eval:opt}
Consider $\epsilon_n\rightarrow 0,\mu_n\rightarrow \infty,\mu_n=(\sqrt{\log \epsilon_n^{-1}})$, as $n\rightarrow \infty$. It holds that 
\begin{align*}
F(\lambdas)=\epsilon_{n} \mu_{n}^{2} - \exp\left[ -\frac{1}{2} \frac{1}{\mu_{n}^{2}} \left(\log \frac{1}{\epsilon_{n}}\right)^{2} \left(1+o(1)\right) \right].
\end{align*}
\end{lemma}

\begin{proof}
This proof deviates a bit from the one of \lemRefineOrder{}. We will more directly utilize the order information of $\lambdas$ proved in Lemma \ref{optimal:tuing:regime:2} to calculate $F(\lambdas)$. Before that, we need a refinement of \eqref{first:order:equation}. This is achieved by refining \eqLassoRegimeILoII\ and \eqLimitCaseILastNew\ with higher-order approximations: 
\begin{align*}
& - \lambdas\int_{\lambdas}^{\infty} \phi(z)dz + \phi(\lambdas) = \frac{1+O(\lambdas^{-2})}{\lambdas^{2}}\phi(\lambdas),  \\
  & -\phi(\lambdas-\mu_{n}) + \lambdas\int_{\lambdas-\mu_{n}}^{\infty} \phi(z)dz = \\
  & \left[\frac{\mu_{n}}{\lambdas-\mu_{n}} - \frac{\lambdas+O(\lambdas^{-1})}{(\lambdas-\mu_{n})^{3}}\right]\phi(\lambdas-\mu_{n}), \\
    & -\phi(\lambdas+\mu_{n}) + \lambdas\int_{\lambdas+\mu_{n}}^{\infty} \phi(z)dz
    = o\left(\frac{1}{\lambdas^{4}}\right) \phi(\lambdas-\mu_{n}).
\end{align*}
Plugging the above into \eqSoftThresholdZeroDerivative\ and arranging terms gives 
\begin{align}
    & e^{\lambdas\mu_n-\frac{\mu_n^2}{2}}\frac{\epsilon_n\mu_n\lambdas^2}{2(\lambdas-\mu_n)}-1 \nonumber \\
    =~&\frac{(1-\epsilon_n)(1+O(\lambdas^{-2}))\mu_n}{\mu_n-(\lambdas-\mu_n)^{-2}(\lambdas+O(\lambdas^{-1}))}-1 \nonumber \\
    =~&\frac{\lambdas(\lambdas-\mu_n)^{-2}+O(\lambdas^{-2}\mu_n)}{\mu_n-(\lambdas-\mu_n)^{-2}(\lambdas+O(\lambdas^{-1}))}=\frac{1+o(1)}{\lambdas\mu_n}, \label{key:refinement}
\end{align}
where in the second equality we have used $\epsilon_n\lambdas^2=o(1)$ and $\lambdas^{-1}\mu_n=o(1)$ which are implied by the order of $\lambdas$ from Lemma \ref{optimal:tuing:regime:2}.

Now we are ready to evaluate $F(\lambdas)$. We first use Gaussian tail bound to approximate the three expectations (i.e. \eqRiskAtZertoSoftSecondMom) in the expression of $F(\lambdas)$ (i.e. Equation \eqref{eq::soft-thresholding-risk-of-lambda}):
\begin{align*}
\E\hat{\eta}^2_S(z,\lambdas) &= \frac{4+O(\lambdas^{-2})}{\lambdas^{3}}\phi(\lambdas), \\
\E\hat{\eta}_S(\mu_n+z,\lambdas)&= \frac{1+O(\lambdas^{-2})}{(\lambdas - \mu_{n})^{2}} \phi(\lambdas-\mu_{n}), \\
\E \etahatS^{2}(\mu_{n}+z, \lambdas)&=\frac{2+O(\lambdas^{-2})}{(\lambdas-\mu_{n})^{3}} \phi(\lambdas - \mu_{n}).
\end{align*}
Using these three approximations in Equation \eqref{eq::soft-thresholding-risk-of-lambda}, we obtain
\begin{align*}
    & F(\lambdas)  = (1-\epsilon_{n}) \frac{4+O(\lambdas^{-2})}{\lambdas^{3}}\phi(\lambdas) + \epsilon_{n} \mu_{n}^{2} \\
    & - 2\epsilon_{n} \mu_{n} \frac{1+O(\lambdas^{-2})}{(\lambdas-\mu_{n})^{2}} \phi(\lambdas-\mu_{n}) + \epsilon_{n} \frac{2+O(\lambdas^{-2})}{(\lambdas-\mu_{n})^{3}} \phi(\lambdas -\mu_{n}) \\
    & = \epsilon_{n} \mu_{n}^{2} - \phi(\lambdas) \left[ \frac{-4+O(\epsilon_{n}+\lambdas^{-2})}{\lambdas^{3}} + \frac{2\epsilon_{n}\mu_{n}}{(\lambdas-\mu_{n})^{2}} \right. \\
    & \cdot \left. e^{\lambdas \mu_{n} - \frac{\mu_{n}^{2}}{2}} \left(1+O\left(\frac{1}{\lambdas\mu_{n}}\right)\right)\right].
    \end{align*}
    We further replace $e^{\lambdas\mu_n-\frac{\mu_n^2}{2}}$ in the above with the result from \eqref{key:refinement} to have
    \begin{align*}
    & F(\lambdas) = \epsilon_{n} \mu_{n}^{2} - \phi(\lambdas) \cdot \left[ \frac{-4 + O(\epsilon_{n} + \lambdas^{-2})}{\lambdas^{3}} \right. \\
    & + \left. \frac{4}{\lambdas^{2}(\lambdas-\mu_{n})} \left(1+O\left(\frac{1}{\lambdas\mu_{n}}\right)\right)\right] \\
    & \overset{(a)}{=} \epsilon_{n} \mu_{n}^{2} - \phi(\lambdas) \frac{4\mu_{n}}{\lambdas^{3}(\lambdas-\mu_{n})} \left(1+O\left(\frac{1}{\mu_{n}^{2}} \right)\right) \\
    & = \epsilon_{n} \mu_{n}^{2} - \frac{4+o(1)}{\sqrt{2\pi}} e^{-\frac{\lambdas^2}{2}} \cdot \frac{\mu_{n}}{\lambdas^4} \nonumber \\
        & \overset{(b)}{=} \epsilon_{n} \mu_{n}^{2} - \exp\left[ -\frac{1}{2} \frac{1}{\mu_{n}^{2}} \left(\log \frac{1}{\epsilon_{n}}\right)^{2} \left(1+o(1)\right) \right].
\end{align*}
Here, to obtain $(a)$ we have used $\epsilon_n\lambdas^2=o(1)$ and $\lambdas^{-1}\mu_n=o(1)$ implied by Lemma \ref{optimal:tuing:regime:2}; $(b)$ is due to the order $\lambdas=\mu_n^{-1}\log\epsilon_n^{-1}(1+o(1))$ again from Lemma \ref{optimal:tuing:regime:2}.
\end{proof}

Lemma \ref{function:eval:opt} readily leads to the supremum risk of optimally tuned soft thresholding:
\begin{align*}
& \inf_{\lambda}\sup_{\theta\in \Theta(k_n,\tau_n)}\E_{\theta}\|\hat{\eta}_S(y,\lambda)-\theta\|_2^2 
= n\sigma_n^2F(\lambdas)\\
= ~& n\sigma_n^2\left(\epsilon_{n} \mu_{n}^{2} - \exp\bigg[ -\frac{1}{2} \frac{1}{\mu_{n}^{2}} \left(\log \frac{1}{\epsilon_{n}}\right)^{2} \left(1+o(1)\right) \bigg]\right).
\end{align*}

%--------------------
%--------------------
%---------------------

\subsection{Proof of Proposition \ref{lem::hardthreshold-risk-regime-2}}\label{proof:hardthreshold:reg2}
The proof of this proposition is similar to the proof of Proposition \ref{lem::hardthreshold-risk-regime-1} presented in Section \ref{proof:hardthreshold:reg1}. Hence, for the sake of brevity we adopt the same notation from Section \ref{proof:hardthreshold:reg1} and only discuss the differences. If $R_H (\Theta (k_n, \tau_n), \sigma_n)$ denotes the supremum risk of optimally tuned hard thresholding estimator, then we will have
\begin{equation*}
   R_H (\Theta (k_n, \tau_n), \sigma_n) = \sigma_n^2\cdot  R_H (\Theta (k_n, \mu_n), 1). 
\end{equation*}
   Without loss of generality, let $\sigma_n = 1$ in the model.  As in the proof of Proposition \ref{lem::hardthreshold-risk-regime-1}, we obtain a lower bound by calculating the risk at the following specific value of $\underline{\theta}$ such that $\underline{\theta}_{i} = \mu_{n}$ for $i \in \{ 1, 2, \ldots, k_n\}$ and $\underline{\theta}_i = 0$ for $i > k_n$. We have
   \begin{equation}
    \E_{\underline{\theta}} \| \etahatH(y,\lambda) -  \underline{\theta}\|_2^{2}  =  n \Big[ (1-\epsilon_{n})r_{H}(\lambda, 0) + \epsilon_{n} r_{H}(\lambda,\mu_{n}) \Big]. \label{eq:hard:decomp:p4}
   \end{equation}
To evaluate $\inf_{\lambda>0}\E_{\underline{\theta}}\| \etahatH(y,\lambda) -  \underline{\theta}\|^{2}$, we consider three scenarios for the optimal choice of $\lambda_n$, denoted by $\lambda_n^*$.
\begin{itemize}
    \item \textbf{Case \rom{1} $\lambda^*_{n} = O(1)$:} In this case, $\lambda^*_{n}\leq c$ for some constant $c >0$. Using the same argument as the one presented for Case \rom{1} in the proof of Proposition \ref{lem::hardthreshold-risk-regime-1}, we have 
    \begin{align*}
     \inf_{\lambda>0} ~\E_{\underline{\theta}}\| \etahatH(y,\lambda) -  \underline{\theta}\|_2^{2} \geq 2n(1-\epsilon_{n}) (1-\Phi(c))  . 
    \end{align*}
    Since $\epsilon_n \mu_{n}^2 \rightarrow 0$ and $(1-\epsilon_{n}) 2(1-\Phi(c)) = \Uptheta(1)$, we conclude that $\inf_{\lambda>0}\E_{\underline{\theta}}\| \etahatH(y,\lambda) -  \underline{\theta}\|_2^{2} = \omega(n\epsilon_n \mu_{n}^2 )$.  
\item \textbf{Case \rom{2}} $\lambda^*_n = \omega(1)$ and $\lambda^*_{n}=O(\mu_n)$: Let $c_{1}$ be a fixed number larger than $1$. There exists $c_{2}$ such that for large enough $n$, $c_{1}<\lambda^*_n \leq c_{2} \mu_n$. We thus obtain 
  \begin{align*}
      & \inf_{\lambda>0} ~\E_{\underline{\theta}}\| \etahatH(y,\lambda) -  \underline{\theta}\|_2^{2}=\E_{\underline{\theta}}\| \etahatH(y,\lambda^*_{n}) -  \underline{\theta}\|_2^{2} \nonumber \\
   &= n \Big[ (1-\epsilon_{n})r_{H}(\lambda^*_{n}, 0) + \epsilon_{n} r_{H}(\lambda^*_{n},\mu_{n}) \Big] \nonumber \\
        & \geq  n(1-\epsilon_{n}) r_{H}(\lambda^*_{n} ,0) \nonumber \\
        &=  n (1-\epsilon_{n}) \Big[ 2\lambda^*_{n}\phi(\lambda^*_{n}) + 2(1-\Phi(\lambda^*_{n})) \Big] \\
        &\geq  2n (1-\epsilon_{n}) \lambda^*_{n}\phi(\lambda^*_{n}) \\
        &\geq 2n (1- \epsilon_n) \frac{c_{1}}{\sqrt{2\pi}} {\rm e}^{- \frac{c_{2}^2 \mu_n^2}{2}}\geq n\epsilon_n\mu_n^2,
    \end{align*}
    where the last inequality is due to the scaling $\mu_n = o(\sqrt{\log\epsilon_{n}^{-1}})$ in the current regime. 
    \item \textbf{Case \rom{3} $\lambda^*_{n}=\omega(\mu_n)$:} In a similar way as in the proof of Case II of Proposition \ref{lem::hardthreshold-risk-regime-1}, we can conclude that 
    \begin{align*}
      &\inf_{\lambda>0} ~\E_{\underline{\theta}}\| \etahatH(y,\lambda) -  \underline{\theta}\|_2^{2}\nonumber \\
      \geq~& k_n\mu_{n}^{2} +k_n(\lambda_n^*-\mu_n+o(\lambda_n^*))\cdot \phi(\lambda^*_{n} - \mu_{n}) \nonumber \\
      +~& k_n (\lambda^*_{n} + \mu_{n}+o(\lambda_n^*))\cdot \phi(\lambda^*_{n} + \mu_{n})\\
        \geq~& k_n \mu_{n}^{2}=n\epsilon_n\mu_n^2.
    \end{align*}
    \end{itemize}
Note that since the three cases we have discussed above cover all the ranges of $\lambda_n^*$, we conclude that
\begin{eqnarray*}
R_H (\Theta (k_n, \mu_n), 1)\geq \inf_{\lambda>0} \E_{\underline{\theta}}\| \etahatH(y,\lambda) -  \underline{\theta}\|_2^{2} \geq n \epsilon_{n} \mu_n^2.
\end{eqnarray*}
The proof of the upper bound is the same as the proof of the upper bound for Proposition \ref{lem::hardthreshold-risk-regime-1} and is hence skipped here.

%------------------------
%------------------------
%-------------------------

\subsection{Proof of Theorem \ref{thm::regime_4}}\label{prf:thm:nine}

Based on the scale invariance property of minimax risk mentioned in \secScaling, it is equivalent to prove
\begin{equation*}
    R(\Theta(k_n,\mu_{n}), 1) = 2n\epsilon_n\log \epsilon_{n}^{-1} - 2 n\epsilon_n \nu_{n} \sqrt{2 \log \nu_{n}} \left( 1+o(1) \right),
\end{equation*}
where $\nu_{n} = \sqrt{2 \log \epsilon_{n}^{-1}}$. As in the proof of Theorems \ref{thm::regime_1} and \ref{thm::regime_2}, we first obtain an upper bound by analyzing the supremum risk of hard thresholding, and then develop a matching lower bound via the Bayesian approach. Before proceeding with the proof, we cover a few properties of the one-dimensional risk function of hard thresholding that becomes useful in the proof of Theorem \ref{thm::regime_4}.  

\subsubsection{Properties of the risk of hard thresholding estimator}

Consider the one-dimensional risk of hard thresholding for $\mu \in \mathbb{R}$ and $\lambda>0$,
\begin{equation*}
    r_{H}(\lambda,\mu) := \E \left( \etahatH(\mu + z,\lambda) - \mu \right)^{2}, \quad z \sim \mathcal{N}(0,1).
\end{equation*}
The following lemma from \cite{johnstone19} gives simple and yet accurate bounds for $r_{H}(\lambda,\mu)$. Let
\begin{equation*}
    \Bar{r}_{H}(\lambda, \mu) = \begin{cases}
    \min \{r_{H}(\lambda,0) + 1.2\mu^{2},  1+ \mu^{2}\} & 0 \leq \mu \leq \lambda \\
    1 + \mu^{2}(1-\Phi(\mu-\lambda)) & \mu \geq \lambda,
    \end{cases}
\end{equation*}
where $\Phi(\cdot)$ is the CDF of standard normal random variable.

\begin{lemma}[Lemma 8.5 in \cite{johnstone19}]
\label{lem::hard-threshold-risk-nonasymptotic}
\emph{}
\begin{enumerate}[label=(\alph*)]
    \item For $\lambda > 0$ and $\mu \in \R$,
    \begin{equation*}
        (5/12)\Bar{r}_{H}(\lambda,\mu) \leq r_{H}(\lambda,\mu) \leq \Bar{r}_{H}(\lambda,\mu).
    \end{equation*}
    \item The large $\mu$ component of $\Bar{r}_{H}$ has the bound
    \begin{equation*}
        \sup_{\mu \geq \lambda} \mu^{2}(1-\Phi(\mu-\lambda)) \leq \begin{cases}
        \lambda^{2}/2 & \text{if } \lambda \geq \sqrt{2\pi} \\
        \lambda^{2} & \text{if } \lambda \geq 1.
        \end{cases}
    \end{equation*}
\end{enumerate}
\end{lemma}

Our main goal in this section is to derive accurate approximations for $\sup_{\mu \geq 0} r_{H}(\lambda, \mu)$. The next lemma provides an accurate characterization of the risk for two different choices of $\mu$. The importance of these choices becomes clear when we analyze $\sup_{\mu \geq 0} r_{H}(\lambda, \mu)$ later in this section.

\begin{lemma}\label{lem::hard-threshold-risk-near-lambda}
As $\lambda \rightarrow \infty$, the risk of the hard thresholding, $r_{H}(\lambda,\mu)$, satisfies
\begin{align*}
r_{H}(\lambda, \lambda) & = \frac{1+o(1)}{2}\lambda^{2},\\
r_{H}(\lambda, \lambda-\sqrt{2 \log \lambda}) & = \lambda^2-(2\sqrt{2}+o(1))\lambda\sqrt{\log\lambda}.
\end{align*}
\end{lemma}

\begin{proof}
First note that the risk of hard thresholding can be written as
% \begin{subequations}
\begin{align}
     r_{H}(\lambda, \mu) = &\mu^{2} \left[ \Phi(\lambda - \mu) - \Phi(-\lambda - \mu) \right] + \int_{|z+\mu|>\lambda} z^{2} \phi(z) d z \nonumber\\
    = & (\mu^{2} - 1) \left[ \Phi(\lambda-\mu) - \Phi(-\lambda - \mu) \right] \nonumber \\
    + & 1 + (\lambda - \mu) \phi(\lambda -\mu) + (\lambda + \mu) \phi(\lambda + \mu). \label{eq::hard-threshold-risk-form-3}
\end{align}
Let $\mu = \lambda - \sqrt{2 \log \lambda}$. As $\lambda \rightarrow \infty$, we analyze the order of each term in the above expression:
\begin{align*}
    & r_{H}(\lambda, \lambda - \sqrt{2 \log \lambda}) \\
    = & ~  [(\lambda - \sqrt{2 \log \lambda})^{2}-1]\cdot \left( 1-\frac{1+o(1)}{\sqrt{2\log \lambda}}\phi(\sqrt{2\log\lambda}) \right) \\
    + & ~  1 + \sqrt{2 \log \lambda} \cdot \phi(\sqrt{2 \log \lambda}) \\
    + & ~ (2\lambda-\sqrt{2\log\lambda})\phi(2\lambda-\sqrt{2\log\lambda})  \\
    = & ~ \left( \lambda - \sqrt{2 \log \lambda} \right)^{2}  + O\left(\frac{\lambda}{\sqrt{\log \lambda}}\right) \\
    = & ~ \lambda^2-(2\sqrt{2}+o(1))\lambda\sqrt{\log\lambda},
\end{align*}
where in the first equality we have applied the Gaussian tail bound: $1-\Phi(x)=(1+o(1))x^{-1}\phi(x)$ as $x \rightarrow \infty$. To prove the first part of the lemma, let $\mu = \lambda$. From \eqref{eq::hard-threshold-risk-form-3} we have
\begin{align*}
    r_{H}(\lambda, \lambda) & = (\lambda^{2} - 1) \left(\frac{1}{2} - \Phi(-2\lambda)\right) + 1 + 2 \lambda \phi(2\lambda) \\
    & = \lambda^{2}/2 \left(1 + o(1)\right).
\end{align*}
\end{proof}

We now obtain the asymptotic approximation of $\sup_{\mu \geq 0} r_{H}(\lambda, \mu)$ in the next lemma.

\begin{lemma}\label{lem::sup-risk-hard-threshold}
 As $\lambda \rightarrow \infty$, the supremum risk satisfies
\begin{align*}
    \sup_{\mu \geq 0} r_{H}(\lambda, \mu) 
    & = \lambda^{2} - 2\sqrt{2} \lambda \sqrt{ \log \lambda} + o(\lambda \sqrt{ \log \lambda}).
\end{align*}
\end{lemma}
\begin{proof}
Define
\begin{equation*}
    \mu^{*} = \argmax_{\mu\geq 0} r_{H}(\lambda,\mu).
\end{equation*}
Comparing the upper bounds from Lemma \ref{lem::hard-threshold-risk-nonasymptotic} and the risk at $\lambda-\sqrt{2\log\lambda}$ in Lemma \ref{lem::hard-threshold-risk-near-lambda}, we can conclude that the superemum risk is attained at $\mu=\mu^* \leq \lambda$ (when $\lambda$ is large). To evaluate $r_H(\lambda,\mu^*)$, it is important to derive an accurate approximation for $\mu^*$. We first claim that $\mu^*/\lambda\rightarrow 1$. Suppose this is not true. Then $\mu^* \leq c \lambda $ for some constant $c \in [0,1)$ (take a sequence if necessary). According to Lemma \ref{lem::hard-threshold-risk-nonasymptotic} (a), for large enough values of $\lambda$, we have 
\begin{equation*}
    r_{H}(\lambda, \mu^*) \leq \Bar{r}_{H}(\lambda,\mu^*) \leq 1 + (\mu^*)^{2} \leq \Tilde{c} \lambda^2, \quad \Tilde{c} 
    \in (0,1).
\end{equation*}
However, the above upper bound is strictly smaller than the risk $r_{H}(\lambda, \lambda-\sqrt{2\log \lambda})$ calculated in Lemma \ref{lem::hard-threshold-risk-near-lambda}, contradicting with the definition of $\mu^*$.

Second, we show that $\lambda - \mu^{*} \rightarrow \infty$, while $(\lambda- \mu^*)/\lambda \rightarrow 0$. Otherwise, it satisfies $0\leq \lambda - \mu^* \leq c$ for some finite constant $c\geq 0$ (take a sequence if necessary). Then from \eqref{eq::hard-threshold-risk-form-3} we have
\begin{align*}
    r_{H}(\lambda,\mu^*) \leq \Phi(c)\lambda^{2} \left( 1+o(1) \right).
\end{align*}
Comparing this with $r_{H}(\lambda, \lambda-\sqrt{2\log \lambda})$ from Lemma \ref{lem::hard-threshold-risk-near-lambda} leads to the same contradiction. 

Third, we prove that for any given $c>1$, $\lambda-\mu^*\leq c\sqrt{2\log\lambda}$ for sufficiently large $\lambda$. Otherwise, there exists some constant $c>1$ such that $\lambda_n-\mu_n^*>c\sqrt{2\log\lambda_n}$ for a sequence $\lambda_n\rightarrow \infty$ as $n\rightarrow \infty$. As a result, using Equation \eqref{eq::hard-threshold-risk-form-3}, and the result proved earlier that $\lambda_n-\mu_n^*\rightarrow \infty$, we obtain that for large $n$,
\begin{align*}
    r_{H}(\lambda_n,\mu_n^*) 
    & \leq (\mu_n^*)^{2}  +1 + (\lambda_n- \mu_n^* ) \phi(\lambda_n-\mu_n^*) \nonumber \\
    & + (\lambda_n+\mu_n^*) \phi(\lambda_n+\mu_n^*) \nonumber\\
    & \leq \left( \lambda_n- c\sqrt{2\log \lambda_n} \right)^{2} +O(1)  \nonumber \\
    & = \lambda_n^{2} - (2c+o(1))\lambda_n \sqrt{2 \log \lambda_n}.
\end{align*}
Again, comparing the above with $r_{H}(\lambda_n, \lambda_n-\sqrt{2\log \lambda_n})=\lambda_n^2- (2+o(1))\lambda_n \sqrt{2 \log \lambda_n}$ in Lemma \ref{lem::hard-threshold-risk-near-lambda}, we see that $r_{H}(\lambda_n,\mu_n^*)<r_{H}(\lambda_n, \lambda_n-\sqrt{2\log \lambda_n})$ when $n$ is large, which is a contradiction.

Finally, we prove that $(\lambda-\mu^*)/\sqrt{2\log\lambda}\rightarrow 1$ as $\lambda\rightarrow \infty$. Suppose this is not true. Given the result proved in the last paragraph, then there exists some constant $c<1$ such that $\lambda_n-\mu_n^*<c\sqrt{2\log\lambda_n}$ for a sequence $\lambda_n\rightarrow \infty$ as $n\rightarrow \infty$. Using Equation \eqref{eq::hard-threshold-risk-form-3} and Gaussian tail bound $1-\Phi(x)=\frac{1+o(1)}{x}\phi(x)$ as $x\rightarrow \infty$, we have
\begin{align*}
    r_{H}(\lambda_n, \mu_n^*) & =  (\mu_n^*)^{2} \left[ \Phi(\lambda_n-\mu_n^*) - \Phi(-\lambda_n-\mu_n^*) \right]  + O(1) \\
    & \leq (\mu_n^*)^{2} \Phi(\lambda_n-\mu_n^*)+ O(1) \\
    & = (\mu_n^*)^{2} \left[ 1 - \frac{1+o(1)}{\lambda_n-\mu_n^*} \phi(\lambda_n-\mu^*_n) \right]+O(1).
\end{align*}
Because $\phi(\lambda_n - \mu_n^*) \geq 1/\sqrt{2\pi} \cdot \exp\left(-\frac{2c^{2}\log \lambda_n}{2}\right) = 1/(\sqrt{2\pi} \lambda_n^{c^2})$, we continue with 
\begin{align*}
    r_{H}(\lambda_n,\mu_n^*) & \leq (\mu_n^*)^{2} - \frac{(\lambda_n- c\sqrt{2 \log \lambda_n})^{2}}{c \sqrt{2\log \lambda_n}} \frac{1}{\sqrt{2\pi}\lambda_n^{c^2}} \\
    & \cdot \left(1+o(1)\right)+O(1)\\
    & \leq \lambda_n^{2} - \frac{\lambda_n^{2-c^2}}{\sqrt{\log\lambda_n}}\cdot \Big(\frac{1}{2c\sqrt{\pi}}+o(1)\Big).
\end{align*}
Note that for $c<1$, $\lambda_n^{2-c^2}/\sqrt{\log\lambda_n} = \omega(\lambda_n\sqrt{\log \lambda_n})$. Hence $r_{H}(\lambda_n,\mu_n^*)<r_{H}(\lambda_n, \lambda_n-\sqrt{2\log \lambda_n})$ when $n$ is sufficiently large. The same contradiction arises.

Having the precise order that $\mu^*=\lambda-(1+o(1))\sqrt{2\log \lambda}$, we can easily evaluate $ \sup_{\mu \geq 0} r_H(\lambda, \mu)$ from \eqref{eq::hard-threshold-risk-form-3}: as $\lambda\rightarrow \infty$,
\begin{align*}
 & r_{H}(\lambda, \lambda-\sqrt{2\log \lambda})\leq ~\sup_{\mu\geq 0} r_H(\lambda, \mu) = r_H(\lambda, \mu^*) \\
= & ~ (\mu^*)^2 (\Phi (\lambda - \mu^*)-\Phi(-\lambda - \mu^*))+ O(1)  \\
   \leq & ~ (\mu^*)^2 +O(1)=(\lambda-(1+o(1))\sqrt{2\log \lambda})^2+O(1)  \\
    = & ~\lambda^2-2\sqrt{2}\lambda\sqrt{\log\lambda}+o(\lambda\sqrt{\log\lambda}).
\end{align*}
Combining this result with Lemma \ref{lem::hard-threshold-risk-near-lambda} completes the proof. 
\end{proof}

\subsubsection{Upper bound}

We are in the position to compute the supremum risk of $\hat{\eta}_H(y,\lambda_n)$ with $\lambda_n=\sigma_n\sqrt{2\log\epsilon_n^{-1}}$ in Theorem \ref{thm::regime_4}. First of all, due to the scale invariance of hard thresholding, the supremum risk can be written in the form: 
\begin{align*}
    & \sup_{\theta \in \Theta(k_n,\tau_{n})} \E_{\theta} \norm{\etahatH(y, \lambda_n) - \theta}_{2}^{2} \\
    =& ~ \sigma_n^2\bigg[(n-k_n) r_{H}(\nu_n, 0) + \sup_{\|\tilde{\theta}\|_2^2 \leq k_n \mu_{n}^{2}}\sum_{i=1}^{k_n}r_{H}(\nu_n, \tilde{\theta}_i)\bigg],
\end{align*}
where $\tilde{\theta}\in\mathbb{R}^{k_n}$ and $\nu_n=\sqrt{2\log\epsilon_n^{-1}}$. Given that the one-dimensional risk function $r_{H}(\nu_n, \tilde{\theta}_i)$ is symmetric in $\tilde{\theta}_i$, if its maximizer satisfies $\argmax_{\tilde{\theta}_i\geq 0}r_{H}(\nu_n, \tilde{\theta}_i) \leq \mu_n$, then we will have
\begin{align}\label{eq::sup-risk-hard-threshold}
  & \sup_{\theta \in \Theta(k_n,\tau_{n})} \E_{\theta} \norm{\etahatH(y, \lambda_n) - \theta}_{2}^{2}  \nonumber \\
  = & ~\sigma_n^2\bigg[(n-k_n) r_{H}(\nu_n, 0) + k_n \sup_{\mu\geq 0}r_{H}(\nu_n,\mu)\bigg].
\end{align}
This will allow us to focus on finding the supremum risk of hard thresholding in the univariate setting that we discussed in the last section. In the proof of Lemma \ref{lem::sup-risk-hard-threshold}, we already showed that $\argmax_{\tilde{\theta}_i\geq 0}r_{H}(\nu_n, \tilde{\theta}_i)\leq \nu_n$ when $n$ is large. It is then clear that in the current regime $\mu_{n} = \omega(\sqrt{2 \log \epsilon_{n}^{-1}})$, it holds that $\argmax_{\tilde{\theta}_i\geq 0}r_{H}(\nu_n, \tilde{\theta}_i) \leq \mu_n$ for large $n$. Therefore, the supremum risk of hard thresholding over $\Theta(k_n,\tau_{n})$ can be simplified as in  \eqref{eq::sup-risk-hard-threshold}. We can apply Lemma \ref{lem::sup-risk-hard-threshold} to continue from \eqref{eq::sup-risk-hard-threshold}:
\begin{align}
      &\sup_{\theta \in \Theta(k_n,\tau_{n})} \E_{\theta} \norm{\etahatH(y, \lambda_n) - \theta}_{2}^{2} \nonumber \\
      = & ~ n\sigma_n^2 \left[ (1-\epsilon_{n}) r_{H}(\nu_n,0) + \epsilon_{n} \sup_{\mu \geq 0} r_{H}(\nu_n,\mu) \right]  \nonumber\\
     = & ~ n \sigma_n^2 \left[ (1-\epsilon_{n}) r_{H}(\nu_n,0) + \epsilon_{n}\left(\nu_n^{2} - 2\nu_n \sqrt{2\log \nu_n} \right. \right. \nonumber \\
     + &~ \left. \left. o(\nu_n \sqrt{\log \nu_n}) \right) \right], \label{eq::hard-threshold-sup-risk}
\end{align}
where $\nu_n=\sqrt{2\log\epsilon_n^{-1}}$. We now identify the dominating terms in the above expression. First,
\begin{align}\label{eq::hard-threshold-risk-at-zero}
    r_{H}(\nu_n, 0) & = 2 \int_{\nu_n}^{\infty} z^{2} \phi(z) d z = 2\nu_n \phi(\nu_n) + 2(1-\Phi(\nu_n)) \nonumber \\
    = & (2+o(1))\nu_n \phi(\nu_n)=O(\epsilon_n\nu_n),
\end{align}
where the last two equations are due to the Gaussian tail bound $1-\Phi(x)=\frac{1+o(1)}{x}\phi(x)$ as $x\rightarrow \infty$ and $\nu_n=\sqrt{2\log\epsilon_n^{-1}}$.
Therefore, from \eqref{eq::hard-threshold-sup-risk} we obtain
\begin{align*}
   & \sup_{\theta \in \Theta(k_n,\tau_{n})} \E_{\theta} \norm{\etahatH(y, \lambda_n) - \theta}_{2}^{2} \\
    =& ~ n\sigma_n^2 \left[ \epsilon_{n}\nu_n^{2} - 2\epsilon_{n}\nu_n \sqrt{2\log \nu_n} + o (\epsilon_n\nu_n \sqrt{\log \nu_n}) \right] \nonumber \\
    =& ~ n \sigma_n^2\epsilon_n \left( 2\log\epsilon_n^{-1} - (2+o(1)) \nu_{n} \sqrt{2 \log \nu_{n}}\right). 
\end{align*}
This completes our proof of the upper bound in Theorem \ref{thm::regime_4}.

The sharp upper bound we have derived is from the hard thresholding estimator $\hat{\eta}_H(y,\lambda_n)$ with tuning $\lambda_n=\sigma_n\nu_n$. To shed more light on the performance of hard thresholding, we provide a discussion on the optimal choices of $\lambda_n$. The lemma below characterizes the possible choices of $\lambda_n$ that leads to optimal supremum risk (up to second order).

\begin{lemma} \label{lem::regime_3-hard-thresholding-lambda-choice}
    Consider model \eqref{model::gaussian_sequence}, and parameter space \eqref{eq::parameter-space-SNR} under Regime (\rom{3}), in which $\epsilon_{n}  \rightarrow 0$, $\mu_{n} \rightarrow \infty$, $\mu_{n} = \omega(\sqrt{\log \epsilon_{n}^{-1}})$, as $n \rightarrow \infty$. Let $\nu_{n} = \sqrt{2\log \epsilon_{n}^{-1}}$. Consider the tuning regime $\lambda_n\sigma_n^{-1}\rightarrow \infty$ and $\lambda_n\sigma_n^{-1}\leq \mu_n$. If $\lambda_n$ satisfies:
    \begin{equation*}
        (\nu_{n}^{2} - c_{1} \log \log \nu_{n}) \leq \lambda_n^{2}\sigma^{-2}_n \leq (\nu_{n}^{2} + c_{2} \nu_{n} \sqrt{2\log \nu_{n}})
    \end{equation*}
    when $n$ is large, for some constant $c_{1}<1$ and every $c_{2}>0$, then
    \begin{align}
        & \inf_{\lambda} \sup_{\theta \in \Theta(k_n,\tau_{n})} \E_{\theta} \norm{\etahatH(y, \lambda) - \theta}_{2}^{2} \nonumber \\
        =~ & \sup_{\theta \in \Theta(k_n,\tau_{n})} \E_{\theta} \norm{\etahatH(y, \lambda_{n}) - \theta}_{2}^{2} + o\left( n\sigma_{n}^{2}\epsilon_{n} \nu_{n} \sqrt{ \log \nu_{n}}  \right).\label{eq:sec:or:opt}
    \end{align}
    On the other hand, if $(\nu_{n}^{2} - c_{1} \log \log \nu_{n}) \geq \lambda_n^{2}\sigma^{-2}_n$ for a constant $c_1\geq 1$ or if $\lambda_n^{2}\sigma^{-2}_n \geq (\nu_{n}^{2} + c_{2} \nu_{n} \sqrt{2\log \nu_{n}})$ for some $c_2>0$, then the conclusion \eqref{eq:sec:or:opt} will not hold. 
\end{lemma}
\begin{proof}
Denote $\tilde{\lambda}_{n}=\lambda_n\sigma^{-1}_n$. Given that we focus on the tuning regime $\tilde{\lambda}_{n}\rightarrow \infty$ and $\tilde{\lambda}_{n}\leq \mu_n$, the result \eqref{eq::hard-threshold-sup-risk} continues to hold here:
\begin{align*}
& \sup_{\theta \in \Theta(k_n,\tau_{n})} \E_{\theta} \norm{\etahatH(y, \lambda_n) - \theta}_{2}^{2} \\
=& ~n\sigma_n^2\cdot \bigg[(1-\epsilon_{n})r_{H}(\tilde{\lambda}_{n},0) + \epsilon_{n} \Big(\tilde{\lambda}_{n}^2 - 2\tilde{\lambda}_{n} \sqrt{2\log \tilde{\lambda}_{n}} \nonumber \\
+& o(\tilde{\lambda}_{n} \sqrt{\log \tilde{\lambda}_{n}}) \Big)\bigg].
\end{align*}
Hence, we define 
\begin{align}
    A(\lambda) := & (1-\epsilon_{n})r_{H}(\lambda,0) + \epsilon_{n} \Big[\lambda^2 \nonumber \\
    & - 2\lambda \sqrt{2\log \lambda} +o\left(\lambda \sqrt{\log \lambda}\right)\Big], \label{eq::hard-threshold-risk-of-lambda}
\end{align}
where the notation $o(\cdot)$ is understood as $\lambda\rightarrow \infty$. We proved before that $A(\nu_n)=\epsilon_n(\nu_n^2-(2+o(1))\nu_n\sqrt{2\log\nu_n})$. Now we consider four different regions for $\tilde{\lambda}_{n}$ (when $n$ is large):
\begin{itemize}
    \item Case  $\tilde{\lambda}_n^{2} \leq \nu_{n}^{2} - 2 c\log(\nu_{n}/\sqrt{2\pi})$ for some constant $c >1 $. Equation \eqref{eq::hard-threshold-risk-at-zero} implies
\begin{align*}
     & A(\tilde{\lambda}_{n}) \geq (1-\epsilon_{n}) r_{H}(\tilde{\lambda}_{n}, 0) \\
     \geq ~&  (1-\epsilon_{n}) r_{H}\left(\left(\nu_{n}^{2} - 2 c \log (\nu_{n}/2\pi)\right)^{1/2}, 0\right) \\
    =~ & (2+o(1)) \left(\nu_{n}^{2} - 2 c \log (\nu_{n}/2\pi)\right)^{1/2} \\
    & \cdot \phi\left(\left(\nu_{n}^{2} - 2 c \log (\nu_{n}/\sqrt{2\pi})\right)^{1/2}\right) \\
    = ~& \frac{2+o(1)}{\sqrt{2\pi}}\nu_{n} \exp\left(- \frac{\nu_{n}^{2} - 2c \log \frac{\nu_{n}}{\sqrt{2\pi}}}{2}\right) \\
    = ~&  \Uptheta\left( \epsilon_{n} (\nu_{n})^{1+c}\right).
\end{align*}
Note that $A(\tilde{\lambda}_n)=\omega(A(\nu_n))$, and hence $\tilde{\lambda}_n$ does not satisfy \eqref{eq:sec:or:opt}.

\item Case $\nu_{n}^{2} - 2 c_{1} \log(\nu_{n}/\sqrt{2\pi}) \leq \tilde{\lambda}_n^{2} \leq \nu_{n}^{2} - c_{2} \log \log \nu_{n}$ for any constant $c_{1}\leq 1$ and some constant $c_{2} \geq 1$. Since $\tilde{\lambda}_n^{2} \leq \nu_{n}^{2} - c_{2} \log \log \nu_{n}$, the same argument as in the previous case gives
\begin{align}
\label{part:one:order}
   (1-\epsilon_n) r_{H}(\tilde{\lambda}_{n}, 0) \geq 
     \frac{2+o(1)}{\sqrt{2 \pi}}\left(\epsilon_{n} \nu_{n} \left(\sqrt{\log \nu_{n}}\right)^{c_{2}}\right).
\end{align}
Moreover, using the upper and lower bounds we set for  $\tilde{\lambda}_{n}$, we obtain
\begin{align}
    & \epsilon_{n} \left(\tilde{\lambda}_{n}^{2} - 2\tilde{\lambda}_{n} \sqrt{2\log \tilde{\lambda}_{n}} + o\left(\tilde{\lambda}_{n}\sqrt{\log \tilde{\lambda}_{n}} \right)\right) \nonumber \\
    \geq~ & \epsilon_{n} \left[ \nu_{n}^{2} - 2c_{1}\log \frac{\nu_{n}}{\sqrt{2\pi}} - 2\nu_{n} \sqrt{2\log \nu_{n}} \right. \nonumber \\
    & + \left. o\left(\nu_{n} \sqrt{\log \nu_{n}}\right)\right] \nonumber \\
    = ~& \epsilon_{n} \left[\nu_{n}^{2} - 2\nu_{n} \sqrt{2\log \nu_{n}} + o\left(\nu_{n} \sqrt{\log \nu_{n}}\right)\right].\label{part:two:order}
\end{align}
Combining \eqref{part:one:order}-\eqref{part:two:order} yields
\begin{align*}
A(\tilde{\lambda}_n) & \geq \epsilon_n\Big[\nu_n^2+\nu_n\sqrt{2\log\nu_n}\Big(-2+o(1) \\
& +\frac{2+o(1)}{2\sqrt{\pi}}(\sqrt{\log\nu_n})^{c_2-1}\Big)\Big].
\end{align*}
Since $c_2\geq 1$, it is clear that $A(\tilde{\lambda}_n)-A(\nu_n)=\Omega(\epsilon_n\nu_n\sqrt{\log\nu_n})$. Therefore, this choice of $\tilde{\lambda}_{n}$ does not satisfy \eqref{eq:sec:or:opt}.

\item Case $\nu_{n}^{2} - c_{1} \log \log \nu_{n} \leq \tilde{\lambda}_n^{2} \leq \nu_{n}^{2} + c_{2}\nu_{n}\sqrt{2\log \nu_{n}}$ for some constant $c_{1}<1$ and every $c_{2} > 0$. With the lower bound of $\tilde{\lambda}_{n}$, similar calculations as in the previous two cases lead to
\begin{align*}
    & (1-\epsilon_{n})r_{H}(\tilde{\lambda}_{n}, 0) \leq r_{H}\left(\left(\nu_{n}^{2} - c_{1} \log \log \nu_{n}\right)^{1/2}, 0\right) \\
    =& ~ \Uptheta\left(\epsilon_{n}\nu_{n}\left(\sqrt{\log \nu_{n}}\right)^{c_{1}}\right).
\end{align*}
Furthermore, the upper and lower bounds of $\tilde{\lambda}_n$ for some $c_1<1$ and every $c_2>0$ imply that $\tilde{\lambda}_n^2-\nu_n^2=o(\nu_n\sqrt{\log\nu_n})$. Thus,
\begin{align*}
     & \epsilon_{n} \left(\tilde{\lambda}_{n}^{2} - 2\tilde{\lambda}_{n}\sqrt{2\log \tilde{\lambda}_{n}} + o\left(\tilde{\lambda}_{n}\sqrt{2\log \tilde{\lambda}_{n}}\right)\right)  \\
     \leq ~ & \epsilon_{n}\left(\nu_{n}^{2} - 2\nu_{n}\sqrt{2\log \nu_{n}} + o\left( \nu_{n}\sqrt{\log \nu_{n}}\right)\right).
\end{align*}
Putting together the above two results into \eqref{eq::hard-threshold-risk-of-lambda}, we have
\begin{align*}
    A(\tilde{\lambda}_{n})
    & \leq \Uptheta\left(\epsilon_{n}\nu_{n}\left(\sqrt{\log \nu_{n}}\right)^{c_{1}}\right)  \\
    & + \epsilon_{n}\left(\nu_{n}^{2} - 2\nu_{n}\sqrt{2\log \nu_{n}} + o\left( \nu_{n}\sqrt{\log \nu_{n}}\right)\right) \\
    & = \epsilon_{n}\left(\nu_{n}^{2} - (2+o(1))\nu_{n}\sqrt{2\log \nu_{n}} \right).
\end{align*}
Thus, $A(\tilde{\lambda}_n)\leq A(\nu_n)+o(\epsilon_n\nu_n\sqrt{\log\nu_n})$, and $\tilde{\lambda}_{n}$ satisfies \eqref{eq:sec:or:opt}.

\item Case $\tilde{\lambda}_n^{2} \geq \nu_{n}^{2} + c\nu_{n}\sqrt{2\log \nu_{n}}$ for some constant $c > 0$. We only need consider $\tilde{\lambda}_n=(1+o(1))\nu_n$, because for larger values of $\lambda_n$, \eqref{eq::hard-threshold-risk-of-lambda} implies that $A(\tilde{\lambda}_n)/A(\nu_n)>1$ for large $n$. When $\tilde{\lambda}_n=(1+o(1))\nu_n$, we have
\begin{align*}
A(\tilde{\lambda}_n) & \geq    \epsilon_{n} \left(\tilde{\lambda}_{n}^{2} - 2\tilde{\lambda}_{n}\sqrt{2\log \tilde{\lambda}_{n}} \right. \\
& + \left. o\left(\tilde{\lambda}_{n}\sqrt{2\log \tilde{\lambda}_{n}}\right)\right) \\
&\geq \epsilon_{n}\left(\nu_{n}^{2} - (2-c)\nu_{n}\sqrt{2\log \nu_{n}} \right. \\
& + \left. o\left( \nu_{n}\sqrt{2\log \nu_{n}}\right) \right).
\end{align*}
Since $c>0$, the above implies that $A(\tilde{\lambda}_n)-A(\nu_n)=\Omega(\epsilon_n\nu_n\sqrt{\log\nu_n})$. Hence $\tilde{\lambda}_n$ does not satisfy \eqref{eq:sec:or:opt}.

\end{itemize}
\end{proof}

\subsubsection{Lower bound}

As in the proof of lower bound in Theorems \ref{thm::regime_1}-\ref{thm::regime_2_compactness}, we will apply Theorem \thmMinimaxLower\ and utilize the independent block prior that is first described in \secRegimeILower{}. To simplify the calculations a bit here, we will use the block prior with one minor modification: adopting the notation from \secRegimeILower, the spike prior $\pi_S^{\mu,m}$ in use is now changed to a one-sided spike prior:
\begin{align}
\label{one:side:spike}
    \pi_S^{\mu,m}(\theta^{(j)}=\mu e_i)=\frac{1}{m},\quad 1\leq i \leq m,
\end{align}
where $\mu\in (0,\mu_n]$. The key is to calculate the Bayes risk $B( \pi_S^{\mu,m})$ and obtain a result like Lemma \lemSpikeBayes{}. To this end, we first mention a lemma that will become useful later in the proof.

\begin{lemma}\label{lem::F2-complement}
Let $z_{1}, \ldots, z_{m} \simiid \mathcal{N}(0,1)$ and $\nu_{m} = \sqrt{2\log m}$. Suppose $2\mu > \nu_{m}$ and $\delta < \Phi(\nu_{m}-\mu)$. Then
\begin{align*}
    & \mathbb{P}\left(m^{-1} e^{-\frac{1}{2}\mu^{2}} \sum_{j=1}^{m} e^{\mu z_{j}}
    \leq \delta \right) \\
    \leq ~ & \frac{1}{\sqrt{2\pi}\nu_{m}}  + \frac{1}{\sqrt{2\pi}} \frac{1}{\left[\Phi(\nu_{m}-\mu)-\delta \right]^{2}} \frac{1}{2\mu-\nu_{m}} e^{-(\mu-\nu_{m})^{2}}.
\end{align*}
\end{lemma}
\begin{proof}
Define the notation: 
\begin{align*}
& X_{mj} = e^{\mu z_{j}}, \quad \Bar{X}_{mj} = X_{mj} I_{(X_{mj}\leq e^{\mu\nu_m})},\\
&S_{m}=\sum_{j=1}^mX_{mj},\quad \bar{S}_m=\sum_{j=1}^m\bar{X}_{mj},\\
& a_{m} = \E\bar{S}_m = me^{\mu^{2}/2} \Phi(\nu_{m} - \mu).
\end{align*}
Then 
\begin{align*}
    & \mathbb{P}\Big(m^{-1} e^{-\frac{1}{2}\mu^{2}} \sum_{j=1}^{m} e^{\mu z_{j}} \leq \delta \Big) \\
    = ~ & \mathbb{P}\left \{a_{m}-S_{m} \geq \left[ \Phi(\nu_{m} - \mu)-\delta  \right] \cdot m e^{\frac{1}{2}\mu^{2}} \right \} \\
    = ~ & \mathbb{P}\left(\frac{a_{m} - S_{m}}{e^{\mu \nu_m}} \geq t\right), 
\end{align*}
where $t := \left[\Phi(\nu_{m} - \mu)-\delta  \right] \cdot m e^{\frac{1}{2}\mu^{2} - \mu \nu_{m}}$. Clearly,
\begin{equation*}
    \mathbb{P}\left(\frac{a_{m} - S_{m}}{e^{\mu \nu_m}} \geq t\right) \leq \mathbb{P} \left( S_{m} \neq \Bar{S}_{m} \right) + \mathbb{P}\left( \left| \frac{\Bar{S}_{m} - a_{m}}{e^{\mu\nu_m}} \right| > t \right).
\end{equation*}
For the following calculation, we will use Gaussian tail bound $1-\Phi(x) \leq x^{-1} \phi(x)$ for $x>0$. To obtain a proper upper bound for the first term, we note that
\begin{align*}
     & \mathbb{P} \left( S_{m} \neq \Bar{S}_{m} \right) \leq  \mathbb{P} \left( \cup_{j=1}^{m} \{ \Bar{X}_{mj} \neq X_{mj} \} \right) \\
     \leq & \sum_{j=1}^{m} \mathbb{P} \left( X_{mj} > e^{\mu\nu_m} \right)
     = \sum_{j=1}^{m} \mathbb{P} \left( e^{\mu z_{j}} > e^{\mu \nu_{m}} \right) \\
     =&  ~ m(1- \Phi(\nu_{m})) \leq   \frac{m}{\nu_{m}} \phi(\nu_{m}) =  \frac{1}{\sqrt{2\pi}\nu_{m}}.
\end{align*}
For the second term, we use Chebyshev's inequality and the fact that $a_{m} = \E \Bar{S}_{m}$ and $\Var (X) \leq \E X^{2}$,
\begin{align*}
    & \mathbb{P}\left( \left| \frac{\Bar{S}_{m} - a_{m}}{e^{\mu\nu_m}} \right| > t \right)
    \leq  t^{-2}e^{-2\mu\nu_m} \E  (\Bar{S}_{m} - a_{m})^{2} \\
    \leq & ~ (te^{\mu\nu_m})^{-2} \sum_{j=1}^{m} \E \Bar{X}_{mj}^{2} \\
    \leq & ~ \frac{1}{\left[\Phi(\nu_{m}-\mu)- \delta  \right]^{2}} \frac{1}{\sqrt{2\pi}} \frac{1}{2\mu - \nu_{m}} e^{-(\mu -\nu_{m})^{2}}.
\end{align*}
The last inequality is based on the following calculation:
\begin{align*}
     & \E \Bar{X}_{mj}^{2} = \E \left( e^{\mu z_{j}} I_{(e^{\mu z_{j}} \leq e^{\mu \nu_{m}})} \right)^{2} = \int_{z\leq \nu_{m}} e^{2\mu z} \phi(z) d z \\
     & =  e^{2\mu^{2}} (1-\Phi (2\mu - \nu_{m}) )
     \leq \frac{1}{\sqrt{2\pi}} \frac{1}{2\mu-\nu_{m}} e^{2\mu^{2} - \frac{1}{2}(2\mu-\nu_{m})^{2}} \\
     & = \frac{1}{\sqrt{2\pi}} \frac{1}{2\mu-\nu_{m}} e^{-\frac{1}{2}\nu_{m}^{2} + 2\mu \nu_{m}},
\end{align*}
and
% and $\nu_{n}-\mu \rightarrow \infty$
\begin{align*}
    & (te^{\mu\nu_m})^{-2} m \cdot \frac{1}{\sqrt{2\pi}} \frac{1}{2\mu-\nu_{m}} e^{-\frac{1}{2}\nu_{m}^{2} + 2\mu \nu_{m}} \\
    = & ~ \frac{1}{\left[\Phi(\nu_{m}-\mu) -\delta\right]^{2} } \frac{1}{m^{2}} e^{-\mu^{2}} m \frac{1}{\sqrt{2\pi}} \frac{1}{2\mu-\nu_{m}} e^{-\frac{1}{2}\nu_{m}^{2}+2\mu \nu_{m}} \\
    = & ~ \frac{1}{\left[\Phi(\nu_{m}-\mu) -\delta\right]^{2} } \frac{1}{\sqrt{2\pi}} \frac{1}{2\mu-\nu_{m}} e^{-(\mu-\nu_{m})^{2}}.
\end{align*}
\end{proof}

We are now ready to calculate the Bayes risk $B( \pi_S^{\mu,m})$ in the following lemma. 

\begin{lemma}\label{lem::single-one-sided-spike}
Let $\nu_{m}= \sqrt{2 \log m}$ and $\mu = \nu_{m-1} - \sqrt{2 \log \nu_{m-1}}$. As $m\rightarrow\infty$, the Bayes risk $B( \pi_S^{\mu,m})$ satisfies
\begin{equation*}
    B( \pi_S^{\mu,m}) \geq \nu_{m}^{2} - 2\nu_{m} \sqrt{2\log \nu_{m}} \left( 1 + o(1)\right).
\end{equation*}
\end{lemma}

\begin{proof}
For the one-sided spike prior $\pi_S^{\mu,m}$ introduced in \eqref{one:side:spike}, doing similar calculations as in the proof of Lemma \lemSpikeBayesDecompose, we can obtain the expression for the Bayes risk:
\begin{align}\label{eq::bayes-risk-decompose}
      B( \pi_S^{\mu,m}) & = \mu^{2} \E_{\mu e_{1}} ( p_{m} - 1 )^{2} + (m-1) \mu^{2} \E_{\mu e_{2}} p_{m}^{2} \nonumber \\ 
      \geq & \mu^{2} - 2 \mu^{2}\E_{\mu e_{1}} p_{m},
\end{align}
where $p_{m} = \frac{e^{\mu y_1}}{\sum_{j=1}^me^{\mu y_j}}$; $\E_{\mu e_1}(\cdot)$ is taken with respect to $y\sim \mathcal{N}(\mu e_1,I)$ and $\E_{\mu e_2}(\cdot)$ for $y\sim \mathcal{N}(\mu e_2,I)$. Now the goal is to upper bound $\E_{\mu e_{1}} p_{m}$. We have
\begin{align}
\label{inter:med}
    & \E_{\mu e_1} p_{m} = \E\frac{e^{\mu(\mu + z_{1})}}{\sum_{j\neq 1} e^{\mu z_{j}} + e^{\mu (\mu+z_{1})}} \nonumber `\\
    = & ~ \E\frac{ (m-1)^{-1}e^{\frac{1}{2}\mu^{2}+\mu z_{1}}}{(m-1)^{-1}e^{\frac{1}{2}\mu^{2}+\mu z_{1}} + (m-1)^{-1}e^{-\frac{1}{2}\mu^{2}} \sum_{j\neq 1} e^{\mu z_{j}} },
\end{align}
where $z_1,\ldots, z_m\overset{i.i.d.}{\sim}\mathcal{N}(0,1)$. Define the following two events:
\begin{align*}
     & \mathcal{F}_{1} = \Big\{ (m-1) e^{-\frac{1}{2}\mu^{2} - \mu z_{1}} \geq M \Big\}, \\
     & \mathcal{F}_{2} = \Big\{ (m-1)^{-1} e^{-\frac{1}{2}\mu^{2}} \sum_{j\neq 1} e^{\mu z_{j}} \geq \delta \Big\},
\end{align*}
where $\delta$ and $M$ are two positive constants to be determined later. Since the ratio inside the expectation of \eqref{inter:med} is smaller than one, and on the event $\mathcal{F}_{1}\cap \mathcal{F}_{2}$ it is smaller than $\frac{1}{M\delta}$, we can continue from \eqref{inter:med} to obtain
\begin{equation}\label{eq::expectation-p1n}
    \E_{\mu e_{1}} p_{m} \leq \frac{1}{M \cdot \delta} + \mathbb{P}(\mathcal{F}_{1}^{c}) + \mathbb{P}(\mathcal{F}_{2}^{c}).
\end{equation}
Hence, we aim to find upper bounds for $\mathbb{P}(\mathcal{F}_{1}^{c})$ and $\mathbb{P}(\mathcal{F}_{2}^{c})$.
For the first probability, using Gaussian tail bound that $1-\Phi(x) \leq \frac{1}{x}\phi(x)$ for $x>0$, and that $e^{\nu_{m-1}^{2}/2} = m-1$, we have
\begin{align*}
     & \mathbb{P}(\mathcal{F}_{1}^{c}) = ~\mathbb{P}\left( (m-1) e^{-\frac{1}{2}\mu^{2} -\mu z} < M \right) \\
    = & ~ \mathbb{P}\left( z> -\frac{1}{2}\mu - \frac{1}{\mu} \log \frac{M}{m-1} \right) \\
    = & ~ 1-\Phi\left( -\frac{1}{\mu} \log M + \frac{1}{2\mu} (\nu_{m-1}^{2} - \mu^{2}) \right) \\
    \leq & ~ \frac{1}{ -\frac{1}{\mu} \log M + \frac{1}{2\mu} (\nu_{m-1}^{2} - \mu^{2})} \frac{1}{\sqrt{2\pi}} \\
    \cdot & \exp\left(-\frac{1}{2\mu^{2}} \left[ \frac{1}{2}(\nu_{m-1}^{2} - \mu^{2}) - \log M \right]^{2} \right) := U_{1},
    \end{align*}
    as long as $\nu_{m-1}^2-\mu^2>2\log M$.
Regarding $\mathbb{P}(\mathcal{F}_{2}^{c})$, if we limit our choice of $0 < \delta < \Phi(\nu_{m-1}-\mu)$, then from Lemma \ref{lem::F2-complement},
\begin{align*}
    \mathbb{P}(\mathcal{F}_{2}^{c}) & \leq ~ \frac{1}{\sqrt{2\pi}} \frac{1}{\nu_{m-1}} + \frac{1}{\sqrt{2\pi}} \frac{1}{\left[\Phi(\nu_{m-1}-\mu)-\delta\right]^{2}} \\
    & \cdot \frac{1}{2\mu-\nu_{m-1}} e^{-(\mu-\nu_{m-1})^{2}} := U_{2}.
\end{align*}

Now we set $M=\nu_{m-1}$ and recall $\mu=\nu_{m-1} - \sqrt{2\log \nu_{m-1}}$. We will show that $U_{1} = o(\nu_{m-1}^{-1})$ and $U_{2}=O(\nu_{m-1}^{-1})$. First, for $U_{1}$,
\begin{align*}
    & \frac{1}{2\mu^{2}} \left[\frac{1}{2}(\nu_{m-1}^{2} - \mu^{2}) - \log M \right]^{2} \\
    = & ~ \frac{1}{2\mu^{2}} \left[\frac{1}{2}(2 \nu_{m-1} \sqrt{2\log\nu_{m-1}} - 2\log \nu_{m-1}) - \log \nu_{m-1} \right]^{2} \\
    = & ~ \frac{1}{2\mu^{2}} \left[ \nu_{m-1}\sqrt{2\log\nu_{m-1}} - 2\log\nu_{m-1} \right]^{2} \\
    = & ~ \frac{\nu_{m-1}^{2}}{\mu^{2}} \log \nu_{m-1} - \frac{2\sqrt{2}\nu_{m-1}}{\mu^{2}} \left( \log \nu_{m-1} \right)^{3/2} \\
    & + \frac{2}{\mu^{2}} \left( \log \nu_{m-1} \right)^{2} \geq  \log \nu_{m-1} + o(1),
\end{align*}
where in the last inequality we used $\mu^2<\nu_{m-1}^2$ (for large $m$).
Therefore, 
\begin{equation*}
    e^{-\frac{1}{2\mu^{2}} \left[ \frac{1}{2}(\nu_{m-1}^{2} - \mu^{2}) - \log M \right]^{2} } \leq \nu_{m-1}^{-1} \left( 1+o(1) \right),
\end{equation*}
and
\begin{align*}
    & \frac{1}{ -\frac{1}{\mu} \log M + \frac{1}{2\mu} (\nu_{m-1}^{2} - \mu^{2})} \\
    = & ~ \frac{1}{\frac{1}{\mu} \cdot \left( \nu_{m-1}\sqrt{2\log\nu_{m-1}} - 2\log\nu_{m-1} \right)} \\
    \leq & ~ \left( \sqrt{2\log \nu_{m-1}} - \frac{2\log \nu_{m-1}}{\nu_{m-1}} \right)^{-1} = o(1).
\end{align*}
In combination,
\begin{equation}\label{eq::U1}
    U_{1} \leq o(1) \cdot \nu_{m-1}^{-1} \left( 1+o(1) \right)  = o(\nu_{m-1}^{-1}).
\end{equation}
For $U_{2}$, we set $\delta$ to be any fixed constant between $(0,1)$. Since $\nu_{m-1} - \mu \rightarrow +\infty$, it holds that $\Phi(\nu_{m-1}-\mu)-\delta > \delta^{'}$ for some constant $\delta^{'}>0$, when $m$ is large. Also, we have the identity $e^{-(\mu-\nu_{m-1})^{2}} = e^{-2\log \nu_{m-1}} = \nu_{m-1}^{-2}$. So the second term in $U_{2}$ is of order $O(\nu_{m-1}^{-3})$. Thus,
\begin{equation}\label{eq::U2}
    U_{2} =  \frac{1+o(1)}{\sqrt{2\pi}\nu_{m-1}}.
\end{equation}
Note that we have set $M=\nu_{m-1}$. Hence, $1/(M\cdot \delta) = O(1/\nu_{m-1})$. Combining \eqref{eq::expectation-p1n}-\eqref{eq::U2}, we have
\begin{equation*}
    \E_{\mu e_{1}} p_{m} \leq O(1/\nu_{m-1}).
\end{equation*}
Finally, the above together with \eqref{eq::bayes-risk-decompose} shows that
\begin{align*}
    B(\pi^{\mu,m}_S) & \geq \mu^{2} - 2\mu^{2}  O\left(\nu_{m-1}^{-1}\right)  \\
    & = \nu_{m-1}^{2} - 2\nu_{m-1}\sqrt{2\log \nu_{m-1}} \left( 1+o(1) \right) \\
    & = \nu_{m}^{2} - 2\nu_{m}\sqrt{2\log \nu_{m}} \left( 1+o(1) \right).
\end{align*}
\end{proof}

Now, we aim to apply Lemma \ref{lem::single-one-sided-spike} to derive the minimax lower bound. First note that in the current regime $\epsilon_n\rightarrow 0,\mu_n=\omega(\sqrt{\log\epsilon_n^{-1}})$, the choice of $\mu$ with $m=n/k_n=\epsilon_n^{-1}$ in Lemma \ref{lem::single-one-sided-spike} satisfies $\mu<\mu_n$ when $n$ is large. Thus, the constructed block prior is supported on the parameter space $\Theta(k_n,\mu_n)$ so that we can use \EqBlockPriorLower\ and Lemma \ref{lem::single-one-sided-spike} to conclude
\begin{align*}
    & R(\Theta(k_n,\tau_{n}),\sigma_n)= ~\sigma_n^2\cdot R(\Theta(k_n,\mu_{n}),1) \geq k_n\sigma_n^2\cdot B(\pi_S^{\mu,m}) \\
     \geq & ~ k_n\sigma_n^2 \cdot \Big(\nu_m^2-2\nu_m\sqrt{2\log\nu_m}(1+o(1)\Big) \\
     = & ~ n\sigma_n^2\Big(2\epsilon_n\log\epsilon_n^{-1}-2\epsilon_n\nu_m\sqrt{2\log\nu_m}(1+o(1)\Big),
\end{align*}
where $\nu_m=\sqrt{2\log m}=\sqrt{2\log\epsilon_n^{-1}}$.
\subsection{Proof of Proposition \ref{lem::lasso-risk-regime-4}} \label{sec::lasso-risk-regime_3}

\subsubsection{Roadmap of the proof}

Propositions \ref{lem::lasso-risk-regime-1} and \ref{lem::lasso-risk-regime-2} have derived the supremum risk of optimally tuned soft thresholding in Regimes (\rom{1}) and (\rom{2}) respectively. Proposition \ref{lem::lasso-risk-regime-4} continues to obtain it in Regime (\rom{3}). Hence, we will use some existing results from the proof of Propositions \ref{lem::lasso-risk-regime-1} and \ref{lem::lasso-risk-regime-2} to simplify the present proof. First of all, referring to Equations \eqref{eq::soft-thresholding-risk-scalability}-\eqref{eq::soft-thresholding-risk-function} in the proof of Proposition \ref{lem::lasso-risk-regime-1}, the supremum risk can be expressed as
\begin{align*}
  & \inf_{\lambda}\sup_{\theta\in \Theta(k_n,\tau_{n})} \E_{\theta} \norm{\etahatS(y,\lambda) - \theta}_{2}^{2} \\
  = ~& n\sigma_n^2\cdot \inf_{\lambda}\underbrace{\Big[(1-\epsilon_n)\E\hat{\eta}^2_S(z,\lambda)+\epsilon_n\E(\hat{\eta}_S(z+\mu_n,\lambda)-\mu_n)^2\Big]}_{:=F(\lambda)},
\end{align*}
with $z\sim \mathcal{N}(0,1)$.
Define the optimal tuning $\lambdas=\argmin_{\lambda\geq 0}F(\lambda)$. Then it is equivalent to prove
\[
F(\lambdas)=2\epsilon_n\log\epsilon_n^{-1}-(6+o(1))\epsilon_n\log\nu_n,
\]
where $\nu_n=\sqrt{2\log \epsilon_n^{-1}}$. To reach the above, we will first find the tight upper bound for $F(\lambdas)$ in Section \ref{sec::upper-bound-lasso-sup-risk}, and then obtain the matching lower bound in Section \ref{ssec:lower:softhrsh}. Before we do these two parts, let us prove a lemma that provides an approximation for $F(\lambda)$. This approximation will help us in the calculation of both the upper and lower bounds.

\begin{lemma}\label{lem::analysis-second-order-lasso-risk-regime-4}
Consider $\epsilon_{n}\rightarrow 0$, $\mu_{n}=\omega(\sqrt{\log \epsilon_{n}^{-1}})$, as $n\rightarrow \infty$. If $\lambda \rightarrow \infty$ and $\mu_{n} - \lambda \rightarrow +\infty$, then
\begin{align*}
    F(\lambda) & = ~ 2(1-\epsilon_{n}) \left[ (1+\lambda^{2}) (1-\Phi(\lambda)) - \lambda \phi(\lambda) \right]  \\
    & +  \epsilon_{n} \left[ \lambda^{2}  + 1 - \frac{(2+o(1))\mu_{n}}{(\mu_{n}-\lambda)^{2}} \phi(\mu_{n}-\lambda) \right].
\end{align*}
Furthermore, when $\lambda$ is large, it holds that
\begin{equation*}
    C(\lambda) \leq F(\lambda) \leq D(\lambda),
\end{equation*}
where 
\begin{align}\label{eq::C(lambda)}
    C(\lambda) & := 2(1-\epsilon_{n}) \cdot \left( \frac{2}{\lambda^{3}} - \frac{12}{\lambda^{5}} \right) \frac{1}{\sqrt{2\pi}} \epsilon_{n} \cdot e^{\frac{1}{2}(\nu_{n}^{2} - \lambda^{2})} \\
    & +  \epsilon_{n}\left[ \lambda^{2}  + 1- \frac{(2+o(1))\mu_{n}}{(\mu_{n}-\lambda)^{2}}\phi(\mu_{n}-\lambda) \right]  , \nonumber
\end{align}
and 
\begin{equation}\label{eq::D(lambda)}
    D(\lambda) :=  \epsilon_{n} \left\{ (1-\epsilon_{n}) \frac{4}{\sqrt{2\pi}\lambda^{3}}e^{\frac{1}{2}(\nu_{n}^{2}-\lambda^{2})}  +\lambda^{2}+1 \right\} .
\end{equation}
\end{lemma}

\begin{proof}

Throughout the proof, we will use the Gaussian tail bound in \lemGaussianMill\ to do calculations. With the expression of $F(\lambda)$ calculated in \eqRiskOfLambdatoSoftSecondMom, we have that as $\lambda\rightarrow \infty,\mu_n-\lambda\rightarrow +\infty$,
\begin{align*}
    F(\lambda) & = 2(1-\epsilon_{n}) \cdot  \left[ (1+\lambda^{2}) (1-\Phi(\lambda)) - \lambda \phi(\lambda) \right] \\
    & + \epsilon_{n} \cdot \bigg\{ (\lambda^{2}+1) + \Big[ (\mu_{n}^{2} - \lambda^{2} -1) (1-\Phi(\mu_{n}-\lambda))  \\
    & - (\mu_{n}+\lambda) \phi(\mu_{n}-\lambda) \Big] - \Big[ (\mu_{n}^{2} - \lambda^{2} - 1) \\
    & \cdot (1-\Phi(\mu_{n}+\lambda))   - (\mu_{n}-\lambda)\phi(\mu_{n}+\lambda) \Big]  \bigg\} \\
    & =  2(1-\epsilon_{n}) \left[ (1+\lambda^{2}) (1-\Phi(\lambda)) - \lambda \phi(\lambda) \right] \\
    & +  \epsilon_{n} \left[ \lambda^{2}  + 1 - \frac{(2+o(1))\mu_{n}}{(\mu_{n}-\lambda)^{2}} \phi(\mu_{n}-\lambda) \right] ,
\end{align*}
where in the last equation we have used $1-\Phi(x)=\left(\frac{1}{x}-\frac{1+o(1)}{x^3}\right)\phi(x)$ as $x\rightarrow \infty$.

As $\lambda\rightarrow \infty$, we obtain
\begin{align*}
    & (1+\lambda^{2}) (1-\Phi(\lambda)) - \lambda \phi(\lambda) \\
    = ~& \left[ (1+\lambda^{2}) \left( \frac{1}{\lambda} - \frac{1}{\lambda^{3}} + \frac{3}{\lambda^{5}} - \frac{15}{\lambda^{7}} + \frac{105}{\lambda^{9}} \right) -\lambda \right] \phi(\lambda ) \\
    + ~& O\left(\frac{\phi(\lambda)}{\lambda^9}\right) 
    = \left( \frac{2}{\lambda^{3}} -\frac{12}{\lambda^{5}} + \frac{90}{\lambda^{7}} \right)\phi(\lambda)+ O\left(\frac{\phi(\lambda)}{\lambda^9}\right).
\end{align*}
Thus, 
\begin{align*}
    F(\lambda)  & =  2(1-\epsilon_{n}) \cdot  \left( \frac{2}{\lambda^{3}} - \frac{12}{\lambda^{5}} + \frac{90}{\lambda^{7}} + O\left( \frac{1}{\lambda^{9}}\right) \right) \\
    & \cdot \frac{1}{\sqrt{2\pi}} \epsilon_{n} \cdot e^{\frac{1}{2}(\nu_{n}^{2}-\lambda^{2})}  +  \epsilon_{n} \Big[ \lambda^{2}  + 1  \\
    & -  \frac{(2+o(1))\mu_{n}}{(\mu_{n}-\lambda)^{2}} \phi(\mu_{n}-\lambda) \Big].
\end{align*}
As a result, it is straightforward to verify that $C(\lambda)$ and $D(\lambda)$ defined in \eqref{eq::C(lambda)}-\eqref{eq::D(lambda)} provide lower and upper bounds for $F(\lambda)$.
\end{proof}

%%%%%%%%%%%% the part to be editied
\subsubsection{Upper bound}\label{sec::upper-bound-lasso-sup-risk}
Consider $\lambda = \sqrt{\nu_{n}^{2}- 6\log \nu_{n}}$, then $\lambda \rightarrow \infty$ and $\mu_n-\lambda\rightarrow \infty$. From Lemma \ref{lem::analysis-second-order-lasso-risk-regime-4}, 
\begin{align}
    & F(\lambdas)\leq F(\lambda)\leq D(\lambda) \nonumber \\
    = ~& \epsilon_{n} \left\{ (1-\epsilon_{n}) \frac{4}{\sqrt{2\pi}} e^{\frac{1}{2} \left[(\nu_{n}^{2}-\lambda^{2}) - 6\log\lambda \right]} + \lambda^{2}+1 \right\}  \nonumber \\
    = ~& \epsilon_{n}\left\{\frac{4+o(1)}{\sqrt{2\pi}} +\lambda^{2}+1 \right\} =  \epsilon_n\nu_n^2- 6\epsilon_{n}  \log \nu_{n} \left( 1+o(1)\right).\label{bench:mark:risk}
\end{align}

\subsubsection{Lower bound}\label{ssec:lower:softhrsh}

We now derive a matching lower bound for $F(\lambdas)$. This requires a careful analysis of the order of the optimal tuning $\lambdas$. We break it down in several steps:

\textbf{\underline{Step 1}}.
First, we show that $\lambdas\rightarrow \infty, \mu_n-\lambdas\rightarrow +\infty$. We will need the following lemma.
\begin{lemma}[Lemma 8.3 in \cite{johnstone19}] \label{lem::soft-threshold-risk-nonasymptotic}
Define $r_S(\lambda,\mu)=\E(\hat{\eta}_S(\mu+z,\lambda)-\mu)^2$, and $\Bar{r}_{S}(\lambda,\mu) = \min \{ r_{S}(\lambda,0) + \mu^{2}, 1 + \lambda^{2} \}$. For all $\lambda>0$ and $\mu \in \R$,
\begin{equation*}
    \frac{1}{2}\Bar{r}_{S}(\lambda,\mu) \leq r_{S}(\lambda,\mu) \leq \Bar{r}_{S}(\lambda,\mu).
\end{equation*}
\end{lemma}
Suppose $\lambdas\rightarrow \infty$ is not true. Then $\lambdas\leq c$ for some finite constant $c\geq 0$ (take a subsequence if necessary). Then, from the definition of $F(\lambdas)$ we have
\begin{align*}
F(\lambdas)& \geq (1-\epsilon_n)\E\hat{\eta}^2_S(z,\lambdas)\geq (1-\epsilon_n)\E\hat{\eta}^2_S(z,c) \\
& =\Omega(1)=\omega(\epsilon_n\nu_n^2),
\end{align*}
which contradicts with \eqref{bench:mark:risk}. Further suppose $\mu_n-\lambdas\rightarrow +\infty$ is not true. Then $\lambdas\geq \mu_n-c$ for some finite constant $c$ (take a subsequence if necessary). From Lemma \ref{lem::soft-threshold-risk-nonasymptotic} we obtain for large $n$,
\begin{align*}
F(\lambdas) & \geq \epsilon_n r_S(\lambdas,\mu_n) \geq \frac{1}{2}\epsilon_n\min(\mu_n^2,\lambdas^2) \\
& \geq \frac{1}{4}\epsilon_n\mu_n^2 = \omega(\epsilon_{n} \nu_{n}^{2}),
\end{align*}
where we used $\mu_{n} = \omega(\sqrt{2\log \epsilon_{n}^{-1}}) = \omega( \nu_{n})$. The same contradiction arises. 

\textbf{\underline{Step 2}}.
We next claim that $\lambdas=(1+o(1))\nu_n$. Otherwise, $\lambdas=(c+o(1))\nu_n$ for some constant $c \neq 1$ (take a subsequence if necessary). For $c>1$, given that we have proved $\lambdas\rightarrow \infty,\mu_n-\lambdas\rightarrow +\infty$, we can apply Lemma \ref{lem::analysis-second-order-lasso-risk-regime-4} to reach
\begin{align*}
    F(\lambdas) & \geq \epsilon_{n} \left[ \lambdas^{2}  + 1 - \frac{(2+o(1))\mu_{n}}{(\mu_{n}-\lambdas)^{2}} \phi(\mu_{n}-\lambdas) \right] \\
    & = \epsilon_{n} \lambdas^2(1+o(1)) =(c^2+o(1))\cdot \epsilon_n\nu_n^2.
\end{align*}
This contradicts with \eqref{bench:mark:risk}. For $c<1$, we have the same contradiction by applying Lemma \ref{lem::analysis-second-order-lasso-risk-regime-4} again:
\begin{align*}
 F(\lambdas)=\frac{4+o(1)}{\lambdas^3}\phi(\lambdas)+\epsilon_n\lambdas^2(1+o(1))=\omega(\epsilon_n\nu_n^2).
\end{align*}
Here, the last inequality holds because $\lambdas\leq (1-\gamma)\nu_{n}$ for some constant $\gamma \in (0,1)$ when $n$ is large, so that
\begin{align*}
    & \frac{1}{\lambdas^{3}} e^{-\frac{\lambdas^{2}}{2}} \geq \frac{1}{(1-\gamma)^{3}\nu_{n}^{3}} e^{-\frac{(1-\gamma)^{2}}{2}\nu_{n}^{2}} \\
    & = \epsilon_{n} \frac{1}{(1-\gamma)^{3}\nu_{n}^{3}} e^{(\gamma-\frac{\gamma^{2}}{2})\nu_{n}^{2}}= \omega(\epsilon_{n} \nu_{n}^{2}).
\end{align*}

\textbf{\underline{Step 3}}. Finally, we prove that $\nu_n^2-\lambdas^2=(6+o(1))\log\nu_n$. Suppose this is not true. Then $\nu_n^2-\lambdas^2=(c+o(1))\log\nu_n$ for some $c\neq 6$ (take a subsequence if necessary). Since we have proved $\lambdas=(1+o(1))\nu_n$, we can use the lower bound in Lemma \ref{lem::analysis-second-order-lasso-risk-regime-4} and simplify it to
\begin{align}
\label{final:ref:to}
    F(\lambdas)\geq C(\lambdas)& =\frac{(4+o(1))\epsilon_n}{\sqrt{2\pi}}\frac{e^{\frac{1}{2}(\nu_n^2-\lambdas^2)}}{\lambdas^3} \nonumber \\
    & +\epsilon_n\Big(\lambdas^2+1+o(1)\Big).
\end{align}
For the case $c>6$, since
\begin{equation*}
    \frac{1}{\lambdas^{3}} e^{\frac{1}{2}(\nu_{n}^{2}-\lambdas^{2})} = e^{\frac{1}{2}(\nu_{n}^{2}-\lambdas^{2}-6\log\nu_{n}) + 3\log \frac{\nu_{n}}{\lambda}} =\nu_n^{\tilde{c}},
\end{equation*}
with $\tilde{c}=\frac{c-6+o(1)}{2}>0$, \eqref{final:ref:to} implies that
\[
F(\lambdas)\geq \Uptheta(\epsilon_n\nu_n^{\tilde{c}})+\epsilon_n\nu_n^2-(c+o(1))\epsilon_n\log \nu_n,
\]
contradicting with \eqref{bench:mark:risk}. Regarding the case $c<6$, \eqref{final:ref:to} directly leads to
\begin{align*}
    F(\lambdas)\geq \epsilon_n\nu_n^2-(c+o(1))\epsilon_n\log\nu_n+(1+o(1))\epsilon_n.
\end{align*}
No mater what value $c\in [-\infty,6)$ takes, the above lower bound is larger than the upper bound in \eqref{bench:mark:risk}, resulting in the same contradiction. 

Now that we have derived the accurate order information for $\lambdas$: $\lambdas^2=\nu_n^2-(6+o(1))\log\nu_n$, we can plug it into \eqref{final:ref:to} to obtain the sharp lower bound:
\[
F(\lambdas)\geq \epsilon_n\Big(\nu_n^2-(6+o(1))\log\nu_n\Big).
\]

\subsection{Proof of Proposition \ref{prop:ridge:thirdregime}}\label{sec:proof:thridregime:ridge}
Using the simple form of $\hat{\eta}_L(y,\lambda)$, the calculation is straightforward:
\begin{eqnarray*}
    && \inf_{\lambda} \sup_{\theta \in \Theta(k_n , \tau_n)} \mathbb{E}_{\theta}\|\hat{\eta}_L(y,\lambda)-\theta\|_2^2 \\
    &=& \inf_{\lambda} \sup_{\theta \in \Theta(k_n , \tau_n)} \E_{\theta} \sum_{i=1}^{n} \left(\frac{1}{1+\lambda} y_{i} - \theta_{i}\right)^{2} \\
    &=&  \inf_{\lambda} \sup_{\theta \in \Theta(k_n , \tau_n)}  \sum_{i=1}^{n} \left[ \left(\frac{\lambda}{1+\lambda}\right)^{2} \theta_{i}^{2} + \left(\frac{1}{1+\lambda}\right)^{2}\sigma_n^2 \right] \\
    &=&  \inf_{\lambda}\frac{\lambda^2k_n\tau_n^2+n\sigma_n^2}{(1+\lambda)^2} = \frac{n\sigma_n^2\epsilon_n\mu_n^2}{1+\epsilon_n\mu_n^2}. 
       \end{eqnarray*}

%%%%%%%%%%%%%%%%%%%%%%% the part to be edited

%\section*{Acknowledgement}
%This work is supported by NSF-DMS 2210506, and NSF-DMS 2210505.

% \section{References}
% You can use a bibliography generated by BibTeX as a .bbl file.
%  BibTeX documentation can be easily obtained at:
%  http://mirror.ctan.org/biblio/bibtex/contrib/doc/
%  The IEEEtran BibTeX style support page is:
%  http://www.michaelshell.org/tex/ieeetran/bibtex/

\bibliographystyle{IEEEtran}
\bibliography{reference.bib}
 
\begin{IEEEbiographynophoto}{Yilin Guo}
holds a Ph.D. in Statistics from Columbia University in 2023. Before joining Columbia, she received a B.S. in Statistics from University of Science and Techonology of China in 2018. Her research interests include high-dimensional statistics and statistical machine learning.
\end{IEEEbiographynophoto}

\begin{IEEEbiographynophoto}{Haolei Weng}
 is currently an Assistant Professor at the Department of Statistics and Probability, Michigan State University. Prior to MSU, he completed his Ph.D. in Statistics from Columbia University in 2017 and was a postdoctoral researcher at Princeton University in 2018. Before going to Columbia, he received a B.S. in Statistics from University of Science and Technology of China. His research interests are broadly in the area of high-dimensional statistics and statistical machine learning.
\end{IEEEbiographynophoto}

%\newpage

\begin{IEEEbiographynophoto}{Arian Maleki}
 is an associate professor in the Department of Statistics at Columbia University. He received his PhD from Stanford University in 2011. Before joining Columbia University, he was a postdoctoral scholar at Rice University. Arian’s research interests include high-dimensional statistics, computational imaging, compressed sensing, and machine learning.

\end{IEEEbiographynophoto}

% \newpage

\end{document}